\documentclass[11pt,a4paper]{article}
\usepackage[T1]{fontenc}
\usepackage[utf8]{inputenc}
\usepackage{lmodern}
\usepackage[margin=28mm]{geometry}
\usepackage{amsmath,amssymb,amsthm,mathtools}
\allowdisplaybreaks[4]
\usepackage{microtype}
\usepackage[colorlinks,linkcolor=blue,anchorcolor=blue,citecolor=blue]{hyperref}

\numberwithin{equation}{section}
\newtheorem{theorem}{Theorem}[section]
\newtheorem{proposition}[theorem]{Proposition}
\newtheorem{lemma}[theorem]{Lemma}
\newtheorem{corollary}[theorem]{Corollary}
\newtheorem*{historicaltheorem}{Theorem}
\newtheorem*{conjecture}{Conjecture}
\theoremstyle{definition}
\newtheorem{definition}{Definition}[section]

\DeclareMathOperator{\Ric}{Ric}
\DeclareMathOperator{\tr}{tr}
\DeclareMathOperator{\PSH}{PSH}

\newcommand{\R}{\mathbb R}
\newcommand{\C}{\mathbb C}

\title{Pseudo-Calabi Flow on Toric Variety}
\author{Junhao Tian \and Kai Zheng}
\date{}
\begin{document}

\maketitle

\begin{abstract}
In this paper, we prove the long time existence of the pseudo-Calabi
flow on a compact toric variety for any $T^n$-invariant initial
K\"ahler metric. If the Mabuchi functional is proper modulo
automorphisms, the flow converges exponentially to a constant scalar
curvature K\"ahler (cscK) metric. In particular, this gives a flow
proof that properness of the K-energy implies the existence of a
cscK metric in the toric setting.
\end{abstract}

\tableofcontents

\section{Introduction}\label{sec:introduction}
The problem of finding a constant scalar curvature K\"ahler metric
in a fixed K\"ahler class is part of Calabi's theory of extremal
metrics \cite{Calabi1982,Calabi1985}. Let $(M,\omega_0)$ be a connected
compact K\"ahler manifold of complex dimension $n$, and put the volume
$V:=\int_M\omega_0^n$. Throughout the paper, we use
$dd^c:=i\partial\bar\partial$.

The metrics in $[\omega_0]$ are represented by the K\"ahler potentials
\begin{equation}\label{eq:calc-s1-001}
 \mathcal H:=\{\varphi\in C^\infty(M,\R):
              \omega_\varphi:=\omega_0+dd^c\varphi>0\}.
\end{equation}
Potentials differing by a constant represent the same metric.
Writing $g_{\varphi,i\bar j}$ for the local coefficients of
$\omega_\varphi$ and $g_\varphi^{i\bar j}$ for their inverse, we set
$\Delta_\varphi:=g_\varphi^{i\bar j}\partial_i\partial_{\bar j}$ and
$R(\omega_\varphi):=\tr_{\omega_\varphi}\Ric(\omega_\varphi)$.
A metric $\omega_\varphi$ is a constant scalar curvature K\"ahler
metric, abbreviated cscK, if
\begin{equation}\label{eq:calc-s1-002}
 R(\omega_\varphi)=\underline R,\qquad
 \underline R:=\frac1V\int_M R(\omega_\varphi)\,\omega_\varphi^n.
\end{equation}
Here $\underline R$ depends only on the K\"ahler class and the
complex structure.

Mabuchi's $K$-energy $\mathcal M:\mathcal H\to\R$
\cite{Mabuchi1986} is normalized by $\mathcal M(0)=0$ and has
first variation
\begin{equation}\label{eq:calc-s1-003}
 \frac d{ds}\mathcal M(\varphi_s)
 =\frac1V\int_M\dot\varphi_s
           \bigl(\underline R-R(\omega_{\varphi_s})\bigr)
                           \,\omega_{\varphi_s}^n
\end{equation}
for every smooth path $\varphi_s$ in $\mathcal H$. Its critical
points are precisely the cscK metrics.

Calabi proposed the flow
\begin{equation}\label{eq:calc-s1-004}
 \partial_t\varphi
 =-\operatorname{grad}_{L^2}\mathcal M(\varphi)
 =R(\omega_\varphi)-\underline R
\end{equation}
as a method for finding canonical metrics
\cite{Calabi1982,Calabi1985}. Here the gradient is taken with
respect to the Mabuchi $L^2$ metric on $\mathcal H$, as described in
\cite[Section~2]{ChengTian2026}:
\begin{equation}\label{eq:calc-s1-005}
 (a,b)_{L^2,\varphi}
 :=\frac1V\int_Mab\,\omega_\varphi^n,\qquad
 \operatorname{grad}_{L^2}\mathcal M(\varphi)
 =\underline R-R(\omega_\varphi).
\end{equation}
The tangent vectors $a,b$ are smooth real functions on $M$.

Short time existence of the Calabi flow holds for every smooth
initial K\"ahler metric; see Chen--He \cite{ChenHe2008}.
For its long time behavior, Li--Wang--Zheng
\cite{LiWangZheng2018} developed estimates for the regularity
scale and proved convergence of the modified Calabi flow on
extremal K\"ahler surfaces, provided that the initial metric can
be joined to an extremal metric by a path of initial metrics
whose Calabi flows all exist for a long time. Their higher
dimensional result also requires a common scalar curvature
bound along these flows. More recently, they proved that a
Calabi flow on a compact K\"ahler surface can be continued
through a finite time whenever the Sobolev constants remain
uniformly bounded \cite{arXiv:2610.05217}.

In view of the difficulties in studying the Calabi flow as a
fourth order nonlinear parabolic equation, we turn to the
pseudo-Calabi flow, which provides another natural and promising
approach to the construction of cscK metrics.

We use the complex gradient convention of \cite{ChengTian2026}:
for a real function $a$ and a K\"ahler metric $\omega$ with local
coefficients $g_{i\bar j}$,
\begin{equation}\label{eq:calc-s1-006}
 |\nabla_\omega a|_\omega^2:=g^{i\bar j}a_i a_{\bar j}.
\end{equation}
The real inner product
$\langle\nabla_\omega a,\nabla_\omega b\rangle_\omega$ is obtained
by polarization. For $\omega=\omega_\varphi$, we write
$\nabla_\varphi$ and $|\cdot|_\varphi$.
Another natural choice is then the Dirichlet metric (see \cite{MR3412344}) on potentials
modulo constants:
\begin{equation}\label{eq:calc-s1-007}
 (a,b)_{D,\varphi}
 :=\frac1V\int_M
       \langle\nabla_\varphi a,\nabla_\varphi b\rangle_\varphi
                                      \,\omega_\varphi^n.
\end{equation}
The corresponding negative gradient flow of $\mathcal M$ is
the pseudo-Calabi flow introduced by Chen--Zheng
\cite{ChenZheng2013}.

\begin{definition}\label{def:pcf}
A smooth path of K\"ahler metrics $\omega_{\varphi(t)}$ is a
pseudo-Calabi flow if, modulo functions of time, its potentials
satisfy
\begin{equation}\label{eq:calc-s1-008}
 \partial_t\varphi
 =\Delta_\varphi^{-1}\bigl(\underline R-R(\omega_\varphi)\bigr).
\end{equation}
Here $\Delta_\varphi^{-1}$ denotes the inverse of $\Delta_\varphi$
on functions of zero $\omega_\varphi^n$-mean.
\end{definition}
The variational formulation in
\cite[Theorem~3.1 and Remark~3.2]{ChenZheng2013}
gives the gradient with respect to the Dirichlet metric:
\begin{equation}\label{eq:calc-s1-009}
 \operatorname{grad}_D\mathcal M(\varphi)
 =\Delta_\varphi^{-1}\bigl(R(\omega_\varphi)-\underline R\bigr).
\end{equation}
Thus the stationary metrics are precisely the cscK metrics.

To exhibit its analytic structure, write
$F:=\log(\omega_\varphi^n/\omega_0^n)$. The identity
\begin{equation}\label{eq:calc-s1-010}
 \Delta_\varphi F
 =\tr_{\omega_\varphi}\Ric(\omega_0)-R(\omega_\varphi)
\end{equation}
gives the following equivalent system, with the sign and zero-mean
normalization of $P$ used in \cite[(1.4)]{ChengTian2026}:
\begin{equation}\label{eq:pcf}
 \partial_t\varphi=F+P,\qquad
 \Delta_\varphi P=\underline R-
               \tr_{\omega_\varphi}\Ric(\omega_0),\qquad
 \int_M P\,\omega_\varphi^n=0.
\end{equation}
The metric at time zero is taken as the reference metric, so
$\varphi(0)=0$. We use this normalization unless otherwise stated.
The system is formally of second order, but $P$ depends on the
metric through an elliptic equation on all of $M$.

Chen--Zheng \cite[Theorems~4.1 and~4.43]{ChenZheng2013}
proved short time existence and uniqueness for smooth initial
K\"ahler metrics, together with instantaneous smoothing for
initial potentials in the little H\"older space $C^{2,\alpha}$,
$0<\alpha<1$, with positive associated metric. Here little
H\"older means the closure of smooth functions in the indicated
norm. They also proved exponential convergence when the initial
potential is sufficiently close in $C^{2,\alpha}$ to a cscK
potential \cite[Theorem~1.3]{ChenZheng2013}.

The ultimate expectation combines the long time existence
conjecture and the convergence question attributed to Chen in
\cite[Conjecture~8.3 and Question~8.5]{ChenZheng2013}.
We state them together as follows.
\begin{conjecture}
On a compact K\"ahler manifold, the pseudo-Calabi flow from every
smooth initial K\"ahler metric exists for all $t\geq0$.
If its K\"ahler class contains a cscK metric, the flow converges
smoothly to a cscK metric in that class as $t\to\infty$.
\end{conjecture}
Once smooth convergence is known, the local stability theorem
gives exponential convergence.

Cheng--Tian \cite[Proposition~1.1]{ChengTian2026} obtained a short
time regularization result for initial potentials plurisubharmonic
with respect to $\omega_0$, with sufficiently small logarithmic
volume ratio in supremum norm. The solution exists on $[0,T_0]$,
where $T_0$ depends only on the smooth background, and is smooth
for $t>0$, with all derivatives bounded away from time zero.
This leads to stability under a $C^0$ perturbation of the volume
form, improving the earlier $C^{2,\alpha}$ condition on the
potential. We recall the full stability statement from
\cite[Theorem~1.2]{ChengTian2026}.

\begin{historicaltheorem}[Cheng--Tian]
Let $(M,\omega_*)$ be a compact cscK manifold. There exists
$\varepsilon_*>0$ such that, for every initial K\"ahler metric
$\omega\in[\omega_*]$ satisfying
\begin{equation}\label{eq:calc-s1-011}
 1-\varepsilon_*\leq\frac{\omega^n}{\omega_*^n}
                          \leq1+\varepsilon_*,
\end{equation}
the pseudo-Calabi flow $\omega(t)$ with $\omega(0)=\omega$ exists
for all $t\geq0$ and converges exponentially to a cscK metric
$\omega_\infty\in[\omega_*]$, which need not equal $\omega_*$.
More precisely, for each $k\geq0$ there are $C_k,\lambda_k>0$ such
that
\begin{equation}\label{eq:calc-s1-012}
 \|\omega(t)-\omega_\infty\|_{C^k(M,\omega_*)}
                  \leq C_ke^{-\lambda_k t},\qquad t\geq1.
\end{equation}
\end{historicaltheorem}

The estimates of Cheng--Tian
\cite[Definition~4.3 and Corollary~4.7]{ChengTian2026}
also show that a bound
for $F$, together with uniformly continuity control of $P$,
gives higher order estimates. A central issue is therefore to
control $P$ with respect to a fixed background metric. A volume
bound alone does not give uniform ellipticity of $\Delta_\varphi$,
so the usual elliptic H\"older estimates do not apply directly.

In the toric setting, Guillemin \cite{Guillemin1994} and Abreu
\cite{Abreu1998,Abreu2003} described $T^n$-invariant K\"ahler metrics by
convex functions on a fixed moment polytope. Donaldson
\cite{Donaldson2002,Donaldson2005,Donaldson2009} developed this
description to study scalar curvature and stability, and proved
existence of cscK metrics on K-stable polarized toric surfaces
with zero Futaki invariant. Zhou--Zhu \cite{ZhouZhu2008} studied
relative stability and properness of the modified Mabuchi functional
on toric manifolds. Chen--Li--Sheng \cite{ChenLiSheng2014} proved
that uniform stability is necessary for the existence of toric
extremal metrics. Chen--Cheng
\cite[Theorem~1.10]{ChenCheng2021} subsequently characterized
toric cscK existence in arbitrary dimension by uniform stability.

For the $T^2$-invariant Calabi flow on toric Fano surfaces, Chen--He
\cite{ChenHe2010} proved long time existence and subsequential
convergence in the Cheeger--Gromov sense to an extremal metric,
assuming positive extremal Hamiltonian potential and an initial
Calabi energy below an explicit threshold determined by the
K\"ahler class. Sz\'ekelyhidi \cite{Szekelyhidi2008} proved that,
if a toric Calabi flow exists for all time, its Calabi energy
approaches the infimum, and related this infimum to optimal
destabilization. Huang \cite{Huang2016} obtained
$C^0$ estimates under a uniform Sobolev bound, and exponential
convergence of the modified Calabi flow under uniform $K$-stability
and uniform bounds for the Ricci curvature and Sobolev constant.
The analysis of these flows requires estimates near the boundary
of the moment polytope.

In this paper, we study the pseudo-Calabi flow on compact
toric varieties from $T^n$-invariant initial
metrics. On each finite time interval, we bound the volume ratio
and $P$, and prove that $P$ has a uniform modulus of
continuity, then the higher order
estimates of Cheng--Tian then allow us to continue the flow.
We obtain the first main theorem.

\begin{theorem}\label{thm:main}
Let $M$ be a compact toric variety, and
let $\omega_0$ be a K\"ahler metric invariant under the compact
torus $T^n$. The pseudo-Calabi flow \eqref{eq:pcf} with initial metric
$\omega_0$ has a unique smooth solution for all $t\geq0$.
\end{theorem}
This proves the long time existence conjecture of
\cite{ChenZheng2013} for $T^n$-invariant metrics in every dimension and
every K\"ahler class of a smooth compact toric variety. 

To state the convergence theorem, we introduce the $E$-functional:
\begin{equation}\label{eq:energy-normalization}
 E(\varphi):=\frac1{(n+1)V}\sum_{j=0}^n
       \int_M\varphi\,\omega_\varphi^j\wedge\omega_0^{n-j}.
\end{equation}
The space of normalized potentials is
$\mathcal H_0:=\{\varphi\in\mathcal H:E(\varphi)=0\}$.
Since $E(\varphi+c)=E(\varphi)+c$,
each metric has a unique representative in $\mathcal H_0$.

Let $G:=\operatorname{Aut}_0(M)$. Its elements act trivially on
cohomology and therefore preserve $[\omega_0]$.
Let $d_1$ be the path length distance on $\mathcal H$ induced by
the norm $V^{-1}\int_M|a|\,\omega_\varphi^n$ on
$a\in T_\varphi\mathcal H$.
\begin{definition}
\label{def:properness}
For $g\in G$, let $g\cdot\varphi\in\mathcal H_0$ represent
$g^*\omega_\varphi$. For a subgroup $A\subset G$, set
\begin{equation}\label{eq:calc-s1-013}
 d_{1,A}(0,\varphi):=\inf_{g\in A}d_1(0,g\cdot\varphi).
\end{equation}
Following \cite[Definition~4.8]{ChenCheng2021},
the Mabuchi functional is proper on $\mathcal S\subset\mathcal H_0$
modulo $A$ if it is bounded below on $\mathcal S$ and, for every
sequence $\varphi_j\in\mathcal S$,
\begin{equation}\label{eq:properness-assumption}
 d_{1,A}(0,\varphi_j)\longrightarrow\infty
 \quad\Longrightarrow\quad
 \mathcal M(\varphi_j)\longrightarrow+\infty.
\end{equation}
\end{definition}
Thus properness bounds $d_{1,A}(0,\varphi)$ on every sublevel
of $\mathcal M$.

The relation between properness and canonical metrics has its origins
in work on positive K\"ahler--Einstein metrics by Ding
\cite{Ding1988}, Ding--Tian \cite{DingTian1992}, and Tian
\cite{Tian1997}. Its formulation modulo
automorphisms was developed by Darvas--Rubinstein
\cite{DarvasRubinstein2017}. Berman--Darvas--Lu
\cite{BermanDarvasLu2020} proved that cscK existence implies
properness, and Chen--Cheng \cite{ChenCheng2021I,ChenCheng2021}
proved the converse. We recall their main result here
\cite[Theorem~1.6]{ChenCheng2021}.

\begin{historicaltheorem}[Chen--Cheng]
Let $(M,\omega_0)$ be a compact K\"ahler manifold, and let
$G=\operatorname{Aut}_0(M)$. The class $[\omega_0]$ contains a
cscK metric if and only if the Mabuchi functional is proper on
the full space $\mathcal H_0$ modulo $G$ in the sense of
Definition~\ref{def:properness}.
\end{historicaltheorem}

There is a folklore belief that Chen-Cheng's theorem ought to admit a proof using geometric flow, in which a cscK metric is
obtained as the limit of a global solution from arbitrary initial
data. 

We give such a proof in the toric setting. Let
$T_{\C}\simeq(\C^*)^n$ be the complexification of $T^n$, and denote by
$\mathcal H_0^T$ the $T^n$-invariant elements of $\mathcal H_0$.
The group $G$ can be larger than $T_{\C}$. Our hypothesis is
Definition~\ref{def:properness} with
$(\mathcal S,A)=(\mathcal H_0^T,G)$.

\begin{theorem}\label{thm:proper-convergence}
Let $M$ be a smooth compact toric variety, and fix a K\"ahler class.
Suppose that the Mabuchi functional is proper modulo $G$ on
$\mathcal H_0^T$ in the sense of Definition~\ref{def:properness},
or, equivalently, that the associated moment polytope is
$L^1$ stable in the sense of
\cite[Definition~7.7]{ChenCheng2021}. Then the
pseudo-Calabi flow from any smooth $T^n$-invariant initial metric
converges in $C^\infty$, exponentially in time, to a $T^n$-invariant
cscK metric $\omega_\infty$ (We do NOT assume the existence of cscK as a condition). The convergence holds for the metric
flow itself and for the potentials normalized by vanishing
$E$-functional.  
\end{theorem}

The implication from $L^1$ stability to properness follows from
\cite[Proposition~7.13 and Lemma~7.14]{ChenCheng2021},
which give coercivity modulo $T_{\C}$, and from
$d_{1,G}(0,\varphi)\leq d_{1,T_{\C}}(0,\varphi)$.
Conversely, the proof below obtains a $T^n$-invariant cscK metric
under the properness assumption. Its existence implies $L^1$
stability by the necessity direction of
\cite[Theorem~7.2]{ChenCheng2021}.

We explain the main estimates. On finite time intervals, the
exponential decay of the background density and a convex weight
in logarithmic coordinates give an $L^p$ bound for the volume
ratio, for some $p>1$. This bounds the normalized potential.
The uniform Sobolev follows, and then yield
spatial moment estimates for $P$. An auxiliary parabolic Monge--Amp\`ere
equation gives the inverse-density integral bound needed for the
pointwise estimate of $P$ by Cheng--Tian. A maximum principle
argument with a time factor then bounds the volume ratio after
a positive waiting time. Together with short-time existence, this
gives the finite-time $C^0$ estimates. To control the continuity of $P$, we
subtract a local potential and estimate the oscillation of the
remaining harmonic function on toric sublevel domains. The Guillemin
boundary formula and convexity make the argument valid at every
boundary stratum.

For convergence, we first reduce properness modulo $G$ to
properness modulo $T_{\C}$. The Mabuchi bound then controls the
$d_1$ distance and the entropy after a continuous torus normalization, whose
speed is controlled by the Dirichlet dissipation. A global
contact-set argument gives a negative exponential estimate for
$P$. Differentiating a weighted density integral along the flow
and using Jensen's inequality yields uniform $L^{1+\varepsilon}$
integrability of the volume ratio. We obtain smooth compactness
on fixed time intervals and a cscK limit at times with vanishing
dissipation. Only then do we apply local stability, with one
fixed automorphism, to obtain exponential convergence of the
original flow.

Section~\ref{sec:preliminaries} collects the analytic identities
and toric preliminaries, proves the reduction of properness, and
derives the polytope form of the flow.
Section~\ref{sec:volume} first establishes an $L^p$ bound for the
volume ratio and then proves a positive-time $C^0$ estimate from
this bound, independently of the initial pointwise density.
The finite-time volume and potential bounds follow as a corollary. The two parts of Section~\ref{sec:P-estimates}
prove, respectively, the modulus of continuity of $P$ needed for
Theorem~\ref{thm:main} and the negative exponential estimate used
for convergence. Section~\ref{sec:proper-convergence} completes
the proof of Theorem~\ref{thm:proper-convergence}.

\textbf{Acknowledgment:} The first author would like to thank Prof.~Jingrui Cheng for
inspiring him to study this problem and Prof.~Xiuxiong Chen for
enlightening discussions. The second author is partially funded by
NSFC grant No.~12571091.

\textbf{AI Disclosure:} AI tools were used for language editing and
typesetting, to generate some elementary results, including
Lemma~\ref{lem:real-variable-iteration}. ChatGPT 6.1 Sol suggests the proof
strategy for Proposition~\ref{prop:relative-green}. The authors take
full responsibility for all content of this paper.

\section{Preliminaries}\label{sec:preliminaries}
We use the notation and functionals introduced in
Section~\ref{sec:introduction}. We first recall some basic properties needed in the later sections, then turn to toric geometry and the flow
in moment coordinates. The moment polytope is denoted by $\Delta$.
\subsection{Basic properties}
With the complex gradient convention of
Section~\ref{sec:introduction}, we have
\begin{equation}\label{eq:calc-s2-001}
 \begin{gathered}
 \Delta_\omega(ab)=a\Delta_\omega b+b\Delta_\omega a
                +2\langle\nabla_\omega a,\nabla_\omega b\rangle_\omega,\quad
 \Delta_\omega a\,\omega^n=n\,dd^ca\wedge\omega^{n-1}.
 \end{gathered}
\end{equation}
The distance associated with the fixed background metric is $d_0$.
We use the normalized volume measures
\begin{equation}\label{eq:normalized-volume-measures}
 d\mu_\varphi:=V^{-1}\omega_\varphi^n,\qquad
 d\mu_0:=V^{-1}\omega_0^n.
\end{equation}
With $F$ as in Section~\ref{sec:introduction}, the relative entropy
and Dirichlet energy are
\begin{equation}\label{eq:entropy-dirichlet-definitions}
 \begin{aligned}
 \int_M\log\frac{d\mu_\varphi}{d\mu_0}\,d\mu_\varphi
 =\int_M F\,d\mu_\varphi,\quad
 \|\nabla_\varphi a\|_{L^2(d\mu_\varphi)}^2
 :=\int_M|\nabla_\varphi a|_\varphi^2\,d\mu_\varphi.
 \end{aligned}
\end{equation}

The first variation of the $E$-functional in
\eqref{eq:energy-normalization} is
\begin{equation}\label{eq:calc-s2-002}
 \frac d{ds}E(\varphi_s)=\int_M\dot\varphi_s\,d\mu_{\varphi_s}.
\end{equation}
For the estimate and the convergence argument, we add a function of time to the
potential so that $E(\varphi(t))=0$. The preceding variation and
\eqref{eq:pcf} give its normalized speed
\begin{equation}\label{eq:normalized-pcf-speed}
 f:=\partial_t\varphi=F+P-\int_M F\,d\mu_\varphi,\qquad
 \int_M f\,d\mu_\varphi=0.
\end{equation}
The function $P$ retains the zero-mean normalization in
\eqref{eq:pcf}. 

Using Definition~\ref{def:pcf} in the first variation
of $\mathcal M$ and integrating by parts gives the dissipation
identity of Chen--Zheng
\cite[Theorem~3.1 and Remark~3.2]{ChenZheng2013}:
\begin{equation}\label{eq:Mabuchi-dissipation}
 \frac d{dt}\mathcal M(\varphi(t))=-\|\nabla_\varphi f\|_{L^2(d\mu_\varphi)}^2.
\end{equation}
Both sides are unchanged by adding a function of time to the
potential. 

We use \eqref{eq:pcf} for the finite time argument in
Section~\ref{sec:volume} and in Subsection~\ref{sec:P-modulus}.
We use \eqref{eq:normalized-pcf-speed} in
Proposition~\ref{prop:positive-time-volume} and
Section~\ref{sec:proper-convergence}.
The two normalizations differ only by a function of time added
to the potential, so they do not change the metric flow or our
geometric conclusions. The first simplifies the finite time
computations, while the second removes the drift of the additive
constant and makes the long time estimates more concise. 

%%%%%%%%%%%%%%%%%%%%%%%%%%%%%%%%%%%%%%%%%%%%%%%%%%%%%%%%%%%%%%%%%%%%%%%%%%%%%%%%%%%%%%%%%%%%%%%%%%%%%%%%%%%%%%%%%%%%%%%%%%%%%%%%%%%%%%%%%%%%%%%%%%%%%%%%%%%%
\subsection{Compact toric variety}
We now assume that $M$ is a smooth compact toric variety: the
algebraic torus $T_{\C}\simeq(\C^*)^n$ is a dense open orbit and
its multiplication action extends to $M$. The fixed metric
$\omega_0$ is invariant under $T^n\simeq(S^1)^n$.
A function $a$ is $T^n$-invariant if $a\circ g=a$ for every
$g\in T^n$; a metric $\omega$ is $T^n$-invariant if
$g^*\omega=\omega$ for every $g\in T^n$.

We choose dual integral bases for the character and one-parameter subgroup
lattices, identifying each with $\mathbb Z^n$. Moment coordinates
and facet normals are paired using these bases; see
\cite[Chapters~2--3]{CoxLittleSchenck2011}.

\begin{definition}\label{def:delzant-polytope}
A full-dimensional compact convex polytope $\Delta\subset\R^n$
is Delzant if it has an irredundant presentation
\begin{equation}\label{eq:calc-s2-003}
 \Delta=\{x\in\R^n:\ell_r(x):=\langle x,\nu_r\rangle-\lambda_r\geq0,
                       \quad 1\leq r\leq d\},
\end{equation}
where $\lambda_r\in\R$, the inward normals $\nu_r\in\mathbb Z^n$
are primitive, and exactly $n$ facets meet at every vertex, with
their normals forming a basis of $\mathbb Z^n$. 

A normal is
primitive if the greatest common divisor of its components is one. We denote the facets by
$\mathcal F_r:=\{x\in\Delta:\ell_r(x)=0\}$.
\end{definition}
The compact torus action is Hamiltonian for a $T^n$-invariant K\"ahler
form. Delzant's theorem \cite{Delzant1988} identifies its moment
image with a Delzant polytope, which determines the Hamiltonian
torus action up to equivariant symplectomorphism. We fix the additive
constant of the moment map so that $\Delta$ remains fixed as the
metric varies in the fixed K\"ahler class.

We recall the orbit and chart description from
\cite[Example~1.2.21 and Sections~3.1--3.2]{CoxLittleSchenck2011}.
\begin{itemize}
\item
A face of $\Delta$ of codimension $k$ corresponds to a complex
torus orbit of dimension $n-k$, with compact-torus stabilizer of
dimension $k$. 
\item
Facets correspond to toric divisors and vertices
to fixed points. 
\item
Each vertex $v$ determines a holomorphic chart
$U_v\simeq\C^n$. It consists of the complex torus orbits
corresponding to faces containing $v$. 
\item These charts cover $M$,
and their coordinates on the open orbit are monomials in any fixed
character basis.
\end{itemize}

We use the Legendre description of $T^n$-invariant metrics in
\cite[Proposition~3.1.7]{Donaldson2002}, with the following normalization.
On the open orbit $(\C^*)^n$, let $z_1,\ldots,z_n$ be the character
coordinates associated with the chosen character basis, and put
\begin{equation}\label{eq:calc-s2-004}
 \xi_i:=\log|z_i|^2,\qquad \theta_i:=\arg z_i.
\end{equation}
The angular variables have period $2\pi$. A $T^n$-invariant metric in the
fixed class is represented by a strictly convex function of $\xi$:
\begin{equation}\label{eq:calc-s2-005}
 \omega_\varphi=dd^c\Phi,\qquad
 \Phi(\xi):=\Phi_0(\xi)+\varphi(\xi),\qquad D^2_\xi\Phi>0.
\end{equation}
We fix the affine normalization of $\Phi_0$ once and for all.

\begin{definition}
\label{def:moment-legendre}
For the strictly convex potential $\Phi$, the moment coordinates
and the Legendre transform are
\begin{equation}\label{eq:legendre}
 x_\varphi:=D_\xi\Phi:\R^n\longrightarrow\Delta^\circ,
 \qquad
 u(x):=\sup_{\xi\in\R^n}
       \{\langle x,\xi\rangle-\Phi(\xi)\}.
\end{equation}
The moment map convention is
$\iota_{\partial/\partial\theta_i}\omega_\varphi=-dx_{\varphi,i}$.
\end{definition}
In particular,
\begin{equation}\label{eq:calc-s2-006}
 D_xu=\xi,\qquad D_x^2u=(D_\xi^2\Phi)^{-1},\qquad
 \Phi(\xi)+u(x)=\langle x,\xi\rangle
 \quad\text{when }x=D_\xi\Phi(\xi).
\end{equation}
The map $x_\varphi$ extends to the moment map on $M$. Thus $x$ denotes a
coordinate on the fixed polytope, while $x_\varphi$ denotes the moment map
of $\omega_\varphi$. In these coordinates,
\begin{equation}\label{eq:toric-volume}
 \omega_\varphi=\sum_i dx_{\varphi,i}\wedge d\theta_i,\qquad
 \frac{\omega_\varphi^n}{n!}=dx_\varphi\,d\theta,\qquad
 m:=\det D_\xi^2\Phi,\qquad e^F=\frac{m}{m_0}.
\end{equation}
Here $m_0:=\det D_\xi^2\Phi_0$ is the fixed reference density.

For a $T^n$-invariant function $a$ on $M$, let
$\widehat a(x):=a(D_xu(x))$ denote its expression in moment
coordinates; here $a$ is viewed as a function of $\xi$ on the open
orbit. Thus
\begin{equation}\label{eq:toric-integration}
 \int_M a\,\omega_\varphi^n
 =n!(2\pi)^n\int_\Delta\widehat a\,dx,
 \qquad V=n!(2\pi)^n|\Delta|.
\end{equation}

Write $u_0:=\Phi_0^*$ for the reference symplectic potential.
We shall also use the elementary identity
\begin{equation}\label{eq:legendre-C0}
 \|u-u_0\|_{L^\infty(\Delta)}
 =\|\varphi\|_{L^\infty(M)}.
\end{equation}
Indeed, a bound $|\Phi-\Phi_0|\leq C$ gives the same bound for
their Legendre transforms by taking suprema in \eqref{eq:legendre}.
The reverse inequality follows by applying the same argument to
$\Phi=u^*$ and $\Phi_0=u_0^*$.

\subsection{Boundary conditions}
Guillemin \cite{Guillemin1994} gave an explicit potential determined
by the facet functions of $\Delta$; see also
\cite[Theorem~2.5]{Abreu2003}. With the normalization used here,
it is
\begin{equation}\label{eq:guillemin-potential}
 u_G(x):=\sum_{r=1}^d\ell_r(x)\log\ell_r(x).
\end{equation}
\begin{definition}\label{def:symplectic-potential}
A symplectic potential on $\Delta$ is a function
$u\in C^\infty(\Delta^\circ)$ of the form
\begin{equation}\label{eq:calc-s2-007}
 u=u_G+h,\qquad h\in C^\infty(\Delta).
\end{equation}
The condition $h\in C^\infty(\Delta)$ means that $h$ extends smoothly
to a neighborhood of the closed polytope. 

We require $D^2u>0$ on $\Delta^\circ$ and
$(\det D^2u)\prod_r\ell_r$ to extend to a smooth positive function
on $\Delta$. These conditions are the Guillemin boundary conditions
in the conventions of this paper.
\end{definition}
By Abreu's boundary characterization \cite[Theorem~2.8]{Abreu2003},
Legendre duality identifies smooth $T^n$-invariant K\"ahler metrics in
the fixed class with these symplectic potentials, modulo constants.
In particular, $u$ extends continuously
to $\Delta$. The logarithmic terms describe the collapse
of the circle orbits at the toric divisors. The coefficient in
\eqref{eq:guillemin-potential} corresponds to
$\xi_i=\log|z_i|^2$ and $dd^c=i\partial\bar\partial$:
Abreu's logarithmic radial variable is $\xi/2$ and his symplectic
potential is $u/2$. This accounts for the coefficient $1$, rather
than $1/2$, in \eqref{eq:guillemin-potential}.

Write $h_0:=u_0-u_G\in C^\infty(\Delta)$ and
$x_0:=D_\xi\Phi_0$ for the background moment map. The Guillemin
formula gives
\begin{equation}\label{eq:background-log-coordinates}
 \xi=D u_0(x_0)
 =\sum_r\nu_r\log\ell_r(x_0)+\sum_r\nu_r+D h_0(x_0).
\end{equation}
This identity describes the logarithmic behavior near every face
without changing coordinates.

The following background estimates
follow from Guillemin's boundary description \cite{Guillemin1994},
in the form of \cite[Theorem~2.8]{Abreu2003}, and the smooth torus
action. The logarithmic moment bound then follows from
Jensen's inequality. We include the derivation in the
fixed coordinates used here.

\begin{proposition}
\label{prop:logarithmic-tail}
There are positive constants $b_0,\kappa_0$, depending only on $\Delta$
in the fixed lattice coordinates, and positive background constants
$c,C$ such that
\begin{equation}\label{eq:background-exponential-tail}
 \begin{gathered}
 c e^{-b_0|\xi|}\leq m_0(\xi)\leq C e^{-\kappa_0|\xi|},\\
 |\log m_0(\xi)|\leq C(1+|\xi|),\qquad
 |D_\xi\log m_0(\xi)|\leq C.
 \end{gathered}
\end{equation}
For every smooth $T^n$-invariant K\"ahler potential $\varphi$,
\begin{equation}\label{eq:entropy-log-coordinate-moment}
 \int_M|\xi|\,d\mu_\varphi
 \leq C\bigl(1+\int_M F\,d\mu_\varphi\bigr).
\end{equation}
Integrals involving $\xi$ are taken on the open orbit; its
complement has zero volume for every K\"ahler metric.
\end{proposition} 

\begin{proof}
Cover $\Delta$ by finitely many closed sets, each associated
with a vertex $v$, on which $\ell_r$ has a positive lower bound
whenever $v\notin\mathcal F_r$. To obtain such a cover, for each
point choose a vertex of the smallest face containing it.
Every facet containing the point also contains this vertex,
so the remaining $\ell_r$ stay positive on a sufficiently small
closed neighborhood. Compactness then gives a finite cover.
On the set associated with $v$, relabel the facets through $v$
as $\mathcal F_1,\ldots,\mathcal F_n$. For $x_0\in\Delta^\circ$
in this set, the logarithms $\log\ell_r(x_0)$ with $r>n$ are
bounded, and \eqref{eq:background-log-coordinates} gives
\begin{equation}\label{eq:calc-s2-008}
 \left|\xi-\sum_{r=1}^{n}\nu_r\log\ell_r(x_0)\right|
 \leq C_{\omega_0}.
\end{equation}
By Definition~\ref{def:delzant-polytope}, the normals
$\nu_1,\ldots,\nu_n$ form a basis. The matrix with these columns
is therefore invertible, and
\begin{equation}\label{eq:calc-s2-009}
 \begin{pmatrix}
 \log\ell_1(x_0)\\ \vdots\\ \log\ell_n(x_0)
 \end{pmatrix}
 =\begin{pmatrix}\nu_1&\cdots&\nu_n\end{pmatrix}^{-1}
   \left(\sum_{r=1}^{n}\nu_r\log\ell_r(x_0)\right).
\end{equation}
The inverse matrix depends only on $\Delta$ and the chosen
vertex, so
\begin{align}
 \sum_{r=1}^{n}|\log\ell_r(x_0)|
 &\leq C_\Delta
       \left|\sum_{r=1}^{n}\nu_r\log\ell_r(x_0)\right|\notag\\
 &\leq C_\Delta|\xi|+C_{\omega_0}.
 \label{eq:calc-s2-010}
\end{align}
Here $C_\Delta\geq1$ depends only on $\Delta$ in the fixed lattice
coordinates, while $C_{\omega_0}\geq1$ may also depend on the
fixed background and may increase from line to line.

For $r>n$, the functions $\ell_r$ are bounded above and away
from zero on this set, so
$\sum_{r=n+1}^{d}|\log\ell_r(x_0)|\leq C_\Delta$. Adding these
bounded terms gives the estimate for all facets:
\begin{align}
 \sum_{r=1}^{d}|\log\ell_r(x_0)|
 &=\sum_{r=1}^{n}|\log\ell_r(x_0)|
   +\sum_{r=n+1}^{d}|\log\ell_r(x_0)|\notag\\
 &\leq C_\Delta|\xi|+C_{\omega_0}.
 \label{eq:calc-s2-011}
\end{align}
The Guillemin boundary conditions give
\begin{equation}\label{eq:calc-s2-012}
 C_{\omega_0}^{-1}\prod_{r=1}^{d}\ell_r(x_0)
 \leq m_0(\xi)\leq C_{\omega_0}\prod_{r=1}^{d}\ell_r(x_0).
\end{equation}
Hence we have 
\begin{equation}\label{eq:calc-s2-013}
 -C_\Delta|\xi|-C_{\omega_0}
 \leq\log m_0(\xi)
 \leq-C_\Delta^{-1}|\xi|+C_{\omega_0}.
\end{equation}
This proves the first two estimates in
\eqref{eq:background-exponential-tail} on this set.
Taking uniform constants over the finite cover proves them
for all $\xi\in\R^n$, with $b_0=C_\Delta$ and
$\kappa_0=C_\Delta^{-1}$.

Let $X_i$ be the global smooth real vector field generated by
the complex torus action $z_i\mapsto e^{s/2}z_i$. On the open orbit,
$X_i=\partial_{\xi_i}$ and
$\omega_0^n=n!m_0\,d\xi\,d\theta$, so
\begin{equation}\label{eq:calc-s2-014}
 \partial_{\xi_i}\log m_0
 =\operatorname{div}_{\omega_0^n}X_i.
\end{equation}
The right-hand side is smooth on the compact manifold $M$,
which proves the derivative bound in
\eqref{eq:background-exponential-tail}.

Choose $0<s<\kappa_0$. By toric integration and
\eqref{eq:background-exponential-tail},
\begin{align}
 \int_Me^{s|\xi|}\,d\mu_0
 &=\frac1{|\Delta|}\int_{\R^n}e^{s|\xi|}m_0\,d\xi\notag\\
 &\leq\frac C{|\Delta|}\int_{\R^n}
            e^{-(\kappa_0-s)|\xi|}\,d\xi<\infty.
 \label{eq:calc-s2-015}
\end{align}
Since $d\mu_\varphi=e^F d\mu_0$ is a probability measure,
Jensen's inequality for the concave function $\log$ gives,
for every $N>0$, 
\begin{align}
 &s\int_M\min\{|\xi|,N\}\,d\mu_\varphi
       -\int_M F\,d\mu_\varphi=\int_M\log\!\left(e^{s\min\{|\xi|,N\}-F}\right)
                                             \,d\mu_\varphi\notag\\
 &\leq\log\int_M e^{s\min\{|\xi|,N\}-F}\,d\mu_\varphi =\log\int_M e^{s\min\{|\xi|,N\}}\,d\mu_0\notag\\
 &\leq\log\int_M e^{s|\xi|}\,d\mu_0.
 \label{eq:calc-s2-016}
\end{align}
The last expression is finite and independent of $N$.
Letting $N\to\infty$ and using monotone convergence, we obtain
\begin{equation}\label{eq:calc-s2-017}
 s\int_M|\xi|\,d\mu_\varphi
 \leq\int_M F\,d\mu_\varphi
       +\log\int_M e^{s|\xi|}\,d\mu_0.
\end{equation}
Dividing by $s$ and using $\int_M F\,d\mu_\varphi\geq0$, we obtain
\begin{align}
 \int_M|\xi|\,d\mu_\varphi
 &\leq\frac1s\int_M F\,d\mu_\varphi
       +\frac1s\log\int_M e^{s|\xi|}\,d\mu_0\notag\\
 &\leq C\int_M F\,d\mu_\varphi+C
 =C\left(1+\int_M F\,d\mu_\varphi\right).
 \label{eq:calc-s2-018}
\end{align}
This proves \eqref{eq:entropy-log-coordinate-moment}.
The constant involves only the fixed background and the chosen
$s$, and is independent of $\varphi$. For example, fixing
$s=\kappa_0/2$ makes it a background constant in the fixed
lattice coordinates.
\end{proof}

\subsection{Formulas on toric variety} 
Subscripts $i,j$ now denote derivatives with respect to $x$, and
repeated indices are summed. Write $(u^{ij}):=(D^2u)^{-1}$.
In the formulas below, $\omega:=\omega_\varphi$.
On $T^n$-invariant functions, the Laplacian has the divergence form
\begin{equation}\label{eq:polytope-laplacian}
 \Delta_u a:=\partial_i(u^{ij}\partial_j a).
\end{equation}
Indeed, put $H:=D_\xi^2\Phi$.
Then $\partial_{\xi_i}=H_{ki}\partial_{x_k}$ and
$H_{ij}(\xi)=u^{ij}(D_\xi\Phi(\xi))$. Therefore
\begin{align}
 \Delta_\omega(a\circ D_\xi\Phi)
 &=H^{ij}\partial_{\xi_i}\bigl(H_{kj}(a_k\circ D_\xi\Phi)\bigr)\notag\\
 &=\bigl[\partial_j(u^{kj}a_k)\bigr]\circ D_\xi\Phi
 =(\Delta_u a)\circ D_\xi\Phi.
 \label{eq:calc-s2-019}
\end{align}
Here $H^{ij}$ denotes the inverse matrix of $H$.
This is the complex Laplacian fixed above. In particular, no factor
of two is inserted into \eqref{eq:polytope-laplacian}.
The gradient inner product has the corresponding expression
\begin{equation}\label{eq:calc-s2-020}
 \langle\nabla_\omega(a\circ D_\xi\Phi),
             \nabla_\omega(b\circ D_\xi\Phi)\rangle_\omega
       =(u^{ij}a_i b_j)\circ D_\xi\Phi.
\end{equation}

Abreu
\cite[Theorem~4.1]{Abreu1998} expressed scalar curvature in terms of the symplectic
potential. In our conventions the formula is
\begin{equation}\label{eq:abreu-scalar}
 R(\omega)=-\sum_{i,j=1}^n
             \frac{\partial^2u^{ij}}{\partial x_i\partial x_j}.
\end{equation}
The scalar curvature here is $R(\omega)=\tr_\omega\Ric(\omega)$.
Thus the cscK equation is a fourth order equation for $u$, with
the Guillemin boundary conditions. 

Differentiation of the inverse Hessian gives
\begin{equation}\label{eq:calc-s2-021}
 \partial_i u^{ij}
 =-u^{ia}u^{bj}u_{abi}
 =-u^{bj}\partial_b\log\det D^2u.
\end{equation}
Consequently,
\begin{equation}\label{eq:polytope-logdet-curvature}
 \Delta_u\log\det D^2u
 =\partial_i\bigl(u^{ij}\partial_j\log\det D^2u\bigr)
 =-\partial_i\partial_j u^{ij} =R(\omega).
\end{equation}

For a facet with primitive inward normal $\nu_r$, let $d\sigma$
be Euclidean surface measure divided by $|\nu_r|$.
Donaldson's formula \cite[Proposition~3.2.8]{Donaldson2002}, with
our energy normalization, reads
\begin{equation}\label{eq:toric-mabuchi}
 \mathcal M(u)=\frac1{|\Delta|}
 \left[-\int_\Delta\log\frac{\det D^2u}{\det D^2u_0}\,dx
       +\int_{\partial\Delta}(u-u_0)\,d\sigma
       -\underline R\int_\Delta(u-u_0)\,dx\right].
\end{equation}

The toric integration-by-parts formula
\cite[Lemma~3.3.5]{Donaldson2002} gives, for $a\in C^\infty(\Delta)$,
\begin{equation}\label{eq:calc-s2-022}
 \int_\Delta u^{ij}a_{ij}\,dx
   =\int_{\partial\Delta}a\,d\sigma-\int_\Delta R(\omega_\varphi)a\,dx.
\end{equation}
Taking $a=1$ gives
\begin{equation}\label{eq:toric-average-scalar-curvature}
 \underline R=\frac1{|\Delta|}\int_{\partial\Delta}d\sigma>0.
\end{equation}
Hence
\begin{equation}\label{eq:calc-s2-023}
 \delta\mathcal M(u)[a]
     =\frac1{|\Delta|}\int_\Delta(R(\omega_\varphi)-\underline R)a\,dx,
\end{equation}
which agrees with the first variation of $\mathcal M$ by
Legendre duality.

\subsection{\texorpdfstring{$J$}{J} functional} 
In this subsection, we record basic properties of the $J$ functional,
defined by
\begin{equation}\label{eq:J-definition}
 J(\varphi):=\frac1V\int_M\varphi\,\omega_0^n-E(\varphi),
 \qquad \varphi\in\mathcal H.
\end{equation}
It satisfies $J\geq0$ and $J(\varphi+c)=J(\varphi)$. 

\begin{proposition}\label{prop:toric-J-L1}
Let $\varphi\in\mathcal H_0^T$, with symplectic potential $u$
as in Definition~\ref{def:moment-legendre}. The following hold.
\begin{enumerate}
\item[(i)] The energy normalization is equivalent to the fixed mean condition
\begin{equation}\label{eq:E-polytope-mean}
 E(\varphi)=-\frac1{|\Delta|}\int_\Delta(u-u_0)\,dx=0.
\end{equation}
\item[(ii)] There is a background constant $C$ such that
\begin{equation}\label{eq:toric-J-L1}
 C^{-1}\|u\|_{L^1(\Delta)}-C
       \le J(\varphi)\le C\|u\|_{L^1(\Delta)}+C.
\end{equation}
\item[(iii)] For each fixed $x_*\in\Delta^\circ$, there is a constant
$C$, depending only on the background and $x_*$, such that
\begin{equation}\label{eq:convex-interior-estimate}
 |u(x_*)|+|Du(x_*)|\leq C(1+J(\varphi)).
\end{equation}
\end{enumerate}
\end{proposition}

\begin{proof}
The energy identity is the volume-normalized form of
Donaldson's formula \cite[Lemma~3.2.6]{Donaldson2002}.
Along any smooth $T^n$-invariant path, Legendre duality and
\eqref{eq:toric-integration} give
$\frac d{ds}E(\varphi_s)
=-|\Delta|^{-1}\int_\Delta\dot u_s\,dx$.
This proves \eqref{eq:E-polytope-mean}.

Darvas's comparison between $d_1$ and $J$ \cite{Darvas2015}
gives, for every $\varphi\in\mathcal H_0$,
\begin{equation}\label{eq:d1-J-comparison}
 C^{-1}J(\varphi)-C\leq d_1(0,\varphi)\leq CJ(\varphi)+C.
\end{equation}
For $T^n$-invariant potentials, the distance formula of
Darvas--George--Smith \cite[Theorem~3.1]{DarvasGeorgeSmith2022} gives
\begin{equation}\label{eq:toric-d1-L1}
 d_1(0,\varphi)=\frac1{|\Delta|}\|u-u_0\|_{L^1(\Delta)}.
\end{equation}
Combining \eqref{eq:d1-J-comparison} and \eqref{eq:toric-d1-L1} with
$\bigl|\|u-u_0\|_{L^1(\Delta)}-\|u\|_{L^1(\Delta)}\bigr|
\leq\|u_0\|_{L^1(\Delta)}$ proves \eqref{eq:toric-J-L1}.

We also record a bound needed below. Since $E(\varphi)=0$,
\cite[Proposition~3.2]{DarvasGeorgeSmith2022} gives
$|J(\varphi)+\min_\Delta u|\leq C$.
In particular, $u\geq-J(\varphi)-C$, and order reversal under
Legendre duality gives
\begin{equation}\label{eq:toric-sup-J}
 \sup_M\varphi=\sup_\Delta(u_0-u)
 \leq\max_\Delta u_0-\min_\Delta u\leq J(\varphi)+C.
\end{equation}
The bound $\sup_M\varphi\leq J(\varphi)+C$ also holds without
torus invariance. Indeed, the standard supremum--mean estimate
and $E(\varphi)=0$ give
\begin{equation}\label{eq:general-sup-J}
 \sup_M\varphi\leq\frac1V\int_M\varphi\,\omega_0^n+C
                  =J(\varphi)+C;
\end{equation}
see \cite[(26)--(27)]{DarvasRubinstein2017}.

It remains to prove \eqref{eq:convex-interior-estimate}.
Fix $x_*\in\Delta^\circ$ and choose $r>0$ such that
$\overline{B_{2r}(x_*)}\subset\Delta^\circ$. We give the interior
estimate for convex functions used here. For every
$y\in\overline{B_r(x_*)}$, Jensen's inequality gives
\begin{equation}\label{eq:calc-s2-024}
 u(y)\leq\frac1{|B_r|}\int_{B_r(y)}u(x)\,dx
       \leq\frac1{|B_r|}\|u\|_{L^1(\Delta)}
       \leq C(1+J(\varphi)).
\end{equation}
The lower bound $u\geq-J(\varphi)-C$ proved above therefore gives
$|u(x_*)|\leq C(1+J(\varphi))$.
For every unit vector $e$, the supporting-plane inequality yields
\begin{equation}\label{eq:calc-s2-025}
 \frac{u(x_*)-u(x_*-re)}r
 \leq Du(x_*)\cdot e
 \leq\frac{u(x_*+re)-u(x_*)}r.
\end{equation}
The preceding upper bounds at $x_*\pm re$ and lower bound at
$x_*$ control both sides. Taking the supremum over unit vectors
therefore proves
\begin{equation}\label{eq:calc-s2-026}
 |u(x_*)|+|Du(x_*)|\leq C(1+J(\varphi)),
\end{equation}
where $C$ depends only on the fixed background and $x_*$.
This proves \eqref{eq:convex-interior-estimate}.
\end{proof}

\begin{proposition}\label{prop:toric-mabuchi-invariance}
If $\mathcal M$ is bounded below on $\mathcal H_0^T$, then
\begin{equation}\label{eq:calc-s2-027}
 \mathcal M(g\cdot\varphi)=\mathcal M(\varphi),
 \qquad g\in T_{\C},\quad\varphi\in\mathcal H_0^T.
\end{equation}
\end{proposition}

\begin{proof}
Let $u$ be the symplectic potential of
$\varphi\in\mathcal H_0^T$. We use the affine variation of the toric
Mabuchi functional
\cite[Proposition~3.2.8 and Lemma~3.2.9]{Donaldson2002}.
A translation in logarithmic coordinates,
induced by the positive real part of $T_{\C}$, changes $u$
by a linear function, with an additional constant allowed in the
potential normalization \cite[Proposition~3.1.7]{Donaldson2002}. The logarithmic determinant term in
the toric Mabuchi formula is unchanged, so for every affine
function $\ell$,
\begin{equation}\label{eq:calc-s2-028}
 \mathcal M(u+s\ell)-\mathcal M(u)
 =\frac{s}{|\Delta|}\left(
       \int_{\partial\Delta}\ell\,d\sigma
       -\underline R\int_\Delta\ell\,dx\right),\qquad s\in\R.
\end{equation}
The metrics remain smooth, and the additive constant of the
potential can always be chosen so that $E=0$.
The lower bound as $s\to\pm\infty$ forces the coefficient
of $s$ to vanish. The compact part of $T_{\C}$ fixes
the metric, proving the assertion. 
\end{proof}

\subsection{Properness}
For a subgroup $A\subset G$, set
\begin{equation}\label{eq:calc-s2-029}
 J_A(\varphi):=\inf_{g\in A}J(g\cdot\varphi),
 \qquad \varphi\in\mathcal H_0.
\end{equation}
Applying \eqref{eq:d1-J-comparison} to $g\cdot\varphi$ and
taking the infimum over $g\in A$ gives
\begin{equation}\label{eq:reduced-d1-J-comparison}
 C^{-1}J_A(\varphi)-C\leq d_{1,A}(0,\varphi)
                         \leq CJ_A(\varphi)+C.
\end{equation}
Thus properness in Definition~\ref{def:properness} bounds $J_A$
on every sublevel of $\mathcal M$.
We also evaluate these functionals at metrics:
$J_{\omega_0}(\omega_\varphi):=J(\varphi)$, and likewise for
$J_A$ and $\mathcal M$. A reference form in the subscript of
$E$ or $J$ specifies the background used to define that functional.

The inclusion $T_{\C}\subset G=\operatorname{Aut}_0(M)$ need not
be an equality; for example,
$\operatorname{Aut}_0(\mathbb P^n)=\operatorname{PGL}(n+1,\C)$.
We compare the reduced functionals $J_G$ and $J_{T_{\C}}$ on
$\mathcal H_0^T$, without assuming that $G$ is reductive.

We first state the comparison and its consequence for properness.

\begin{proposition}\label{prop:toric-group-reduction}
There are a compact subgroup $K\subset G$ containing $T^n$ and a
smooth $K$-invariant K\"ahler form $\omega_K\in[\omega_0]$ such that
every smooth $T^n$-invariant K\"ahler form $\eta\in[\omega_0]$ satisfies
\begin{equation}\label{eq:group-J-equality}
 \inf_{g\in G}J_{\omega_K}(g^*\eta)
       =\inf_{a\in T_{\C}}J_{\omega_K}(a^*\eta).
\end{equation}
For the original background there is a constant $C$, independent
of $\eta$, such that
\begin{equation}\label{eq:group-J-comparison}
 J_G(\eta)\le J_{T_{\C}}(\eta)\le J_G(\eta)+C.
\end{equation}
\end{proposition}

By Definition~\ref{def:properness} and
\eqref{eq:reduced-d1-J-comparison},
\eqref{eq:group-J-comparison} immediately gives the following corollary.

\begin{corollary}\label{cor:properness-reduction}
If $\mathcal M$ is proper on $\mathcal H_0^T$ modulo $G$, then it
is proper on $\mathcal H_0^T$ modulo $T_{\C}$. In particular,
for every $B$ there is a constant $C_B$ such that
\begin{equation}\label{eq:calc-s2-030}
 \varphi\in\mathcal H_0^T,\quad \mathcal M(\varphi)\le B
       \quad\Longrightarrow\quad J_{T_{\C}}(\varphi)\le C_B.
\end{equation}
\end{corollary}

The proof of Proposition~\ref{prop:toric-group-reduction} uses
the following two auxiliary lemmas.

The first lemma combines the positive grading of the toric
unipotent radical in \cite[Lemma~4.1]{BercziDoranHawesKirwan2018}
with the Iwasawa decomposition \cite[Proposition~7.31]{Knapp2002}.

\begin{lemma}\label{lem:toric-contracting-decomposition}
There are a compact subgroup $K\subset G$ containing $T^n$, a
unipotent subgroup $\mathcal N\subset G$, and an algebraic
one-parameter subgroup $\lambda:\C^*\to T_{\C}$ such that
\begin{equation}\label{eq:calc-s2-031}
 G=T_{\C}\mathcal N K.
\end{equation}
For every $n_0\in\mathcal N$, the map
$\zeta\mapsto\lambda(\zeta)n_0\lambda(\zeta)^{-1}$ extends holomorphically
to $\zeta=0$ with value the identity.
\end{lemma}

\begin{proof}
Demazure's description gives a Levi decomposition
$G=R_u(G)\rtimes L_{\mathrm{red}}$ with
$T_{\C}\subset L_{\mathrm{red}}$; see
\cite[the discussion preceding Proposition~2.2]{Nill2006}.
By \cite[Lemma~4.1 and its preceding construction]{BercziDoranHawesKirwan2018},
there is an algebraic one-parameter subgroup
$\lambda_0:\C^*\to T_{\C}$, central in $L_{\mathrm{red}}$,
whose conjugation action has strictly positive weights on
$\operatorname{Lie}R_u(G)$. If $R_u(G)$ is trivial, take
$\lambda_0$ to be trivial.
Choose a regular algebraic one-parameter subgroup
$\lambda_1:\C^*\to T_{\C}$ and the corresponding maximal
unipotent subgroup $N_+\subset L_{\mathrm{red}}$ on whose
Lie algebra its conjugation weights are positive. Set
$\mathcal N:=R_u(G)\rtimes N_+$. Since $\lambda_0$ is central
in $L_{\mathrm{red}}$, for a sufficiently large integer $m$ the subgroup
\begin{equation}\label{eq:calc-s2-036}
 \lambda(\zeta):=\lambda_0(\zeta)^m\lambda_1(\zeta)
\end{equation}
has strictly positive weights on $\operatorname{Lie}\mathcal N$.
Indeed, it has the same weights as $\lambda_1$ on
$\operatorname{Lie}N_+$, while the positive weights of $\lambda_0$
dominate those of $\lambda_1$ on $\operatorname{Lie}R_u(G)$.
Writing $n_0=\exp(\sum_{j>0}Y_j)$ in its positive integer
$\lambda$-weight spaces gives
\begin{equation}\label{eq:calc-s2-037}
 \lambda(\zeta)n_0\lambda(\zeta)^{-1}
     =\exp\Bigl(\sum_{j>0}\zeta^jY_j\Bigr).
\end{equation}
This proves the asserted holomorphic extension.

Choose a maximal compact subgroup $K$ of $L_{\mathrm{red}}$ containing $T^n$.
The Iwasawa decomposition of the reductive group $L_{\mathrm{red}}$
\cite[Proposition~7.31]{Knapp2002} gives
$L_{\mathrm{red}}=T_{\C}N_+K$. Since $T_{\C}$ normalizes $R_u(G)$,
\begin{equation}\label{eq:calc-s2-038}
 G=R_u(G)L_{\mathrm{red}}=T_{\C}R_u(G)N_+K=T_{\C}\mathcal N K.
\end{equation}
Only the Levi factor has been assumed reductive.
\end{proof}

We also use Hisamoto's orbitwise energy calculation
\cite[Theorem~2.5]{Hisamoto2020} in the following form.

\begin{lemma}\label{lem:J-orbit-psh}
For fixed smooth K\"ahler forms $\omega_K$ and $\eta$ in the
same class, the function
\begin{equation}\label{eq:calc-s2-039}
 g\longmapsto J_{\omega_K}(g^*\eta),\qquad g\in G,
\end{equation}
is plurisubharmonic.
\end{lemma}

Although stated in \cite{Hisamoto2020} for a polarized class,
the same fiber-integration argument applies to the real K\"ahler
class here. Indeed, $H^{1,0}(M)=0$ ensures that, locally in the
parameter, the pullback of $\eta$ by a holomorphic family of
automorphisms admits a smooth global potential relative to
the pullback of $\omega_K$ from $M$. The calculation of $dd^cJ$
in \cite[the proof of Theorem~2.5]{Hisamoto2020} then applies
without change.

\begin{proof}[Proof of Proposition~\ref{prop:toric-group-reduction}]
Choose $K$ as in Lemma~\ref{lem:toric-contracting-decomposition}
and average $\omega_0$ over $K$ to obtain $\omega_K$. This is a
smooth K\"ahler form in the same class, since $K\subset G$ acts
trivially on cohomology.
The $T^n$-invariance of $\eta$ and $K$-invariance of $\omega_K$
give
\begin{equation}\label{eq:calc-s2-043}
 J_{\omega_K}((\tau gk)^*\eta)=J_{\omega_K}(g^*\eta),
 \qquad \tau\in T^n,\quad k\in K.
\end{equation}
Write an arbitrary $g\in G$ as $g=an_0k$ by
Lemma~\ref{lem:toric-contracting-decomposition}.
Consider the holomorphic disk
\begin{equation}\label{eq:calc-s2-044}
 h(\zeta):=a\lambda(\zeta)n_0\lambda(\zeta)^{-1},\qquad h(0)=a.
\end{equation}
For $|\zeta|=1$, we have $\lambda(\zeta)\in T^n\subset K$, and $a$
commutes with $\lambda(\zeta)$. Hence
$J_{\omega_K}(h(\zeta)^*\eta)=J_{\omega_K}(g^*\eta)$ on the unit circle.
By Lemma~\ref{lem:J-orbit-psh}, the submean inequality gives
\begin{equation}\label{eq:calc-s2-045}
 J_{\omega_K}(a^*\eta)
 \le\frac1{2\pi}\int_0^{2\pi}
       J_{\omega_K}(h(e^{i\theta})^*\eta)\,d\theta
 =J_{\omega_K}(g^*\eta).
\end{equation}
Taking infima and using $T_{\C}\subset G$ proves
\eqref{eq:group-J-equality}.

It remains to compare the backgrounds. Write
$\omega_K=\omega_0+dd^c\psi_K$ and $\eta=\omega_0+dd^c v$.
The cocycle identity for $E$ gives
\begin{equation}\label{eq:calc-s2-046}
 \begin{split}
 J_{\omega_K}(\eta)-J_{\omega_0}(\eta)
 ={}&\frac1V\int_Mv(\omega_K^n-\omega_0^n)\\
 &+E_{\omega_0}(\psi_K)-\frac1V\int_M\psi_K\,\omega_K^n.
 \end{split}
\end{equation}
Integration by parts gives
\begin{equation}\label{eq:calc-s2-047}
 \int_Mv(\omega_K^n-\omega_0^n)
       =\sum_{j=0}^{n-1}\int_M\psi_K(\eta-\omega_0)
                         \wedge\omega_K^j\wedge\omega_0^{n-1-j}.
\end{equation}
Each positive mixed volume in this expression equals $V$.
It follows that
\begin{equation}\label{eq:calc-s2-048}
 |J_{\omega_K}(\eta)-J_{\omega_0}(\eta)|
       \le(2n+2)\|\psi_K\|_{L^\infty}.
\end{equation}
The same bound holds after taking an infimum over either group.
Combining the two bounds with \eqref{eq:group-J-equality}
proves \eqref{eq:group-J-comparison}.
\end{proof}

\subsection{The toric pseudo-Calabi flow.}
We now write the pseudo-Calabi flow on the fixed polytope.
Let $\varphi_t$ satisfy \eqref{eq:pcf}. We use a subscript $t$
for the associated quantities
$\omega_t,\Phi_t,u_t,x_t,m_t$ in
\eqref{eq:legendre}--\eqref{eq:toric-volume}.
The complex coordinates are fixed, whereas $x_t$ depends on time.
A dot denotes differentiation with respect to \(t\), with the moment coordinate \(x\) held fixed.

Differentiating the Legendre identity \eqref{eq:legendre} at fixed $x$ gives
\begin{equation}\label{eq:calc-s2-049}
 \dot u_t(x)
 =\langle x,\partial_tD_xu_t\rangle
  -\partial_t\Phi_t(D_xu_t)
  -\langle D_\xi\Phi_t(D_xu_t),\partial_tD_xu_t\rangle
 =-\partial_t\varphi_t(D_xu_t).
\end{equation}
The first and third terms cancel because $D_\xi\Phi_t(D_xu_t)=x$.
Here $\partial_t\Phi_t$ is computed at fixed $\xi$ before evaluation
at $D_xu_t(x)$. Using the moment-coordinate convention preceding
\eqref{eq:toric-integration}, we have
\begin{equation}\label{eq:calc-s2-050}
 \widehat F_t=-\log\det D_x^2u_t-\log m_0(D_xu_t).
\end{equation}
Hence \eqref{eq:pcf} becomes
\begin{equation}\label{eq:polytope-pcf}
 \begin{aligned}
 \dot u_t
 &=\log\det D_x^2u_t+\log m_0(D_xu_t)-\widehat P_t,\\
 \Delta_{u_t}\widehat P_t
 &=\underline R+\Delta_{u_t}\bigl[\log m_0(D_xu_t)\bigr],
 \qquad \int_\Delta\widehat P_t\,dx=0.
 \end{aligned}
\end{equation}
The second equation follows from
$\Ric(\omega_0)=-dd^c\log m_0$ on the open orbit.

Equation~\eqref{eq:polytope-pcf} is the polytope formulation of
the pseudo-Calabi flow. We next write the pseudo-Calabi flow in
complex logarithmic coordinates for use in the later sections.

\begin{proposition}\label{prop:toric-flow-identities}
For the smooth $T^n$-invariant pseudo-Calabi flow above, the
following identities hold in either normalization
\eqref{eq:pcf} or \eqref{eq:normalized-pcf-speed}.
\begin{enumerate}
\item[(i)] The Mabuchi functional satisfies
\begin{equation}\label{eq:polytope-dissipation}
 \frac d{dt}\mathcal M(u_t)
 =-\frac1{|\Delta|}\int_\Delta
           u_t^{ij}(\dot u_t)_i(\dot u_t)_j\,dx.
\end{equation}
\item[(ii)] On the open orbit, with time derivatives taken at fixed
$\xi$, we have
\begin{equation}\label{eq:fixed-complex-density}
 \begin{gathered}
 \partial_t\log m_t=\underline R-R(\omega_t),\qquad
 \Delta_\varphi\log m_t=-R(\omega_t),\\
 (\partial_t-\Delta_\varphi)\log m_t=\underline R.
 \end{gathered}
\end{equation}
\end{enumerate}
\end{proposition}

\begin{proof}
By \eqref{eq:polytope-logdet-curvature}, the pseudo-Calabi flow can
equivalently be written as
\begin{equation}\label{eq:polytope-pcf-laplacian}
 \partial_i\bigl(u_t^{ij}\partial_j\dot u_t\bigr)
       =-\partial_i\partial_j u_t^{ij}-\underline R.
\end{equation}
The additive constant of $\dot u_t$ is fixed by
\begin{equation}\label{eq:polytope-normalization}
 \int_\Delta\dot u_t\,dx=-\int_\Delta\widehat F_t\,dx.
\end{equation}

If instead $E(\varphi_t)=0$, the first equation in
\eqref{eq:polytope-pcf} has the additional term
$|\Delta|^{-1}\int_\Delta\widehat F_t\,dx$ on the right, and
$\int_\Delta\dot u_t\,dx=0$. The elliptic equations are unchanged.
In either normalization, the toric gradient formula and the
Legendre identity above give
\begin{equation}\label{eq:calc-s2-051}
 \|\nabla_{\varphi_t}\partial_t\varphi_t\|_{L^2(d\mu_{\varphi_t})}^2
 =\frac1{|\Delta|}\int_\Delta
       u_t^{ij}(\dot u_t)_i(\dot u_t)_j\,dx.
\end{equation}
Combining this equality with \eqref{eq:Mabuchi-dissipation}
proves \eqref{eq:polytope-dissipation}.

To prove \eqref{eq:fixed-complex-density}, we keep $\xi$ fixed
on the open orbit. By \eqref{eq:toric-volume},
$\log m_t=F_t+\log m_0$, where $m_0$ is independent of time.
The volume variation formula and Definition~\ref{def:pcf} give
\begin{align}
 \partial_t\log m_t =\partial_tF_t =\Delta_{\varphi_t}(\partial_t\varphi_t)=\underline R-R(\omega_t).
 \label{eq:calc-s2-052}
\end{align}
On the open orbit, $\Ric(\omega_t)=-dd^c\log m_t$.
Taking the trace gives
\begin{equation}\label{eq:calc-s2-053}
 \Delta_{\varphi_t}\log m_t=-R(\omega_t).
\end{equation}
Subtracting the two identities proves \eqref{eq:fixed-complex-density}.
Both conclusions are unchanged by adding a function of time to
$\varphi_t$: this only changes $\dot u_t$ by a spatial constant
and leaves $m_t$ unchanged.
\end{proof}

\section{\texorpdfstring{$C^0$}{C0} estimates}\label{sec:volume}

We first establish an $L^p$ bound for the volume ratio along the
toric pseudo-Calabi flow on finite time intervals.
We then prove a $C^0$ estimate with constants
independent of the pointwise density at the initial time.
Together, the two estimates give the finite-time bounds in
Corollary~\ref{cor:finite-time-volume}. 

\subsection{\texorpdfstring{$L^p$}{Lp} bound of the volume form}
We use the density in \eqref{eq:toric-volume}.
Time derivatives in this subsection are taken at fixed complex
coordinates, or equivalently at fixed $\xi$.

\begin{proposition}\label{prop:toric-density-lp}
Let $\varphi$ be a smooth $T^n$-invariant solution of \eqref{eq:pcf}
on $M\times[0,T)$, where $0<T<\infty$, with $\varphi(0)=0$.
There are $p>1$, depending only on $\Delta$, and $C_T>0$,
depending only on $(M,\omega_0)$, the fixed toric action, and $T$,
such that
\begin{equation}\label{eq:calc-s3-001}
 \sup_{0\leq t<T}\int_M e^{pF(t)}\,\omega_0^n\leq C_T.
\end{equation}
\end{proposition}
\begin{proof}
Let $b_0,\kappa_0$ be the exponents in
Proposition~\ref{prop:logarithmic-tail}, and set
\begin{equation}\label{eq:calc-s3-002}
 \alpha:=\frac{\kappa_0}{2},\qquad
 \rho(\xi):=\sqrt{1+|\xi|^2},\qquad
 \mathcal U:=\log m_t+\alpha\rho.
\end{equation}
Write $H_t:=D_\xi^2\Phi_t$. Since
\begin{equation}\label{eq:calc-s3-003}
 D_\xi^2\rho=\frac{I}{\rho}
                  -\frac{\xi\otimes\xi}{\rho^3}\geq0,
 \qquad
 \Delta_\varphi\rho
 =\operatorname{tr}(H_t^{-1}D_\xi^2\rho)\geq0,
\end{equation}
identity \eqref{eq:fixed-complex-density} gives
\begin{equation}\label{eq:weighted-log-density}
 (\partial_t-\Delta_\varphi)\mathcal U
 =\underline R-\alpha\Delta_\varphi\rho
 \leq\underline R.
\end{equation}

First fix $0<\tau<T$. Compactness of $M$ and smoothness on
$M\times[0,\tau]$ give
$m_t=e^{F(t)}m_0\leq C_\tau m_0$. This bound is used only to
control the behavior at infinity. Indeed,
\eqref{eq:background-exponential-tail} and $\alpha<\kappa_0$ imply
\begin{equation}\label{eq:calc-s3-004}
 \mathcal U(t,\xi)
 \leq C_\tau'-\kappa_0|\xi|
             +\alpha\sqrt{1+|\xi|^2}
 \longrightarrow-\infty
 \quad\text{as }|\xi|\to\infty,
\end{equation}
uniformly for $0\leq t\leq\tau$. With $\tau$ fixed, choose a
ball $B_R\subset\R^n$ sufficiently large that
$\mathcal U(t,\xi)-\underline Rt<\sup_{\R^n}\mathcal U(0)$
for all $0\leq t\leq\tau$ and $|\xi|\geq R$.
Thus the lateral boundary lies strictly below the initial
maximum, which is attained inside $B_R$. The weak parabolic
maximum principle on $B_R\times[0,\tau]$, together with the
bound outside $B_R$, gives
\begin{equation}\label{eq:calc-s3-005}
 \sup_{\R^n}\mathcal U(t)
 \leq\sup_{\R^n}\mathcal U(0)+\underline Rt.
\end{equation}
Since the initial metric is $\omega_0$, the number
$A_0:=\sup_{\R^n}(m_0e^{\alpha\rho})$ is finite, and hence
\begin{equation}\label{eq:weighted-density-bound}
 m_t(\xi)\leq A_0e^{\underline Rt-\alpha\rho(\xi)}
             \leq C_Te^{-\alpha|\xi|},
 \qquad 0\leq t\leq\tau.
\end{equation}
The constant $C_T$ depends on the initial metric and $T$, but not
on $C_\tau$ or $\tau$.

Choose
\begin{equation}\label{eq:calc-s3-006}
 p:=1+\frac{\alpha}{2b_0}>1,\qquad
 p\alpha-b_0(p-1)
 =\frac\alpha2+\frac{\alpha^2}{2b_0}>0.
\end{equation}
The toric boundary has zero background volume. Integrating over
the angular variables and using $e^F=m_t/m_0$, we have
\begin{align}
 \int_Me^{pF(t)}\,\omega_0^n
 &=n!(2\pi)^n\int_{\R^n}m_t^p m_0^{1-p}\,d\xi\notag\\
 &\leq C_T\int_{\R^n}
       e^{-[p\alpha-b_0(p-1)]|\xi|}\,d\xi
 \leq C_T.
 \label{eq:calc-s3-007}
\end{align}
Here we used \eqref{eq:weighted-density-bound} and the lower
bound for $m_0$ in \eqref{eq:background-exponential-tail}.
The exponent $p$ depends only on $\Delta$ in the fixed lattice
coordinates. Letting $\tau\uparrow T$ completes the proof.
\end{proof}

\subsection{\texorpdfstring{$C^0$}{C0} bound of the volume form} 
The following estimate converts a positive $L^p$ density bound
into pointwise bounds after any fixed waiting time.
It does not require torus invariance.

\begin{proposition}\label{prop:positive-time-volume}
Let $(M,\omega_0)$ be a connected compact K\"ahler manifold,
let $L>0$, and let $\varphi(t)$ be a smooth pseudo-Calabi flow
on $[s,s+L]$ with $E(\varphi(t))=0$.
Suppose that, for some $p>1$ and $B>0$,
\begin{equation}\label{eq:calc-s3-008}
 \sup_{s\leq t\leq s+L}\int_Me^{pF(t)}\,d\mu_0\leq B.
\end{equation}
There is $\beta\in(0,1)$, depending only on $(M,\omega_0)$ and $p$,
such that for every $0<\delta<L$ there is a constant $C>0$,
depending only on $(M,\omega_0)$, $p$, $B$, $L$, and $\delta$, with
\begin{equation}\label{eq:positive-time-volume}
 \sup_{s+\delta\leq t\leq s+L}
 \bigl(\|\varphi(t)\|_{C^\beta(M,\omega_0)}
       +\|P(t)\|_{L^\infty(M)}+\|F(t)\|_{L^\infty(M)}\bigr)
 \leq C.
\end{equation}
In particular, $C$ is independent of $s$ and
$\|F(s)\|_{L^\infty(M)}$. The estimate is uniform when $L$ is
bounded above and $\delta>0$ is uniformly bounded below.
\end{proposition}

We first prove a $L^{n+1}(d\mu_\varphi)$ bound for $P$ needed in the argument.

\begin{lemma}\label{lem:P-spatial-moment}
Let $(M,\omega_0)$ be a connected compact K\"ahler manifold,
let $\varphi\in\mathcal H$, and let $P$ satisfy the elliptic equation
and zero-mean normalization in \eqref{eq:pcf}.
If, for some $p>1$ and $B>0$,
\begin{equation}\label{eq:calc-s3-009}
 \int_Me^{pF}\,d\mu_0\leq B,
\end{equation}
then there is a constant $C$, depending only on $(M,\omega_0)$,
$p$, and $B$, such that
\begin{equation}\label{eq:P-spatial-moment-volume}
 \int_M(1+P^2)^{(n+1)/2}\,\omega_\varphi^n\leq C.
\end{equation}
\end{lemma}
\begin{proof}
Put $p':=p/(p-1)$. By \cite[Theorem~A*]{DemaillyEtAl}, for any fixed
$0<\beta<2/(1+np')$,
\begin{equation}\label{eq:potential-osc-volume}
 \|\varphi-\sup_M\varphi\|_{C^\beta(M,\omega_0)}
 \leq C.
\end{equation}

By the assumed density bound and the Poincar\'e and Sobolev
estimates of Guo--Phong--Song--Sturm
\cite[Theorem~3.1(2), (4)]{GPSS2024II} and Guedj--T\^o
\cite[Theorem~2.6(1)]{GuedjTo2025}, we may fix
$1<\chi<n/(n-1)$ when $n\geq2$, and $\chi=2$ when $n=1$, such that
\begin{align}
 \left\|v-\int_Mv\,d\mu_\varphi\right\|_{L^2(d\mu_\varphi)}^2
 &\leq C\|\nabla_\varphi v\|_{L^2(d\mu_\varphi)}^2,\notag\\
 \|v\|_{L^{2\chi}(d\mu_\varphi)}^2
 &\leq C\left(\|\nabla_\varphi v\|_{L^2(d\mu_\varphi)}^2
                  +\|v\|_{L^2(d\mu_\varphi)}^2\right).
 \label{eq:volume-subcritical-sobolev}
\end{align}
Integration by parts gives
\begin{align}
 \|\nabla_\varphi\varphi\|_{L^2(d\mu_\varphi)}^2
 &=-\int_M(\varphi-\sup_M\varphi)
             (n-\tr_{\omega_\varphi}\omega_0)\,d\mu_\varphi\notag\\
 &\leq n\operatorname{osc}_M\varphi\leq C.
 \label{eq:calc-s3-010}
\end{align}
Choose $C_R$ such that
$-C_R\omega_0\leq\Ric(\omega_0)\leq C_R\omega_0$.
The equation for $P$, its zero mean, and
$|\nabla_\varphi\sqrt{1+P^2}|_\varphi
 \leq|\nabla_\varphi P|_\varphi$ give
\begin{align}
 \|\nabla_\varphi P\|_{L^2(d\mu_\varphi)}^2
 &=\int_M P\tr_{\omega_\varphi}\Ric(\omega_0)\,d\mu_\varphi\notag\\
 &\leq C_R\int_M\sqrt{1+P^2}\,
                   \tr_{\omega_\varphi}\omega_0\,d\mu_\varphi\notag\\
 &=nC_R\int_M\sqrt{1+P^2}\,d\mu_\varphi
+C_R\int_M
       \langle\nabla_\varphi\sqrt{1+P^2},\nabla_\varphi\varphi
                    \rangle_\varphi\,d\mu_\varphi\notag\\
 &\leq C\bigl(1+\|P\|_{L^2(d\mu_\varphi)}
                   +\|\nabla_\varphi P\|_{L^2(d\mu_\varphi)}\bigr)\notag\\
 &\leq C\bigl(1+\|\nabla_\varphi P\|_{L^2(d\mu_\varphi)}\bigr).
 \label{eq:calc-s3-011}
\end{align}
The last step uses the first inequality in
\eqref{eq:volume-subcritical-sobolev}. Hence
\begin{equation}\label{eq:P-second-moment-volume}
 \|\nabla_\varphi P\|_{L^2(d\mu_\varphi)}^2
       +\int_MP^2\,d\mu_\varphi\leq C.
\end{equation}

Choose $\lambda>C_R$ and put
\begin{equation}\label{eq:calc-s3-012}
 h:=\sqrt{1+P^2}-\lambda(\varphi-\sup_M\varphi)\geq1.
\end{equation}
By \eqref{eq:potential-osc-volume} and
\eqref{eq:P-second-moment-volume}, $\|h\|_{L^2(d\mu_\varphi)}\leq C$.
Direct calculation gives
\begin{align}
 \Delta_\varphi h
 &=\frac{|\nabla_\varphi P|_\varphi^2}{(1+P^2)^{3/2}}
    +\frac P{\sqrt{1+P^2}}\bigl(\underline R
                  -\tr_{\omega_\varphi}\Ric(\omega_0)\bigr)
    -\lambda(n-\tr_{\omega_\varphi}\omega_0)\notag\\
 &\geq-C+(\lambda-C_R)\tr_{\omega_\varphi}\omega_0\geq-C.
 \label{eq:calc-s3-013}
\end{align}
For $k\geq2$, multiply by $h^{k-1}$ and integrate to obtain
\begin{equation}\label{eq:calc-s3-014}
 (k-1)\int_M h^{k-2}|\nabla_\varphi h|_\varphi^2\,d\mu_\varphi
 \leq C\int_Mh^{k-1}\,d\mu_\varphi.
\end{equation}
Apply the second inequality in \eqref{eq:volume-subcritical-sobolev}
to $h^{k/2}$. Since $h\geq1$, this yields
\begin{equation}\label{eq:calc-s3-015}
 \|h\|_{L^{k\chi}(d\mu_\varphi)}^k
 \leq C\left(\frac{Ck^2}{4(k-1)}
             \int_Mh^{k-1}\,d\mu_\varphi
             +\int_Mh^k\,d\mu_\varphi\right)
 \leq C_k\|h\|_{L^k(d\mu_\varphi)}^k.
\end{equation}
Apply this estimate successively with $k=2,2\chi,2\chi^2,\ldots$
until the resulting exponent is at least $n+1$. Since $\chi>1$
depends only on the stated data, the number of applications is fixed.
Since $\sqrt{1+P^2}\leq h$ and
$d\mu_\varphi$ has total mass one, this proves
\eqref{eq:P-spatial-moment-volume}.
In complex dimension one the second-moment bound already suffices.
\end{proof}

\begin{proof}[Proof of Proposition~\ref{prop:positive-time-volume}]
We begin with the calculation used to estimate the inverse density.
Let $\lambda,\varepsilon>0$ and let $\psi(t)\in\mathcal H$ be a
smooth family on $[s,s+L]$, all to be chosen later.
Differentiating with respect to the fixed measure $d\mu_0$ gives
\begin{align}
 &\frac d{dt}\int_Me^{-2F-\lambda\varphi+\varepsilon\psi}\,d\mu_0\notag\\
 &\quad=-\frac23\int_Me^{-3F-\lambda\varphi+\varepsilon\psi}
       |\nabla_\varphi(-3F-\lambda\varphi+\varepsilon\psi)|_\varphi^2
       \,d\mu_\varphi\notag\\
 &\qquad+\int_Me^{-2F-\lambda\varphi+\varepsilon\psi}
       \left[\frac{2\lambda}{3}\Delta_\varphi\varphi
       -2\bigl(\underline R-\tr_{\omega_\varphi}\Ric(\omega_0)\bigr)\right]
       \,d\mu_0\notag\\
 &\qquad+\int_Me^{-2F-\lambda\varphi+\varepsilon\psi}
       \left[\varepsilon\left(\partial_t\psi-\frac23\Delta_\varphi\psi\right)
       -\lambda\left(F+P-\int_MF\,d\mu_\varphi\right)\right]
       \,d\mu_0.
 \label{eq:auxiliary-weighted-moment} 
\end{align}
The first term on the right-hand side
of \eqref{eq:auxiliary-weighted-moment} is nonpositive.
We estimate the remaining terms in turn.

For the normalization term, Jensen's inequality gives
\begin{align}
 \int_MF\,d\mu_\varphi
 &\leq\frac1{p-1}\log\int_Me^{(p-1)F}\,d\mu_\varphi=\frac1{p-1}\log\int_Me^{pF}\,d\mu_0
 \leq\frac{\log B}{p-1}.
 \label{eq:calc-s3-016}
\end{align}
Consequently, its contribution to
\eqref{eq:auxiliary-weighted-moment} satisfies
\begin{equation}\label{eq:calc-s3-017}
 \lambda\left(\int_MF\,d\mu_\varphi\right)
 \int_Me^{-2F-\lambda\varphi+\varepsilon\psi}\,d\mu_0
 \leq\frac{\lambda\log B}{p-1}
 \int_Me^{-2F-\lambda\varphi+\varepsilon\psi}\,d\mu_0.
\end{equation}

For $K>0$ to be fixed later, let $\psi$ solve the auxiliary equation
of Chen--Cheng \cite{ChenChengFlow} on $[s,s+L]$:
\begin{equation}\label{eq:volume-auxiliary-flow}
 -\partial_t\psi\,\omega_\psi^n
 =K(1+P^2)^{(n+1)/2}\,\omega_\varphi^n,
 \qquad \psi(s)=0,\qquad \omega_\psi>0.
\end{equation}
The right-hand side is smooth and strictly positive on this
interval. By \cite[Proposition~2.3]{ChenChengFlow}, there is a unique
smooth admissible solution. The equation gives $\partial_t\psi<0$
and hence $\psi\leq0$.
Since $-\partial_t\psi>0$, the arithmetic--geometric mean inequality
and \eqref{eq:volume-auxiliary-flow} give
\begin{align}
 -\partial_t\psi+\frac23\tr_{\omega_\varphi}\omega_\psi
 &\geq(n+1)(2/3)^{n/(n+1)}
       \left((-\partial_t\psi)
             \frac{\omega_\psi^n}{\omega_\varphi^n}\right)^{1/(n+1)}\notag\\
 &=(n+1)(2/3)^{n/(n+1)}K^{1/(n+1)}\sqrt{1+P^2}.
 \label{eq:auxiliary-volume-amgm}
\end{align}

Choose $C_R$ such that
$-C_R\omega_0\leq\Ric(\omega_0)\leq C_R\omega_0$. 
For each $\varepsilon>0$, choose $\lambda$ and then $K$ so that
\begin{equation}\label{eq:calc-s3-018}
 \frac23(\lambda-\varepsilon)\geq2C_R+1,
 \qquad
 (n+1)(2/3)^{n/(n+1)}\varepsilon K^{1/(n+1)}\geq\lambda.
\end{equation}
Then we obtain
\begin{align}
 &\frac d{dt}\int_Me^{-2F-\lambda\varphi+\varepsilon\psi}\,d\mu_0\notag\\
 &\quad\leq\int_Me^{-2F-\lambda\varphi+\varepsilon\psi}
   \Bigl[\frac{2n\lambda}{3}-2\underline R
       +\frac{\lambda\log B}{p-1}
       -\tr_{\omega_\varphi}\omega_0-\lambda F\notag\\
 &\qquad\qquad
       +\lambda|P|-\lambda\sqrt{1+P^2}\Bigr]\,d\mu_0\notag\\
 &\quad\leq\int_Me^{-2F-\lambda\varphi+\varepsilon\psi}
   \Bigl[\frac{2n\lambda}{3}-2\underline R
       +\frac{\lambda\log B}{p-1}
       -\tr_{\omega_\varphi}\omega_0-\lambda F\Bigr]\,d\mu_0.
 \label{eq:auxiliary-absorbed-moment}
\end{align}

We next estimate the remaining exponential factors.
By Tian's $\alpha$-invariant estimate
\cite[Proposition~2.1]{Tian1987}, fix a sufficiently small
$\varepsilon>0$, depending only on the background, such that
\begin{equation}\label{eq:background-exponential-integrability}
 \int_Me^{-\varepsilon(\eta-\sup_M\eta)}\,d\mu_0
 \leq C,\qquad \eta\in\mathcal H.
\end{equation}
We take $\lambda$, $K$, and $\psi$ as above for this choice of
$\varepsilon$. The constants $\varepsilon$, $\lambda$, and $K$
now depend only on the background.
Since $\tr_{\omega_\varphi}\omega_0\geq ne^{-F/n}$,
we have
\begin{equation}\label{eq:calc-s3-019}
 -\tfrac12\tr_{\omega_\varphi}\omega_0-\lambda F
 \leq-\tfrac n2e^{-F/n}-\lambda F\leq C.
\end{equation}
Therefore we have
\begin{equation}\label{eq:positive-time-negative-moment}
 \frac d{dt}\int_Me^{-2F-\lambda\varphi+\varepsilon\psi}\,d\mu_0
 \leq C\int_Me^{-2F-\lambda\varphi+\varepsilon\psi}\,d\mu_0
       -\frac n2\int_Me^{-(2+1/n)F-\lambda\varphi+\varepsilon\psi}
                  \,d\mu_0.
\end{equation}

To close this differential inequality, we use
\eqref{eq:potential-osc-volume}, which bounds
$\varphi-\sup_M\varphi$ in $C^\beta$.
Since $\inf_M\varphi\leq E(\varphi)=0\leq\sup_M\varphi$,
it follows that $\|\varphi\|_{C^\beta(M,\omega_0)}\leq C$.
Since $\psi\leq0$ and $\|\varphi\|_\infty\leq C$,
H\"older's inequality gives
\begin{align}
 \int_Me^{-2F-\lambda\varphi+\varepsilon\psi}\,d\mu_0
 &\leq
 \left(\int_Me^{-(2+1/n)F-\lambda\varphi+\varepsilon\psi}
                    \,d\mu_0\right)^{2n/(2n+1)}
 \left(\int_Me^{-\lambda\varphi+\varepsilon\psi}
                    \,d\mu_0\right)^{1/(2n+1)}\notag\\
 &\leq C
 \left(\int_Me^{-(2+1/n)F-\lambda\varphi+\varepsilon\psi}
                    \,d\mu_0\right)^{2n/(2n+1)}.
 \label{eq:calc-s3-020}
\end{align} 
Consequently, \eqref{eq:positive-time-negative-moment} implies 
\begin{align}
 \frac d{dt}\int_Me^{-2F-\lambda\varphi+\varepsilon\psi}\,d\mu_0
 &\leq C\int_Me^{-2F-\lambda\varphi+\varepsilon\psi}\,d\mu_0\notag\\
 &\quad-C^{-1}\left(\int_Me^{-2F-\lambda\varphi+\varepsilon\psi}
                       \,d\mu_0\right)^{1+1/(2n)}.
 \label{eq:calc-s3-021}
\end{align}
Multiplying by $(t-s)^{2n+1}$, for $s<t\leq s+L$, we compute
\begin{align}
 &\frac d{dt}\left[(t-s)^{2n+1}
       \int_Me^{-2F-\lambda\varphi+\varepsilon\psi}\,d\mu_0\right]\notag\\
 &\quad=(2n+1)(t-s)^{2n}
       \int_Me^{-2F-\lambda\varphi+\varepsilon\psi}\,d\mu_0\notag\\
 &\qquad\quad+(t-s)^{2n+1}\frac d{dt}
       \int_Me^{-2F-\lambda\varphi+\varepsilon\psi}\,d\mu_0\notag\\
 &\quad\leq\bigl(2n+1+C(t-s)\bigr)(t-s)^{2n}
       \int_Me^{-2F-\lambda\varphi+\varepsilon\psi}\,d\mu_0\notag\\
 &\qquad\quad-C^{-1}(t-s)^{2n+1}
       \left(\int_Me^{-2F-\lambda\varphi+\varepsilon\psi}
                         \,d\mu_0\right)^{1+1/(2n)}\notag\\
 &\quad\leq(2n+1+CL)\left[(t-s)^{2n}
       \int_Me^{-2F-\lambda\varphi+\varepsilon\psi}\,d\mu_0\right]\notag\\
 &\qquad\quad-C^{-1}\left[(t-s)^{2n}
       \int_Me^{-2F-\lambda\varphi+\varepsilon\psi}
                         \,d\mu_0\right]^{1+1/(2n)}\notag\\
 &\quad\leq C_L-\frac1{2C}\left[(t-s)^{2n}
       \int_Me^{-2F-\lambda\varphi+\varepsilon\psi}
                         \,d\mu_0\right]^{1+1/(2n)}\notag\\
 &\quad\leq C_L.
 \label{eq:calc-s3-022}
\end{align}
The penultimate inequality is Young's inequality with conjugate
exponents $1+1/(2n)$ and $2n+1$. Here $C_L$ depends only on
the preceding constants and $L$. Since the solution is smooth up
to $t=s$, the integral is finite there, and the expression inside
the time derivative tends to zero as $t\downarrow s$.
Integrating \eqref{eq:calc-s3-022} from $s$ to $t$, we obtain
\begin{align}
 &\int_Me^{-2F(t)-\lambda\varphi(t)+\varepsilon\psi(t)}\,d\mu_0\notag\\
 &\quad=\frac1{(t-s)^{2n+1}}\int_s^t\frac d{dr}
       \left[(r-s)^{2n+1}
       \int_Me^{-2F(r)-\lambda\varphi(r)+\varepsilon\psi(r)}
                         \,d\mu_0\right]\,dr\notag\\
 &\quad\leq\frac1{(t-s)^{2n+1}}\int_s^t C_L\,dr
       =C_L(t-s)^{-2n}.
 \label{eq:calc-s3-023}
\end{align}
It remains to bound
$\int_Me^{-\varepsilon\psi(t)}\,d\mu_0$.
Applying Lemma~\ref{lem:P-spatial-moment} at each time gives
\begin{equation}\label{eq:calc-s3-024}
 \int_M(1+P^2)^{(n+1)/2}\,\omega_\varphi^n\leq C.
\end{equation}
The energy calculation of \cite[Lemma~2.1 and Corollary~2.2]{ChenChengFlow}
gives, in our normalization,
\begin{equation}\label{eq:calc-s3-025}
 \frac d{dt}E(\psi(t))
 =-K\int_M(1+P^2)^{(n+1)/2}\,d\mu_\varphi\geq-CK.
\end{equation}
Since $E(\psi(s))=0$ and $E(\psi)\leq\sup_M\psi$, we obtain
$\sup_M\psi(t)\geq-CK(t-s)$. Together with
\eqref{eq:background-exponential-integrability}, this gives
\begin{equation}\label{eq:auxiliary-weight-integrability}
 \begin{aligned}
 \int_Me^{-\varepsilon\psi(t)}\,d\mu_0
 &=e^{-\varepsilon\sup_M\psi(t)}
   \int_Me^{-\varepsilon(\psi(t)-\sup_M\psi(t))}\,d\mu_0\\
 &\leq Ce^{\varepsilon CK(t-s)}\leq Ce^{\varepsilon CKL}.
 \end{aligned}
\end{equation}
This bound depends only on the stated data
and $L$, and not on $s$ or the pointwise initial density. Hence we have 
\begin{align}
 \int_Me^{-F(t)}\,d\mu_0
 &\leq
 \left(\int_Me^{-2F(t)-\lambda\varphi(t)+\varepsilon\psi(t)}
                      \,d\mu_0\right)^{1/2}
 \left(\int_Me^{\lambda\varphi(t)-\varepsilon\psi(t)}
                      \,d\mu_0\right)^{1/2}\notag\\
 &\leq C(t-s)^{-n}.
 \label{eq:positive-time-inverse-density}
\end{align}
The positive $L^p$ density bound and
\eqref{eq:positive-time-inverse-density} now satisfy both density
hypotheses of \cite[Theorem~3.2]{ChengTian2026} on
$[s+\delta/2,s+L]$. Hence
\begin{equation}\label{eq:calc-s3-026}
 \|P(t)\|_\infty\leq C,\qquad s+\delta/2\leq t\leq s+L.
\end{equation}
From this point on, $C$ may also depend on $\delta$.

It remains to bound $F$ pointwise after a further waiting time.
Put $\tau:=t-s-\delta/2$ and choose a fixed $A>1$ such that
$A\omega_0\pm\tau\Ric(\omega_0)\geq\omega_0$ for $0\leq\tau\leq L$.
We compute
\begin{align}
 (\partial_t-\Delta_\varphi)(\tau F-A\varphi)
 &=(1-A)F-AP+A\int_MF\,d\mu_\varphi+An+\tau\underline R\notag\\
 &\quad-\tr_{\omega_\varphi}
                    (A\omega_0+\tau\Ric(\omega_0))\notag\\
 &\leq-(A-1)F+C,
 \label{eq:calc-s3-027}\\
 (\partial_t-\Delta_\varphi)(\tau F+A\varphi)
 &=(1+A)F+AP-A\int_MF\,d\mu_\varphi-An+\tau\underline R\notag\\
 &\quad+\tr_{\omega_\varphi}
                    (A\omega_0-\tau\Ric(\omega_0))\notag\\
 &\geq(A+1)F+ne^{-F/n}-C.
 \label{eq:calc-s3-028}
\end{align}
The last inequality uses
$\tr_{\omega_\varphi}\omega_0\geq ne^{-F/n}$.

Fix $T\in(s+\delta/2,s+L]$ and take a maximum of
$\tau F-A\varphi$ on $M\times[s+\delta/2,T]$.
Here $\tau=0$ at the initial time, so the $F$ term vanishes.
At the initial time its value is $-A\varphi(s+\delta/2)$,
which is bounded above. Otherwise, the first differential
inequality gives $F\leq C/(A-1)$ at the maximum.
Since $\tau\leq L$ and $\|\varphi\|_\infty\leq C$, the maximum
is again bounded above. Thus $\tau F-A\varphi\leq C$.
Similarly, at a minimum of $\tau F+A\varphi$,
\begin{equation}\label{eq:calc-s3-029}
 (A+1)F+ne^{-F/n}-C\leq0.
\end{equation}
The left-hand side tends to $+\infty$ as $F\to-\infty$,
so $F\geq-C$ there. The initial value is also bounded below,
and hence $\tau F+A\varphi\geq-C$ on the cylinder.
It follows that $\tau\|F(t)\|_\infty\leq C$.
On $[s+\delta,s+L]$, we have $\tau\geq\delta/2$, proving
\eqref{eq:positive-time-volume}. No pointwise bound for $F$
at either left endpoint enters the constants. 
\end{proof}

Combining Propositions~\ref{prop:toric-density-lp}
and~\ref{prop:positive-time-volume} gives the finite-time $C^0$ estimates
for the pseudo-Calabi flow.

\begin{corollary}\label{cor:finite-time-volume}
Let $\varphi$ be a smooth $T^n$-invariant solution of \eqref{eq:pcf}
on $M\times[0,T)$, where $0<T<\infty$, with $\varphi(0)=0$.
There are constants $C_T$ and $\beta>0$, depending only on
$(M,\omega_0)$, the fixed toric action, and $T$, such that
\begin{equation}\label{eq:calc-s3-030}
 \begin{gathered}
 C_T^{-1}\leq e^{F(t)}\leq C_T,\qquad 0\leq t<T,\\
 \|\varphi(t)\|_{C^\beta(M,\omega_0)}
 +\|P(t)\|_{L^\infty(M)}
 +\|\partial_t\varphi(t)\|_{L^\infty(M)}
 +\|u_t-u_0\|_{L^\infty(\Delta)}
 \leq C_T.
 \end{gathered}
\end{equation}
\end{corollary}

\section{Estimates for the pseudo term \texorpdfstring{$P$}{P}}
\label{sec:P-estimates}

We prove two estimates for $P$. The first gives a uniform spatial
modulus under the finite-time bounds and completes the proof of
long-time existence. The second gives negative exponential
integrability from bounds for $J$ and the entropy. It will be used
in the convergence argument of Section~\ref{sec:proper-convergence}.

\subsection{A modulus of continuity}
\label{sec:P-modulus}

The aim of this subsection is to prove the following uniform
modulus of continuity for $P$.

\begin{proposition}
\label{prop:common-modulus}
Let $\omega_t=\omega_{\varphi_t}$ be a family of smooth
$T^n$-invariant K\"ahler metrics in the fixed class, and let $P_t$ be
the normalized solution of the elliptic equation in \eqref{eq:pcf}.
Suppose that, uniformly in $t$,
\begin{equation}
 c\leq\frac{\omega_t^n}{\omega_0^n}\leq C,\qquad
 \|\varphi_t\|_{C^\beta(M,\omega_0)}+\|P_t\|_\infty\leq C,
 \qquad \beta\in(0,1],
 \label{eq:basic-modulus-inputs}
\end{equation}
Then there is a common spatial modulus $\rho(r)\to0$ as $r\to0$,
measured using the fixed background metric:
\begin{equation}
 |P_t(y)-P_t(y')|\leq\rho(d_0(y,y')).
 \label{eq:common-physical-modulus}
\end{equation}
\end{proposition}

\subsubsection{Proof of Theorem~\ref{thm:main}}
We first use it to complete the proof of Theorem~\ref{thm:main}.
For this purpose, we use the following continuation criterion of
Cheng--Tian, which follows from
\cite[Definition~4.3 and Corollary~4.7]{ChengTian2026}.

\begin{theorem}\label{thm:continuation-criterion}
Let a smooth solution of \eqref{eq:pcf} on a compact K\"ahler manifold
be defined on $[0,T)$, where $T<\infty$. If
\begin{equation}\label{eq:continuation-input}
 \sup_{t<T}\|F(t)\|_\infty<\infty,\qquad
 |P(t,q)-P(t,q')|\leq\rho(d_0(q,q')),
 \quad\lim_{r\downarrow0}\rho(r)=0,
\end{equation}
where $\rho:[0,\infty)\to[0,\infty)$ is continuous,
nondecreasing, and independent of $t$, then the solution extends
smoothly beyond $T$.
\end{theorem}

\begin{proof}[Proof of Theorem~\ref{thm:main}]
Short-time existence and uniqueness follow from
\cite{ChenZheng2013}. Uniqueness also preserves $T^n$-invariance:
pullback by any element of $T^n$ solves the same equations with the
same initial metric. The scalar normalization in \eqref{eq:pcf}
fixes the additive constant of the potential.

Let $[0,T_{\max})$ be the maximal existence interval. Suppose
that $T_{\max}<\infty$.
Corollary~\ref{cor:finite-time-volume} supplies
\eqref{eq:basic-modulus-inputs} and a uniform bound for $F$ on
this interval.
By Proposition~\ref{prop:common-modulus}, $P(t)$ has a uniform
modulus of continuity with respect to $d_0$ for $t<T_{\max}$.
Cheng--Tian's continuation criterion,
Theorem~\ref{thm:continuation-criterion}, then extends the flow smoothly past
$T_{\max}$. This is a contradiction. Hence $T_{\max}=\infty$.
\end{proof}

It remains to prove Proposition~\ref{prop:common-modulus}.
Locally, the equation for $P$ becomes a harmonic equation after
subtracting explicit potential and background terms. We first
choose sublevel domains on which the volume bounds are uniform.
A relative Green estimate and an estimate for positive and negative
small powers then give oscillation decay for $T^n$-invariant harmonic
functions. Finally, Legendre duality relates these sublevel sets
to the fixed topology of $M$.

\subsubsection{Control of drifted domain}
For the family in Proposition~\ref{prop:common-modulus}, we use
the moment maps and Legendre transforms of \eqref{eq:legendre}.
A domain $D\subset\C^n$ is called Reinhardt if it is preserved by
the coordinatewise rotations of $T^n$. On the open orbit, a
$T^n$-invariant function is independent of the angular variables
and hence depends only on $\xi$. For the toric domains constructed
below, we take the logarithmic image in the fixed character coordinates:
\begin{equation}\label{eq:calc-s4-001}
 \log D:=\{\xi(y):y\in D\cap(\C^*)^n\}.
\end{equation}
The domain is logarithmically convex if this image is convex.
A monomial change of character basis induces an invertible linear
change of logarithmic variables, which preserves convexity.

\begin{definition} 
\label{def:moment-sublevel-sets}
For the facets $\mathcal F_r$, $1\leq r\leq d$, of $\Delta$, let
$\mathcal D_r:=x_0^{-1}(\mathcal F_r)$ be the corresponding fixed toric divisors;
equivalently, $\mathcal D_r=x_t^{-1}(\mathcal F_r)$ for every $t$.
For $p\in\Delta$, put
\begin{equation}\label{eq:calc-s4-002}
 U_p:=M\setminus\bigcup_{\ell_r(p)>0}\mathcal D_r.
\end{equation}
On the open orbit define
\begin{equation}
 \begin{aligned}
 \psi_{p,t}(y)
   &:=\Phi_t(\xi(y))-\langle p,\xi(y)\rangle+u_t(p),\\
 D_{p,s,t}&:=\{y\in U_p:\psi_{p,t}(y)<s\},\qquad s>0,
 \end{aligned}
 \label{eq:fixed-moment-domains}
\end{equation}
where $\psi_{p,t}$ is extended smoothly to $U_p$ as verified below.
We call $D_{p,s,t}$ the sublevel set of height $s$ centered at
the moment fiber $x_t^{-1}(p)$.
\end{definition} 

\begin{proposition}\label{prop:sublevel-geometry}
For every $p\in\Delta$ and every $t$, the function $\psi_{p,t}$
extends smoothly to $U_p$, satisfies $dd^c\psi_{p,t}=\omega_t$ and
$\psi_{p,t}\geq0$, and vanishes on the moment fiber $x_t^{-1}(p)$.
\end{proposition} 

\begin{proof}
We describe the divisors in $U_p$ using the toric charts and
orbit correspondence of
\cite[Example~1.2.21 and Theorem~3.2.6]{CoxLittleSchenck2011}.
Let $\mathcal F(p)$ be the smallest face containing $p$, of
codimension $k$, and choose a vertex $v$ of $\mathcal F(p)$.
A complex torus orbit belongs to $U_p$ precisely when its corresponding face
contains $\mathcal F(p)$. Every such face contains $v$, so
$U_p\subset U_v\simeq\C^n$.
Order the $n$ facets through $v$ as
$\mathcal F_{r_1},\ldots,\mathcal F_{r_n}$ so that the first $k$
contain $p$. In the corresponding holomorphic vertex coordinates
$w_1,\ldots,w_n$,
\begin{equation}\label{eq:calc-s4-003}
 \mathcal D_{r_i}\cap U_v=\{w_i=0\},\qquad
 U_p=\{w_{k+1}\ne0,\ldots,w_n\ne0\}
       \simeq\C^k\times(\C^*)^{n-k}.
\end{equation}
Thus the divisors retained in $U_p$ are coordinate hyperplanes
inside this open set. At their intersections several of the
first $k$ coordinates vanish. When $p=v$ is a vertex, $k=n$,
there are no nonvanishing conditions, and $U_p=U_v$ contains
the coordinate origin. We fix this vertex chart when using
Euclidean volume below.

Write $u_t=u_G+h_t$, with $h_t\in C^\infty(\Delta)$, and put
$x:=x_t(y)$ on the open orbit. The Legendre identity gives
\begin{align} 
\psi_{p,t}(y)
 &=\langle x-p,D u_t(x)\rangle-u_t(x)+u_t(p) \notag\\
 &=\left\langle x-p,
       \sum_r\nu_r\bigl(\log\ell_r(x)+1\bigr)+D h_t(x)\right\rangle
       -\sum_r\ell_r(x)\log\ell_r(x)-h_t(x)+u_t(p) \notag\\
 &=\sum_r\bigl(\langle x-p,\nu_r\rangle-\ell_r(x)\bigr)
       \log\ell_r(x)
       +\left\langle x-p,\sum_r\nu_r+D h_t(x)\right\rangle
       -h_t(x)+u_t(p) \notag\\
 &=\sum_r\bigl(\lambda_r-\langle p,\nu_r\rangle\bigr)
       \log\ell_r(x)
       +\left\langle x-p,\sum_r\nu_r+D h_t(x)\right\rangle
       -h_t(x)+u_t(p) \notag\\
 &=-\sum_r\ell_r(p)\log\ell_r(x)
       +\left\langle x-p,\sum_r\nu_r+D h_t(x)\right\rangle
       -h_t(x)+u_t(p).
 \label{eq:sublevel-boundary-formula}
\end{align}
If $\ell_r(p)=0$, the corresponding logarithmic term cancels identically on the open orbit
and is omitted before taking any boundary value. If
$\ell_r(p)>0$, the divisor $\mathcal D_r$ has been removed from
$U_p$, so $\ell_r\circ x_t$ is a smooth strictly positive function
there. Its logarithm is therefore smooth on $U_p$.
The terms not involving logarithms are smooth on all of $M$,
since $x_t$ is a smooth moment map and $h_t\in C^\infty(\Delta)$.
Consequently \eqref{eq:sublevel-boundary-formula} defines a smooth
extension of $\psi_{p,t}$ across each retained coordinate hyperplane
and their intersections, including the origin when $p$ is a vertex.
On the open orbit, $\xi_i=\log|z_i|^2$ gives
$\bar\partial\xi_i=d\bar z_i/\bar z_i$ and hence $dd^c\xi_i=0$.
Since $u_t(p)$ is constant in space, we have
\begin{equation}\label{eq:sublevel-ddc}
 dd^c\psi_{p,t}
 =dd^c\Phi_t-\sum_i p_i\,dd^c\xi_i+dd^c\bigl(u_t(p)\bigr)
 =\omega_t.
\end{equation}
Both sides are smooth on $U_p$, so the equality extends from
the dense open orbit to all of $U_p$.

The definition of the Legendre transform, followed by continuity
when $p\in\partial\Delta$, gives
$u_t(p)\geq\langle p,\xi\rangle-\Phi_t(\xi)$. Thus, on the open orbit,
\begin{equation}\label{eq:sublevel-nonnegative}
 \psi_{p,t}(y)
 =\Phi_t(\xi(y))-\langle p,\xi(y)\rangle+u_t(p)\geq0.
\end{equation}
Continuity gives the same inequality on $U_p$.
Evaluating \eqref{eq:sublevel-boundary-formula} on $x_t^{-1}(p)$ gives
\begin{equation}\label{eq:calc-s4-004}
 \psi_{p,t}|_{x_t^{-1}(p)}
 =-\sum_{\ell_r(p)>0}\ell_r(p)\log\ell_r(p)-h_t(p)+u_t(p)=0.
\end{equation}
The last equality follows from $u_t=u_G+h_t$, with
$0\log0=0$ in $u_G$. The logarithmic terms with zero coefficient
have already been canceled, so the same calculation applies when
$p\in\partial\Delta$.
\end{proof}

\begin{proposition}\label{prop:sublevel-domains}
For every $p\in\Delta$, every $t$, and every $s>0$, the domain
$D_{p,s,t}$ is nonempty, connected, and relatively compact in $U_p$;
its boundary is smooth and strictly pseudoconvex, and its logarithmic
image is convex.
\end{proposition}

\begin{proof} 
Fix $p\in\Delta$, $t$, and $s>0$ throughout the proof.

\noindent\textbf{Nonempty.}
The fiber $x_t^{-1}(p)$ is nonempty, and
Proposition~\ref{prop:sublevel-geometry} gives $\psi_{p,t}=0$ on
this fiber. Hence $x_t^{-1}(p)\subset D_{p,s,t}$.

\noindent\textbf{Relatively compact in $U_p$.}
Since $M$ is compact, it suffices to show that the closure of
$D_{p,s,t}$ in $M$ does not meet $M\setminus U_p$.
The divisors retained in
$U_p$ may still meet the closure.
Let $y_j\in U_p$ converge to
$y_\infty\in M\setminus U_p$. By the definition of $U_p$, there
is an index $r$ such that $\ell_r(p)>0$ and
$y_\infty\in\mathcal D_r$. Since
$\mathcal D_r=x_t^{-1}(\mathcal F_r)$, we have
\begin{equation}\label{eq:calc-s4-005}
 \ell_r(x_t(y_j))\longrightarrow0^+,
 \qquad
 -\ell_r(p)\log\ell_r(x_t(y_j))\longrightarrow+\infty.
\end{equation}
Here $\ell_r(p)$ is a fixed positive coefficient; For every other
index $k$ with $\ell_k(p)>0$,
\begin{equation}\label{eq:calc-s4-006}
 -\ell_k(p)\log\ell_k(x_t(y_j))
 \geq-\ell_k(p)\log\!\left(\max_{x\in\Delta}\ell_k(x)\right)>-\infty.
\end{equation}
In \eqref{eq:sublevel-boundary-formula}, the terms with
$\ell_k(p)=0$ were omitted. Since $h_t\in C^\infty(\Delta)$ and
$\Delta$ is compact, the nonlogarithmic part satisfies
\begin{equation}\label{eq:sublevel-remainder-bound}
 \max_{x\in\Delta}
 \left|\left\langle x-p,\sum_r\nu_r+Dh_t(x)\right\rangle
       -h_t(x)+u_t(p)\right|\leq C_{p,t}.
\end{equation}
Thus no other term can cancel the divergence in
\eqref{eq:calc-s4-005}; hence we obtain
\begin{equation}\label{eq:sublevel-exhaustion}
 \psi_{p,t}(y_j)
 \geq-\ell_r(p)\log\ell_r(x_t(y_j))-C_{p,t}
 \longrightarrow+\infty.
\end{equation}
If the closure of $D_{p,s,t}$ met $M\setminus U_p$, there would
be such a sequence with $y_j\in D_{p,s,t}$, contradicting
$\psi_{p,t}(y_j)<s$. Hence $\overline{D_{p,s,t}}\Subset U_p$.

\noindent\textbf{Smooth boundary.}
We show that $d\psi_{p,t}\ne0$ on every positive level.
For $a\in\R^n$, consider the real one-parameter subgroup
\begin{equation}\label{eq:sublevel-real-torus-action}
 g_\tau:=(e^{\tau a_1/2},\ldots,e^{\tau a_n/2})\in T_{\C},
 \qquad \tau\in\R.
\end{equation} 
The action is defined smoothly on $M$ and preserves all toric
divisors, hence also $U_p$. Let $X_a$ denote its infinitesimal
generator. On the open orbit the action is coordinatewise
multiplication, so $\xi_i=\log|z_i|^2$ gives
\begin{equation}\label{eq:sublevel-log-translation}
 \xi(g_\tau\cdot y)=\xi(y)+\tau a,
 \qquad X_a=\sum_i a_i\frac{\partial}{\partial\xi_i}.
\end{equation}
Consequently, on the open orbit,
\begin{align}
 X_a\psi_{p,t}(y)
 &=\left.\frac{d}{d\tau}\right|_{\tau=0}
       \psi_{p,t}(g_\tau\cdot y)\notag\\
 &=\sum_i a_i\bigl(\partial_{\xi_i}\Phi_t(\xi(y))-p_i\bigr)
  =\langle D_\xi\Phi_t(\xi(y))-p,a\rangle
  =\langle x_t(y)-p,a\rangle.
 \label{eq:calc-s4-007}
\end{align}
Although the logarithmic coordinates are not defined on the
divisors, $X_a$, $\psi_{p,t}$, and $x_t$ are smooth on $U_p$.
Both sides of \eqref{eq:calc-s4-007} therefore extend continuously
from the dense open orbit, and the identity holds on all of $U_p$.

Let $y\in U_p$ satisfy $\psi_{p,t}(y)=s>0$.
By Proposition~\ref{prop:sublevel-geometry}, $x_t(y)\ne p$.
Choose the fixed vector $a:=x_t(y)-p$. Then
\begin{equation}\label{eq:sublevel-positive-level-derivative}
 d\psi_{p,t}|_y\bigl(X_a(y)\bigr)
 =\langle x_t(y)-p,a\rangle
 =|x_t(y)-p|^2>0.
\end{equation}
Thus $d\psi_{p,t}(y)\ne0$, including when $y$ lies on a retained
divisor. 

Relative compactness excludes boundary points outside $U_p$,
and continuity gives $\partial D_{p,s,t}\subset\{\psi_{p,t}=s\}$.
Conversely, at every point with $\psi_{p,t}=s$, the nonzero
differential and the implicit function theorem give nearby points
with values both smaller and larger than $s$. Such a point is
therefore on the boundary. The same theorem gives its smoothness,
and we have
\begin{equation}\label{eq:sublevel-regular-boundary}
 \partial D_{p,s,t}=\{y\in U_p:\psi_{p,t}(y)=s\},
 \qquad d\psi_{p,t}\ne0\quad\hbox{on }\partial D_{p,s,t}.
\end{equation}

\noindent\textbf{Strictly pseudoconvex boundary.} 
By \cite[Chapter~I, \S7.C, (7.10) and Definition~7.13]{DemaillyCADG},
a smooth boundary is strictly pseudoconvex if the Levi form of a
defining function is positive.
By \eqref{eq:sublevel-ddc},
\begin{equation}\label{eq:sublevel-levi-form}
 dd^c(\psi_{p,t}-s)=dd^c\psi_{p,t}=\omega_t>0.
\end{equation}
Its Levi form is therefore the positive Hermitian form of
$\omega_t$. 
Hence the boundary is strictly pseudoconvex.

\noindent\textbf{Convex logarithmic image.}
The logarithmic image is taken only on the open orbit and is
\begin{equation}\label{eq:sublevel-log-domain}
 \log D_{p,s,t}
 =\{\xi\in\R^n:\Phi_t(\xi)-\langle p,\xi\rangle+u_t(p)<s\}.
\end{equation}
The function inside this sublevel condition has Hessian
$D_\xi^2\Phi_t>0$. More explicitly, for
$\xi,\xi'\in\log D_{p,s,t}$ and $0\leq\lambda\leq1$, convexity gives
\begin{align}
 &\Phi_t((1-\lambda)\xi+\lambda\xi')
   -\langle p,(1-\lambda)\xi+\lambda\xi'\rangle+u_t(p)\notag\\
 &\quad\leq(1-\lambda)
       \bigl(\Phi_t(\xi)-\langle p,\xi\rangle+u_t(p)\bigr)
       +\lambda\bigl(\Phi_t(\xi')-\langle p,\xi'\rangle+u_t(p)\bigr)
       <s.
 \label{eq:sublevel-log-convexity}
\end{align}
Thus the line segment joining any two points of the logarithmic
image remains in that image.

\noindent\textbf{Connected.}
Notice that the set
\begin{equation}\label{eq:sublevel-polar-product}
 D_{p,s,t}\cap(\C^*)^n
 \simeq(\log D_{p,s,t})\times T^n
\end{equation}
is connected, since $\log D_{p,s,t}$ is convex and $T^n$ is connected.
As $D_{p,s,t}$ is open and the open orbit is dense in $M$, this
connected subset is dense in $D_{p,s,t}$, so $D_{p,s,t}$ is connected.
\end{proof}

For the uniform estimates below, denote by $\psi_p^0$ the potential
in \eqref{eq:fixed-moment-domains} associated with the fixed
background $\omega_0$; the superscript $0$ denotes the background,
not a value of the family parameter $t$. The definition gives
\begin{equation}\label{eq:sublevel-background-comparison}
 \psi_{p,t}=\psi_p^0+\varphi_t+u_t(p)-u_0(p).
\end{equation}
Here $u_t(p)-u_0(p)$ is constant in space.

\begin{proposition}
\label{prop:fixed-moment-domains}
Suppose $\sup_t\|u_t\|_\infty<\infty$.
For every fixed $p$ and every fixed $s_0>0$, the sets
$\overline{D_{p,s_0,t}}$ lie in a common compact subset of $U_p$.
If in addition $c\leq\omega_t^n/\omega_0^n\leq C$, then
\begin{equation}\label{eq:calc-s4-008}
 C_p^{-1}\leq\frac{\omega_t^n}{dV_{\mathrm{Eucl}}}\leq C_p
       \quad\hbox{on }D_{p,s_0,t},
\end{equation}
where $dV_{\mathrm{Eucl}}$ is Euclidean volume and $C_p\geq1$
is independent of $t$. The compact set and this constant may
depend on $p$ and $s_0$.
\end{proposition}

\begin{proof}
Fix $p$. Proposition~\ref{prop:sublevel-domains}, applied to the
fixed background, shows that every finite sublevel set of $\psi_p^0$
is relatively compact in $U_p$.
By \eqref{eq:sublevel-background-comparison} and \eqref{eq:legendre-C0},
\begin{equation}\label{eq:calc-s4-009}
 |\psi_{p,t}-\psi_p^0|\leq2\|u_t-u_0\|_\infty.
\end{equation}
The assumed bound therefore gives the common compact containment
\begin{equation}\label{eq:calc-s4-010}
 \overline{D_{p,s_0,t}}
 \subset\left\{y\in U_p:\psi_p^0(y)
          \leq s_0+2\sup_t\|u_t-u_0\|_\infty\right\}\Subset U_p.
\end{equation}

Finally, on the common compact subset of $U_p$,
the background volume is smoothly comparable to Euclidean
volume. The two-sided volume-ratio bounds give the stated estimate.
\end{proof}

\subsubsection{Harmonic oscillation}
We now prove the two local estimates used to control harmonic
oscillation on these domains. In the local estimates through
Proposition~\ref{prop:potential-oscillation}, $z$ denotes the standard
coordinates on $\C^n$, and we write
$\xi_i:=\log|z_i|^2$ and $\theta_i:=\arg z_i$. Euclidean volume in
these coordinates is
\begin{equation}\label{eq:calc-s4-011}
 dV_{\mathrm{Eucl}}=\frac1{2^n}e^{\xi_1+\cdots+\xi_n}\,d\xi\,d\theta.
\end{equation}
Integrating over the angular variables, each of period $2\pi$,
gives $\pi^n e^{\xi_1+\cdots+\xi_n}\,d\xi$.

\begin{proposition}\label{prop:relative-green}
Let \(D_s\Subset\mathbb C^n\) be a bounded connected Reinhardt domain
with smooth boundary and convex logarithmic image.
Suppose that \(\psi\) is smooth and \(T^n\)-invariant on a neighborhood of
\(\overline{D_s}\), that \(\omega:=dd^c\psi>0\), and that
\begin{equation}\label{eq:calc-s4-012}
       0\leq\psi<s\quad\hbox{in }D_s,\qquad
       \psi=s,\quad d\psi\ne0\quad\hbox{on }\partial D_s.
\end{equation}
For these local estimates, write \(d\mu:=\omega^n\) without
normalizing its mass, and put \(V_s:=\mu(D_s)\). Suppose that
there are constants \(0<b_-\leq b_+\) such that
\begin{equation}\label{eq:local-density-ratio}
 b_-\leq \frac{d\mu}{dV_{\mathrm{Eucl}}}\leq b_+,\qquad
                    \frac{b_+}{b_-}\leq\kappa.
\end{equation}
For a $T^n$-invariant
Borel set \(E\subset D_s\), denote its relative volume by
\begin{equation}\label{eq:calc-s4-013}
 m_E:=\frac{\mu(E)}{\mu(D_s)}\in[0,1].
\end{equation}
Let \(\mathbf1_E\) be the indicator function of \(E\), and let
\(G_{D_s}\) be the positive Dirichlet inverse of \(-\Delta_\omega\).
Then
\begin{equation}\label{eq:relative-green}
 \bigl\|G_{D_s}\mathbf1_E\bigr\|_{L^\infty(D_s)}
       \leq C(n,\kappa)\,s\,m_E^{1/n}\Bigl(1+\log\frac1{m_E}\Bigr).
\end{equation}
The right-hand side is interpreted as zero when \(m_E=0\).
\end{proposition}

We first prove an exponential integrability estimate that will
be used in the capacity argument. For bounded plurisubharmonic
functions, Monge--Amp\`ere measures are understood in the
Bedford--Taylor sense \cite{BedfordTaylor1982}; they agree with
the usual wedge products for smooth functions. The total mass
on a domain is the integral of the measure over that domain.

\begin{lemma}\label{lem:toric-exponential-integrability}
Under the hypotheses and notation of
Proposition~\ref{prop:relative-green}, let $\chi$ be a bounded
nonpositive $T^n$-invariant plurisubharmonic function on $D_s$ such that
\begin{equation}\label{eq:calc-s4-014}
 \lim_{y\to\partial D_s}\chi(y)=0,\qquad
 \int_{D_s}(dd^c\chi)^n\leq1.
\end{equation}
There is a constant \(C\), depending only on \(\kappa\), such that
\begin{equation}\label{eq:toric-exponential-integrability}
 \int_{D_s}\exp\!\left(-\frac{V_s^{1/n}}{Cs}\,\chi\right)
              \,d\mu\leq CV_s.
\end{equation}
\end{lemma}

\begin{proof} 
The case $\chi\equiv0$ is immediate. We therefore assume
$\chi\not\equiv0$. The maximum principle gives $\chi<0$ in $D_s$.
Both $\chi$ and $\psi-s$ are bounded plurisubharmonic functions
with zero boundary values and finite Monge--Amp\`ere mass.
Thus \cite[Proposition~2.1(a),(d), equation~(2.5)]{Nguyen2018}
provides integration by parts for this pair. Since $dd^c(\psi-s)=\omega$ and $0\leq s-\psi\leq s$,
we obtain
\begin{align}
 \int_{D_s}(-\chi)\,d\mu
 &=\int_{D_s}(s-\psi)\,dd^c\chi\wedge\omega^{n-1}\notag\\
 &\leq s\int_{D_s}dd^c\chi\wedge\omega^{n-1}\notag\\
 &\leq s\left(\int_{D_s}(dd^c\chi)^n\right)^{1/n}
                  V_s^{(n-1)/n}
 \leq sV_s^{(n-1)/n}.
 \label{eq:capacity-mean}
\end{align}
For $t\geq0$, write
\begin{equation}\label{eq:calc-s4-015}
 S(t):=\frac{|\{\chi<-st/V_s^{1/n}\}|}{|D_s|},
\end{equation}
where $|\cdot|$ denotes Euclidean volume. In particular, $S(0)=1$.
The density bounds \eqref{eq:local-density-ratio} therefore imply
\begin{equation}\label{eq:calc-s4-016}
 \frac1{|D_s|}\int_{D_s}(-\chi)\,dV_{\mathrm{Eucl}}
 \leq\frac{\kappa}{V_s}\int_{D_s}(-\chi)\,d\mu
 \leq\frac{\kappa s}{V_s^{1/n}},
\end{equation}
and Markov's inequality gives
\begin{equation}\label{eq:calc-s4-017}
 S(2\kappa)\leq\frac{V_s^{1/n}}{2\kappa s|D_s|}
                  \int_{D_s}(-\chi)\,dV_{\mathrm{Eucl}}
             \leq\frac12.
\end{equation}
Next, we show that $-\log S(t)$ is convex. Define a nonnegative function on
$\mathbb R\times\mathbb R^n$ by
\begin{equation}\label{eq:prekopa-joint-density}
 f_{\mathrm{Pr}}(t,\xi):=
 \begin{cases}
 \displaystyle\frac{\pi^n}{|D_s|}e^{\xi_1+\cdots+\xi_n},
 &\xi\in\log D_s,\quad t\geq0,\quad
       \chi(\xi)+st/V_s^{1/n}<0,\\[4pt]
 0,&\text{otherwise}.
 \end{cases}
\end{equation}
Its positivity set is convex. Indeed, if
$f_{\mathrm{Pr}}(t_1,\xi^{(1)})>0$ and $f_{\mathrm{Pr}}(t_2,\xi^{(2)})>0$, then for
$0<\lambda<1$ the convexity of $\chi$ gives
\begin{align}
 &\chi\bigl(\lambda\xi^{(1)}+(1-\lambda)\xi^{(2)}\bigr)
       +\frac{s\bigl(\lambda t_1+(1-\lambda)t_2\bigr)}{V_s^{1/n}}
       \notag\\*
 &\qquad\leq
 \lambda\left(\chi(\xi^{(1)})+\frac{st_1}{V_s^{1/n}}\right)
 +(1-\lambda)\left(\chi(\xi^{(2)})+\frac{st_2}{V_s^{1/n}}\right)<0.
 \label{eq:calc-s4-018}
\end{align}
The logarithmic coordinate remains in $\log D_s$, and the
parameter remains nonnegative. On this positivity set,
$\log f_{\mathrm{Pr}}=\log(\pi^n/|D_s|)+\xi_1+\cdots+\xi_n$ is affine.
Consequently,
\begin{equation}\label{eq:prekopa-joint-log-concavity}
 f_{\mathrm{Pr}}\bigl(\lambda t_1+(1-\lambda)t_2,
        \lambda\xi^{(1)}+(1-\lambda)\xi^{(2)}\bigr)
 \geq f_{\mathrm{Pr}}(t_1,\xi^{(1)})^\lambda
       f_{\mathrm{Pr}}(t_2,\xi^{(2)})^{1-\lambda}.
\end{equation}
When both factors on the right are positive this is an equality;
when either is zero the inequality follows from $f_{\mathrm{Pr}}\geq0$.
Thus $f_{\mathrm{Pr}}$ is jointly log-concave on the whole space, including
where it vanishes.

After angular integration, Euclidean volume has density
$\pi^n e^{\xi_1+\cdots+\xi_n}\,d\xi$. Hence
\begin{align}
 \int_{\mathbb R^n}f_{\mathrm{Pr}}(t,\xi)\,d\xi
 &=\frac{\pi^n}{|D_s|}
   \int_{\{\xi\in\log D_s:\,\chi(\xi)<-st/V_s^{1/n}\}}
                   e^{\xi_1+\cdots+\xi_n}\,d\xi\notag\\
 &=\frac{|\{\chi<-st/V_s^{1/n}\}|}{|D_s|}
 =S(t),\qquad t\geq0.
 \label{eq:prekopa-marginal-identity}
\end{align}
Then Pr\'ekopa's theorem \cite[Theorem~6]{Prekopa1973} shows that $-\log S$ is convex
on the interval where $S>0$. When $t\geq2\kappa$ and $S(t)>0$, monotonicity gives
$S(2\kappa)\geq S(t)>0$. Since $-\log S(0)=0$ and
$2\kappa=(2\kappa/t)t+(1-2\kappa/t)0$, convexity gives
\begin{align}
 -\log S(2\kappa)
 &\leq\frac{2\kappa}{t}\bigl(-\log S(t)\bigr)
       +\left(1-\frac{2\kappa}{t}\right)\bigl(-\log S(0)\bigr)
       \notag\\*
 &=\frac{2\kappa}{t}\bigl(-\log S(t)\bigr).
 \label{eq:calc-s4-019}
\end{align}
Combining this with \eqref{eq:calc-s4-017} yields:
\begin{equation}\label{eq:calc-s4-020}
 S(t)\leq S(2\kappa)^{t/(2\kappa)}
       \leq2^{-t/(2\kappa)}=e^{-2t/C},\qquad t\geq2\kappa.
\end{equation}
Choose $C:=4\kappa/\log2$. We now estimate the left-hand side of
\eqref{eq:toric-exponential-integrability}:
\begin{align}
\int_{D_s}e^{-V_s^{1/n}\chi/(Cs)}\,d\mu
 &\leq b_+\int_{D_s}e^{-V_s^{1/n}\chi/(Cs)}\,dV_{\mathrm{Eucl}}\notag\\
 &=b_+\int_{D_s}\left(1+\frac1C
       \int_0^{-V_s^{1/n}\chi(y)/s}e^{t/C}\,dt\right)
                           dV_{\mathrm{Eucl}}(y)\notag\\
 &=b_+|D_s|+\frac{b_+}{C}\int_0^\infty e^{t/C}
       \left|\left\{y\in D_s:\chi(y)<-\frac{st}{V_s^{1/n}}\right\}\right|
       \,dt\notag\\
 &=b_+|D_s|\left(1+\frac1C\int_0^\infty e^{t/C}S(t)\,dt\right)\notag\\
 &\leq\kappa V_s\left(1+\frac1C\int_0^{2\kappa}e^{t/C}S(t)\,dt
             +\frac1C\int_{2\kappa}^\infty e^{t/C}S(t)\,dt\right)\notag\\
 &\leq\kappa V_s\left(1+\frac1C\int_0^{2\kappa}e^{t/C}\,dt
             +\frac1C\int_{2\kappa}^\infty e^{-t/C}\,dt\right)\notag\\
 &=\kappa V_s\left(e^{2\kappa/C}+e^{-2\kappa/C}\right)
 \leq3\kappa V_s\leq CV_s.
 \label{eq:calc-s4-021}
\end{align}
For the last integral inequality, we use $S(t)\leq1$ on
$[0,2\kappa]$ and \eqref{eq:calc-s4-020} on $[2\kappa,\infty)$.
\end{proof}

\begin{proof}[Proof of Proposition~\ref{prop:relative-green}]
We may assume $m_E>0$, since the Green potential vanishes when $m_E=0$.
Consider the auxiliary problem
\begin{equation}\label{eq:capacity-auxiliary-equation}
 (dd^cv_E)^n=\mathbf1_E\,d\mu,\qquad
 v_E|_{\partial D_s}=0.
\end{equation}
The function $\psi-s$ is a smooth strictly plurisubharmonic
defining function for $D_s$ and a zero-boundary subsolution of
finite mass. Thus \cite[Theorem~1.2]{Nguyen2018} gives a continuous
plurisubharmonic solution. The Bedford--Taylor comparison principle
\cite[Theorem~4.1 and Corollary~4.4]{BedfordTaylor1982} gives uniqueness
and $\psi-s\leq v_E\leq0$. Since the domain and the data are
$T^n$-invariant, it also implies that $v_E$ is $T^n$-invariant.

The mixed Monge--Amp\`ere inequality of Dinew
\cite[Theorem~1.3]{Dinew2009} yields
\begin{equation}\label{eq:calc-s4-022}
 dd^cv_E\wedge\omega^{n-1}\geq\mathbf1_E\,\omega^n.
\end{equation}
Since the potentials are locally bounded and $\mu$ is smooth,
this holds in the sense of measures. Hence
$\Delta_\omega v_E\geq n\mathbf1_E$ in distributions, and
\begin{equation}\label{eq:calc-s4-023}
 \Delta_\omega\left(G_{D_s}\mathbf1_E+\frac{v_E}{n}\right)\geq0.
\end{equation}
The function in parentheses is continuous with zero boundary values.
The maximum principle therefore gives
\begin{equation}\label{eq:calc-s4-024}
 0\leq G_{D_s}\mathbf1_E\leq-\frac{v_E}{n}.
\end{equation}
It remains to estimate $\|v_E\|_\infty$ in terms of $m_E$.
For this purpose, we first derive a volume--capacity estimate from
Lemma~\ref{lem:toric-exponential-integrability}.

Let $\PSH(D_s)$ denote the plurisubharmonic functions on $D_s$.
For a Borel set $A\subset D_s$, its Bedford--Taylor relative
capacity is
\begin{equation}\label{eq:calc-s4-025}
 \operatorname{Cap}(A,D_s)
 :=\sup_{\substack{a\in\PSH(D_s)\\-1\leq a\leq0}}
                       \int_A(dd^ca)^n.
\end{equation}
For a $T^n$-invariant compact set $K\Subset D_s$, let
\begin{equation}\label{eq:calc-s4-026}
 h_K^*:=\left(\sup\{a\in\PSH(D_s):a\leq0\text{ on }D_s,
                              \ a\leq-1\text{ on }K\}\right)^*,
\end{equation}
where the star denotes upper semicontinuous regularization.
The domain and $K$ are $T^n$-invariant, so the defining envelope
and its upper semicontinuous regularization are $T^n$-invariant.
A set is pluripolar if, locally, it is contained in the
$-\infty$ locus of a plurisubharmonic function that is not
identically $-\infty$.
By \cite[Proposition~5.3, Proposition~6.2(iv), and Theorem~7.1]
{BedfordTaylor1982}, the function $h_K^*$ lies between $-1$ and $0$,
has zero boundary values, equals
$-1$ on $K$ except on a pluripolar set, and satisfies
\begin{equation}\label{eq:calc-s4-027}
 \int_{D_s}(dd^ch_K^*)^n
 =\int_K(dd^ch_K^*)^n=\operatorname{Cap}(K,D_s).
\end{equation}
The smooth measure $\mu$ charges no pluripolar sets, so
$h_K^*=-1$ almost everywhere on $K$ with respect to $\mu$.
If $\operatorname{Cap}(K,D_s)>0$, then
\begin{equation}\label{eq:calc-s4-028}
 \int_{D_s}\left(dd^c\frac{h_K^*}
 {\operatorname{Cap}(K,D_s)^{1/n}}\right)^n
 =\frac{\displaystyle\int_{D_s}(dd^ch_K^*)^n}
 {\operatorname{Cap}(K,D_s)}=1.
\end{equation}
Thus Lemma~\ref{lem:toric-exponential-integrability} applies to
$h_K^*/\operatorname{Cap}(K,D_s)^{1/n}$. Using its value on $K$,
enlarging the domain of integration, and then applying the lemma,
we obtain
\begin{align}
\frac{\mu(K)}{V_s}
 &=\frac1{V_s}\int_K1\,d\mu \notag\\
 &=\frac1{V_s}\exp\!\left(
       -\frac{V_s^{1/n}}{Cs\operatorname{Cap}(K,D_s)^{1/n}}\right)
   \int_K\exp\!\left(
       -\frac{V_s^{1/n}h_K^*}{Cs\operatorname{Cap}(K,D_s)^{1/n}}\right)d\mu \notag\\
 &\leq\frac1{V_s}\exp\!\left(
       -\frac{V_s^{1/n}}{Cs\operatorname{Cap}(K,D_s)^{1/n}}\right)
   \int_{D_s}\exp\!\left(
       -\frac{V_s^{1/n}h_K^*}{Cs\operatorname{Cap}(K,D_s)^{1/n}}\right)d\mu \notag\\
 &\leq C\exp\!\left(
       -\frac{V_s^{1/n}}{Cs\operatorname{Cap}(K,D_s)^{1/n}}\right).
 \label{eq:toric-volume-capacity}
\end{align}
Here \(C\) is the fixed constant chosen in the lemma.
If the capacity is zero, $K$ is pluripolar and has zero measure;
see \cite[Proposition~6.4(iii) and Theorem~6.9]{BedfordTaylor1982}.
The same estimate holds for $T^n$-invariant relatively compact open
sets, by exhaustion with $T^n$-invariant compact subsets and continuity
from below of capacity \cite[Proposition~3.2(4)]{BedfordTaylor1982}.

We now apply this estimate to the sublevel sets of $v_E$.
For $t>0$, put
\begin{equation}\label{eq:calc-s4-029}
 A_t:=\{v_E<-st\},\qquad
 g(t):=\frac{s}{V_s^{1/n}}\operatorname{Cap}(A_t,D_s)^{1/n}.
\end{equation}
The sets $A_t$ are $T^n$-invariant and relatively compact, since $v_E$
is continuous with zero boundary values. The parameter $t$
measures sublevel depth, and $g$ is nonincreasing.
For $r,t>0$ and any plurisubharmonic $a$ with $-1\leq a\leq0$,
\begin{equation}\label{eq:calc-s4-030}
 A_{t+r}\subset\{v_E<sr\,a-st\}\subset A_t.
\end{equation}
Since $v_E$ tends to zero at the boundary and
$sr\,a-st\leq-st<0$ throughout $D_s$, the boundary condition
of the Bedford--Taylor comparison principle
\cite[Theorem~4.1]{BedfordTaylor1982} is satisfied.
Applying it to $v_E$ and $sr\,a-st$ gives
\begin{equation}\label{eq:calc-s4-031}
 (sr)^n\int_{A_{t+r}}(dd^ca)^n
 \leq\int_{\{v_E<sr\,a-st\}}(dd^cv_E)^n
 \leq\int_{A_t}(dd^cv_E)^n.
\end{equation}
Taking the supremum over $a$ and dividing by $s^n$, we obtain
\begin{equation}\label{eq:capacity-comparison-step}
 r^n\operatorname{Cap}(A_{t+r},D_s)
 \leq s^{-n}\int_{A_t}(dd^cv_E)^n
 =s^{-n}\mu(E\cap A_t).
\end{equation}
In terms of $g$, this reads
\begin{equation}\label{eq:calc-s4-032}
 r^ng(t+r)^n\leq\frac{\mu(E\cap A_t)}{V_s}.
\end{equation}
The right-hand side is at most $m_E$. It is also bounded by
$\mu(A_t)/V_s$, to which \eqref{eq:toric-volume-capacity}
applies. Taking $n$th roots, we obtain
\begin{equation}\label{eq:relative-capacity-iteration}
 r\,g(t+r)\leq
       \min\{m_E^{1/n},\,C_1e^{-c_1/g(t)}\},
\end{equation}
with fixed constants $C_1,c_1>0$ depending only on $n,\kappa$.
The exponential term is interpreted as zero if $g(t)=0$.
Lemma~\ref{lem:real-variable-iteration} below, applied to
\eqref{eq:relative-capacity-iteration} with $m=m_E$, gives
\begin{equation}\label{eq:calc-s4-033}
 g(t)=0\qquad\text{for }\quad
 t\geq C m_E^{1/n}\Bigl(1+\log\frac1{m_E}\Bigr),
\end{equation}
where $C$ depends only on $n,\kappa$.
The corresponding sets $A_t$ have zero capacity. Since they are
open, they are empty, and hence $v_E\geq-st$ at these depths.
Together with $v_E\leq0$, this yields
\begin{equation}\label{eq:auxiliary-relative-sup}
 \|v_E\|_\infty
 \leq Cs\,m_E^{1/n}\Bigl(1+\log\frac1{m_E}\Bigr).
\end{equation}
The Green comparison at the beginning of the proof now gives
\eqref{eq:relative-green}.
\end{proof}

\begin{lemma}\label{lem:real-variable-iteration}
Fix an integer $n\geq1$, a number $0<m\leq1$, and constants
$C_1,c_1>0$. Suppose that $g:(0,\infty)\to[0,\infty)$ satisfies
\begin{equation}\label{eq:calc-s4-034}
 r\,g(t+r)\leq
 \min\{m^{1/n},\,C_1e^{-c_1/g(t)}\},
 \qquad t,r>0,
\end{equation}
where the exponential term is defined to be zero when $g(t)=0$.
Then there is a constant $C=C(n,C_1,c_1)>0$ such that
\begin{equation}\label{eq:calc-s4-035}
 g(t)=0\qquad\text{for }\quad
 t\geq C m^{1/n}\Bigl(1+\log\frac1m\Bigr).
\end{equation}
\end{lemma}

\begin{proof}
For fixed $t>0$, use the hypothesis with base point
$\varepsilon\in(0,t)$ and increment $t-\varepsilon$ to obtain
\begin{equation}\label{eq:calc-s4-036}
 (t-\varepsilon)g(t)\leq m^{1/n}.
\end{equation}
Letting $\varepsilon\downarrow0$ gives the starting estimate
\begin{equation}\label{eq:calc-s4-037}
 g(t)\leq\frac{m^{1/n}}t,\qquad t>0.
\end{equation}
Put $L:=1+\log(1/m)\geq1$, and choose $M_0\geq1$, depending only
on $n,C_1,c_1$, such that
\begin{equation}\label{eq:calc-s4-038}
 c_1M_0\geq n^{-1}+\log(4\max\{C_1,1\}).
\end{equation}
At
\begin{equation}\label{eq:calc-s4-039}
 t_0:=M_0m^{1/n}L,
\end{equation}
the starting estimate gives $g(t_0)\leq1/(M_0L)$.
This choice makes the exponential term at most
$C_1\exp(-c_1M_0L)$.
We use this term to halve the upper bound for $g$ successively,
while keeping the total increment below $t_0$.

Set
\begin{equation}\label{eq:calc-s4-040}
 t_j:=t_0(2-2^{-j}),\qquad
 t_{j+1}-t_j=t_0\,2^{-j-1},\qquad j\geq0.
\end{equation}
We prove by induction that
\begin{equation}\label{eq:calc-s4-041}
 g(t_j)\leq\frac{2^{-j}}{M_0L}.
\end{equation}
The case $j=0$ was just established. Suppose the bound holds
at $j$. If $g(t_j)=0$, the hypothesis directly gives
$g(t_j+r)=0$ for every $r>0$, so the next bound holds.
Otherwise, $1/g(t_j)\geq M_0L2^j$. Applying the hypothesis
with base point $t_j$ and increment $t_{j+1}-t_j$, we obtain
\begin{equation}\label{eq:calc-s4-042}
 (t_{j+1}-t_j)g(t_{j+1})
 \leq C_1e^{-c_1M_0L2^j}.
\end{equation}
By the choice of $M_0$, together with $L\geq1$ and
$2^j\geq j+1$,
\begin{equation}\label{eq:calc-s4-043}
 \begin{aligned}
 c_1M_0L2^j
 &\geq\frac{L2^j}{n}+L2^j\log4
                    +L2^j\log\max\{C_1,1\}\\
 &\geq\frac{L-1}{n}+(j+1)\log4+\log C_1.
 \end{aligned}
\end{equation}
Since $L-1=\log(1/m)$, it follows that
\begin{equation}\label{eq:calc-s4-044}
 C_1e^{-c_1M_0L2^j}\leq4^{-j-1}m^{1/n}.
\end{equation}
Dividing by the increment therefore yields
\begin{equation}\label{eq:calc-s4-045}
 g(t_{j+1})
 \leq\frac{4^{-j-1}m^{1/n}}
 {M_0m^{1/n}L\,2^{-j-1}}
 =\frac{2^{-j-1}}{M_0L},
\end{equation}
which completes the induction.

The points $t_j$ increase to $2t_0$. We apply the hypothesis
once more, now with increment $2t_0-t_j=t_0\,2^{-j}$.
The preceding exponential bound gives
\begin{equation}\label{eq:calc-s4-046}
 (2t_0-t_j)g(2t_0)
 \leq C_1e^{-c_1/g(t_j)}
 \leq m^{1/n}4^{-j-1},
\end{equation}
where the same inequality holds if $g(t_j)=0$ under the stated
convention. Consequently,
\begin{equation}\label{eq:calc-s4-047}
 0\leq g(2t_0)
 \leq\frac{m^{1/n}4^{-j-1}}{t_0\,2^{-j}}
 =\frac{2^{-j-2}}{M_0L}\longrightarrow0.
\end{equation}
Thus $g(2t_0)=0$. Taking base point $2t_0$ in the hypothesis
then gives $g(t)=0$ for every $t>2t_0$ as well.
This proves the assertion with $C:=2M_0$, without any continuity
or monotonicity assumption on $g$.
\end{proof} 

We now establish the second key estimate for controlling
the oscillation of harmonic functions.

\begin{proposition}
\label{prop:weighted-crossover}
Under the hypotheses of Proposition~\ref{prop:relative-green}, put
$h:=s-\psi$ and define the weighted probability measure
\begin{equation}
 d\nu_s:=\frac{h\,d\mu}{\int_{D_s}h\,d\mu}.
 \label{eq:torsion-probability}
\end{equation}
There are $p_0\in(0,1)$ and $C_*>0$, depending only on $n$ and
$\kappa$, such that every positive $T^n$-invariant $\omega$-harmonic
function $v$, smooth on a neighborhood of $\overline{D_s}$, satisfies
\begin{equation}
 \left(\int_{D_s}v^{p_0}\,d\nu_s\right)
 \left(\int_{D_s}v^{-p_0}\,d\nu_s\right)\leq C_*.
 \label{eq:weighted-crossover}
\end{equation}
\end{proposition}

\begin{proof}
We first establish a weighted Poincar\'e inequality. Equip the
convex logarithmic domain $\log D_s\subset\R^n$ with the auxiliary
flat metric
\begin{equation}\label{eq:calc-s4-048}
 g_{\mathrm{flat}}:=\sum_{i=1}^n d\xi_i^2.
\end{equation}
The reference measure obtained from $h\,dV_{\mathrm{Eucl}}$
by angular integration is
$\pi^n h e^{\xi_1+\cdots+\xi_n}\,d\xi$. Since
$g_{\mathrm{flat}}$ is flat and the sum of the $\xi_i$ is linear,
its weighted Ricci tensor is
\begin{equation}\label{eq:calc-s4-049}
 \begin{aligned}
 &\operatorname{Ric}(g_{\mathrm{flat}})
 +\operatorname{Hess}_{g_{\mathrm{flat}}}
       (-\log h-\xi_1-\cdots-\xi_n)\\
 &\qquad=D_\xi^2(-\log h)
 =\frac{D_\xi^2\psi}{h}
       +\frac{D_\xi\psi\otimes D_\xi\psi}{h^2}
 \geq\frac{D_\xi^2\psi}{h}>0.
 \end{aligned}
\end{equation}
Here $D_\xi^2\psi>0$ follows from the $T^n$-invariance of $\psi$
and $dd^c\psi>0$. Thus the weighted Ricci tensor is positive
definite for this auxiliary Euclidean metric.
For a smooth $T^n$-invariant function $a$ on a neighborhood of
$\overline{D_s}$, let
\begin{equation}\label{eq:reference-weighted-mean}
 c:=\frac{\int_{D_s}ah\,dV_{\mathrm{Eucl}}}
         {\int_{D_s}h\,dV_{\mathrm{Eucl}}}.
\end{equation}
Using the convexity of the logarithmic domain, we apply
\cite[Theorem~1.2(1), $N=\infty$]{KolesnikovMilman2017}
on smooth convex subdomains compactly contained in $\log D_s$
and pass to the limit as justified below. The cited inequality
imposes no boundary condition on $a$ and gives
\begin{align}
 \int_{D_s}(a-c)^2h\,dV_{\mathrm{Eucl}}
 &=\pi^n\int_{\log D_s}(a-c)^2h
                  e^{\xi_1+\cdots+\xi_n}\,d\xi\notag\\
 &\leq\pi^n\int_{\log D_s}
       (D_\xi a)^{\mathsf T}[D_\xi^2(-\log h)]^{-1}D_\xi a\,
       h e^{\xi_1+\cdots+\xi_n}\,d\xi\notag\\
 &\leq\pi^n\int_{\log D_s}
       h^2(D_\xi a)^{\mathsf T}(D_\xi^2\psi)^{-1}D_\xi a\,
       e^{\xi_1+\cdots+\xi_n}\,d\xi\notag\\
 &=\pi^n\int_{\log D_s}
       h^2|\nabla_\omega a|_\omega^2
       e^{\xi_1+\cdots+\xi_n}\,d\xi\notag\\
 &=\int_{D_s}h^2|\nabla_\omega a|_\omega^2
                         \,dV_{\mathrm{Eucl}}.
 \label{eq:calc-s4-052}
\end{align}
The second inequality follows by taking inverses in
\eqref{eq:calc-s4-049}.

For the limiting argument, exhaust $\log D_s$, which may be
unbounded, by increasing bounded smooth convex domains whose
closures lie in its interior. On each such domain the reference
density is smooth and positive up to the boundary, and the
inverse-tensor comparison above has constant one, so no uniform
positive lower bound for the weighted Ricci tensor is needed.
Since $D_s$ is bounded, $0<h\leq s$, and $a$ is bounded, the
reference measure is finite and the weighted means and second
moments converge; the nonnegative energy integrals converge by
monotone convergence. This justifies the limit as in
\cite[Remark~1.3]{KolesnikovMilman2017}.

We now use the density bounds \eqref{eq:local-density-ratio}
to pass to $d\mu$. The mean minimizes the integral of the
squared difference from a constant. Choosing the reference mean
$c$ above, and then comparing the densities in the two integrals,
gives
\begin{align}
&\int_{D_s}a^2\,d\nu_s
       -\left(\int_{D_s}a\,d\nu_s\right)^2 \notag\\
 &\quad\leq
 \frac{b_+\int_{D_s}(a-c)^2h\,dV_{\mathrm{Eucl}}}
      {\int_{D_s}h\,d\mu} \notag\\
 &\quad\leq
 \frac{b_+\int_{D_s}h^2|\nabla_\omega a|_\omega^2
                       \,dV_{\mathrm{Eucl}}}
      {\int_{D_s}h\,d\mu} \notag\\
 &\quad\leq\frac{b_+}{b_-}
       \int_{D_s}h|\nabla_\omega a|_\omega^2\,d\nu_s
 \leq\kappa\int_{D_s}h|\nabla_\omega a|_\omega^2\,d\nu_s.
 \label{eq:weighted-BL}
\end{align}

Now suppose that $v>0$ on
$\overline{D_s}$, fix $0<|p|<1$, and write $a:=v^{p/2}$.
Testing $\Delta_\omega v=0$ against $h^2v^{p-1}$, and using
$\nabla_\omega a=(p/2)v^{p/2-1}\nabla_\omega v$, gives
\begin{equation}\label{eq:calc-s4-053}
 0=\frac4p\int_{D_s}ha
           \langle\nabla_\omega h,\nabla_\omega a\rangle_\omega\,d\mu
   +\frac{4(p-1)}{p^2}\int_{D_s}h^2
                              |\nabla_\omega a|_\omega^2\,d\mu.
\end{equation}
Consequently,
\begin{align}
 \frac1{|p|}\int_{D_s}h^2|\nabla_\omega a|_\omega^2\,d\mu
 &=\left|\int_{D_s}h^2|\nabla_\omega a|_\omega^2\,d\mu
   +\int_{D_s}ha\langle\nabla_\omega a,
                         \nabla_\omega h\rangle_\omega\,d\mu\right|\notag\\
 &=\left|\int_{D_s}h\langle\nabla_\omega a,
                         \nabla_\omega(ha)\rangle_\omega\,d\mu\right|\notag\\
 &\leq\left(\int_{D_s}h^2|\nabla_\omega a|_\omega^2\,d\mu\right)^{1/2}
       \left(\int_{D_s}|\nabla_\omega(ha)|_\omega^2\,d\mu\right)^{1/2}.
 \label{eq:calc-s4-054}
\end{align}
Thus, 
\begin{equation}\label{eq:calc-s4-056}
 \frac1{p^2}\int_{D_s}h^2|\nabla_\omega a|_\omega^2\,d\mu
 \leq\int_{D_s}|\nabla_\omega(ha)|_\omega^2\,d\mu.
\end{equation}
This argument applies to negative as well as positive $p$.
Expanding the last integrand and using $\Delta_\omega h=-n$ gives
\begin{equation}\label{eq:calc-s4-057}
 \begin{aligned}
 \int_{D_s}|\nabla_\omega(ha)|_\omega^2\,d\mu
 &=\int_{D_s}h^2|\nabla_\omega a|_\omega^2\,d\mu
   +\int_{D_s}\langle\nabla_\omega h,
                         \nabla_\omega(ha^2)\rangle_\omega\,d\mu\\
 &=\int_{D_s}h^2|\nabla_\omega a|_\omega^2\,d\mu
   +n\int_{D_s}ha^2\,d\mu.
 \end{aligned}
\end{equation}
The boundary terms in both integrations by parts vanish because
$h^2v^{p-1}=ha^2=0$ on $\partial D_s$. These integrations are
performed on the smooth domain $D_s\subset\C^n$, so coordinate
hyperplanes introduce no boundary terms.
Since $1/p^2-1>0$, rearranging and dividing by
$\int_{D_s}h\,d\mu$ proves
\begin{equation}
 \int_{D_s}h|\nabla_\omega(v^{p/2})|_\omega^2\,d\nu_s
 \leq\frac{np^2}{1-p^2}\int_{D_s}v^p\,d\nu_s.
 \label{eq:signed-small-power-energy}
\end{equation}
If the gradient integral vanishes, the same estimate is immediate.

Combining \eqref{eq:weighted-BL} with
\eqref{eq:signed-small-power-energy}, we obtain
\begin{equation}\label{eq:calc-s4-058}
 \left(1-\frac{\kappa np^2}{1-p^2}\right)
       \int_{D_s}v^p\,d\nu_s
 \leq\left(\int_{D_s}v^{p/2}\,d\nu_s\right)^2.
\end{equation}
Take $p_0:=(2\sqrt{\kappa n+1})^{-1}$. For $0<|p|\leq p_0$,
\begin{equation}\label{eq:calc-s4-059}
 0\leq\frac{\kappa np^2}{1-p^2}\leq\frac12,
 \qquad
 -\log\left(1-\frac{\kappa np^2}{1-p^2}\right)
 \leq\frac{2\kappa np^2}{1-p^2}\leq Cp^2,
\end{equation}
where $C$ depends only on $n,\kappa$.
Then we have
\begin{equation}\label{eq:calc-s4-060}
 \log\int_{D_s}v^p\,d\nu_s
 \leq2\log\int_{D_s}v^{p/2}\,d\nu_s+Cp^2.
\end{equation}
Iterating \eqref{eq:calc-s4-060}, the substitution at exponent
$p/2^j$, $j\geq0$, contributes the error
\begin{equation}\label{eq:calc-s4-062}
 2^j C(p/2^j)^2=Cp^2\,2^{-j}.
\end{equation}
Thus, for every integer $k\geq1$,
\begin{equation}\label{eq:calc-s4-063}
 \begin{aligned}
 \log\int_{D_s}v^p\,d\nu_s
 &\leq2^k\log\int_{D_s}v^{p/2^k}\,d\nu_s
       +Cp^2\sum_{j=0}^{k-1}2^{-j}\\
 &\leq2^k\log\int_{D_s}v^{p/2^k}\,d\nu_s+2Cp^2.
 \end{aligned}
\end{equation}

For this fixed $v>0$ on $\overline{D_s}$, the function $\log v$
is bounded, so differentiation under the integral sign is valid
near exponent zero. Since $\nu_s(D_s)=1$, we have
$\log\int_{D_s}v^0\,d\nu_s=0$. The definition of the derivative
therefore gives
\begin{equation}\label{eq:calc-s4-064}
 \begin{aligned}
 \lim_{k\to\infty}2^k\log\int_{D_s}v^{p/2^k}\,d\nu_s
 &=p\left.\frac{d}{d\tau}
            \log\int_{D_s}v^\tau\,d\nu_s\right|_{\tau=0}\\
 &=p\int_{D_s}\log v\,d\nu_s.
 \end{aligned}
\end{equation}
This limit holds for either sign of $p$. It is taken for each
fixed $v$, while the constant $C$ in the iterated estimate is
independent of $v$.
Letting $k\to\infty$ in the iterated estimate, we conclude that
\begin{equation}\label{eq:calc-s4-065}
 \log\int_{D_s}v^p\,d\nu_s
 \leq p\int_{D_s}\log v\,d\nu_s+2Cp^2,
 \qquad 0<|p|\leq p_0.
\end{equation}
Apply this inequality at $p=p_0$ and $p=-p_0$ and add. The
terms containing the mean of $\log v$ cancel, leaving
\begin{equation}\label{eq:calc-s4-066}
 \log\int_{D_s}v^{p_0}\,d\nu_s
 +\log\int_{D_s}v^{-p_0}\,d\nu_s
 \leq4Cp_0^2.
\end{equation}
Exponentiation proves \eqref{eq:weighted-crossover} with
$C_*:=\exp(4Cp_0^2)$, which depends only on $n,\kappa$ and is
independent of the positive lower bound for $v$.

If $v$ is positive in $D_s$ but may vanish on its boundary,
apply the estimate just proved to $v+\varepsilon$.
For each fixed $\varepsilon>0$, the entire iteration above is
valid with the same $p_0,C_*$. We may now let
$\varepsilon\downarrow0$.
The positive-power integrals converge by dominated convergence,
since for $0<\varepsilon\leq1$ their integrands are bounded by
$(\|v\|_\infty+1)^{p_0}$. Their limit
$\int_{D_s}v^{p_0}\,d\nu_s$ is strictly positive.
The product bound thus bounds the negative-power integrals,
whose integrands increase to $v^{-p_0}$.
Monotone convergence proves \eqref{eq:weighted-crossover} also
in this case.
\end{proof}

\begin{proposition}
\label{prop:potential-oscillation}
Fix $s_0>0$. Let $\psi\geq0$ be $T^n$-invariant, and set
$D_s:=\{\psi<s\}$ for $0<s\leq s_0$.
Suppose that $\psi$ is smooth on a neighborhood of
$\overline{D_{s_0}}$, with $\omega:=dd^c\psi>0$, and that each
$D_s$ satisfies the hypotheses of
Proposition~\ref{prop:relative-green} with the same density ratio
$\kappa$.
There are $\delta\in(0,1)$ and $\alpha>0$, depending only on
$n,\kappa$, such that every real $T^n$-invariant $\omega$-harmonic function $q$
on that neighborhood satisfies, for $0<s\leq s_0$,
\begin{align}
 \operatorname{osc}_{D_{s/4}}q
 &\leq(1-\delta)\operatorname{osc}_{D_s}q,
 \label{eq:potential-contraction}\\
 \operatorname{osc}_{D_s}q
 &\leq C(s/s_0)^\alpha
          \operatorname{osc}_{D_{s_0}}q.
 \label{eq:potential-holder}
\end{align}
\end{proposition}

\begin{proof}
Fix $0<s\leq s_0$, and use $h=s-\psi$ and the probability
measure $\nu_s$ from \eqref{eq:torsion-probability}.
We first show that a low value of a positive harmonic function in
$D_{s/4}$ forces its lower level set in $D_{s/2}$ to have a
definite relative volume. Proposition~\ref{prop:weighted-crossover}
will then give an interior lower bound, from which the oscillation
estimate follows.

Let $v>0$ be a $T^n$-invariant $\omega$-harmonic function,
smooth on a neighborhood of $\overline{D_s}$.
For $b>0$, put $E:=D_{s/2}\cap\{v<b\}$, and suppose that
$v(y)\leq b/4$ at some $y\in D_{s/4}$.
Consider
\begin{equation}\label{eq:calc-s4-067}
 \zeta:=\frac{s/2-\psi}{n}-\frac{s v}{2nb}
\end{equation}
as a barrier. Since $\Delta_\omega\psi=n$ and
$\Delta_\omega v=0$, we have
\begin{equation}\label{eq:calc-s4-068}
 -\Delta_\omega\zeta=1,\qquad
 \zeta(y)\geq\frac{s/2-s/4}{n}
                    -\frac{s(b/4)}{2nb}
             =\frac{s}{8n}.
\end{equation}
The subtraction of the harmonic term restricts the positive set
of $\zeta$ without changing its Laplacian. Indeed, $\zeta\leq0$
when $\psi\geq s/2$, and when $v\geq b$ we have
$\zeta\leq-\psi/n\leq0$. Thus $\{\zeta>0\}\subset E$,
and $\zeta_+:=\max\{\zeta,0\}$ vanishes outside $E$ and near
$\partial D_s$.
Smooth convex approximation of the positive part gives
\begin{equation}\label{eq:calc-s4-069}
 -\Delta_\omega\zeta_+
 \leq\mathbf1_{\{\zeta>0\}}(-\Delta_\omega\zeta)
 =\mathbf1_{\{\zeta>0\}}\leq\mathbf1_E
\end{equation}
in the distribution sense. Here convexity makes the additional
gradient term in $-\Delta_\omega\zeta_+$ nonpositive.
By definition, $G_{D_s}\mathbf1_E$ has zero boundary values and
solves $-\Delta_\omega(G_{D_s}\mathbf1_E)=\mathbf1_E$.
The weak maximum principle therefore yields
$\zeta_+\leq G_{D_s}\mathbf1_E$ on $D_s$, and in particular
\begin{equation}\label{eq:calc-s4-070}
 \frac{s}{8n}\leq\zeta_+(y)\leq G_{D_s}\mathbf1_E(y).
\end{equation}
This comparison is made on $D_s$ and requires no regularity of
$\partial E$.

Since $E$ is $T^n$-invariant, Proposition~\ref{prop:relative-green}
applies. With $m_E=\mu(E)/\mu(D_s)$, it gives
\begin{equation}\label{eq:calc-s4-071}
 \frac{s}{8n}\leq G_{D_s}\mathbf1_E(y)
 \leq Cs\,m_E^{1/n}\Bigl(1+\log\frac1{m_E}\Bigr).
\end{equation}
After division by $s$, the expression on the right tends to zero
as $m_E\downarrow0$. Consequently $m_E\geq m_*>0$, where
$m_*$ depends only on $n,\kappa$.
Since $h\geq s/2$ on $E$ and $0<h\leq s$ on $D_s$, we obtain
\begin{equation}
 \nu_s(E)
 =\frac{\int_Eh\,d\mu}{\int_{D_s}h\,d\mu}
 \geq\frac{(s/2)\mu(E)}{s\mu(D_s)}
 =\frac{m_E}{2}\geq\frac{m_*}{2}.
 \label{eq:lowset-weighted-mass}
\end{equation}
Thus the existence of a point $y\in D_{s/4}$ with $v(y)\leq b/4$
forces the weighted measure of $E$ to be bounded below independently
of $s$, $v$, and $b$.

We next use Proposition~\ref{prop:weighted-crossover} to bound
the same set from above. Write
$m_v:=(\int_{D_s}v^{p_0}\,d\nu_s)^{1/p_0}$, where $p_0$ is
the exponent in that proposition. Its estimate gives
\begin{equation}\label{eq:calc-s4-072}
 \nu_s\{v<b\}
 \leq b^{p_0}\int_{D_s}v^{-p_0}\,d\nu_s
 \leq C_*(b/m_v)^{p_0}.
\end{equation}
Choose $c_2\in(0,1/4)$, depending only on $n,\kappa$, so small
that $C_*(4c_2)^{p_0}<m_*/2$, and set $b=4c_2m_v$.
If $v(y)\leq c_2m_v=b/4$ at any point of $D_{s/4}$,
\eqref{eq:lowset-weighted-mass} would imply
\begin{equation}\label{eq:calc-s4-073}
 \frac{m_*}{2}\leq\nu_s(E)
 \leq\nu_s\{v<b\}
 \leq C_*(4c_2)^{p_0}<\frac{m_*}{2},
\end{equation}
a contradiction. We have therefore proved the weak Harnack estimate
\begin{equation}
 \inf_{D_{s/4}}v
 \geq c_2\left(\int_{D_s}v^{p_0}\,d\nu_s\right)^{1/p_0},
 \qquad c_2=c_2(n,\kappa)>0.
 \label{eq:weighted-weak-Harnack}
\end{equation}

We now apply this estimate to $q$. If
$\operatorname{osc}_{D_s}q=0$, then
\eqref{eq:potential-contraction} is immediate. Otherwise the sets
\begin{equation}\label{eq:calc-s4-074}
 \left\{q-\inf_{D_s}q\geq\frac12\operatorname{osc}_{D_s}q\right\},
 \qquad
 \left\{\sup_{D_s}q-q\geq\frac12\operatorname{osc}_{D_s}q\right\}
\end{equation}
cover $D_s$, so at least one has $\nu_s$-measure at least $1/2$.
If the first set has this property, then
\begin{align}
 \int_{D_s}\left(q-\inf_{D_s}q\right)^{p_0}\,d\nu_s
 &\geq\left(\frac{\operatorname{osc}_{D_s}q}{2}\right)^{p_0}
       \nu_s\left\{q-\inf_{D_s}q
              \geq\frac12\operatorname{osc}_{D_s}q\right\}\notag\\
 &\geq2^{-p_0-1}\left(\operatorname{osc}_{D_s}q\right)^{p_0}.
 \label{eq:calc-s4-075}
\end{align}
Apply \eqref{eq:weighted-weak-Harnack} to the positive harmonic
function $q-\inf_{D_s}q+\varepsilon$ and let
$\varepsilon\downarrow0$. The constant is independent of
$\varepsilon$, and the integral converges by dominated convergence.
With $\delta:=c_2\,2^{-1-1/p_0}\in(0,1)$, we obtain
\begin{equation}\label{eq:calc-s4-076}
 \inf_{D_{s/4}}q-\inf_{D_s}q
 \geq c_2\left(\int_{D_s}
              \left(q-\inf_{D_s}q\right)^{p_0}\,d\nu_s\right)^{1/p_0}
 \geq\delta\operatorname{osc}_{D_s}q.
\end{equation}
In the other case, the same argument applied to
$\sup_{D_s}q-q+\varepsilon$ gives
\begin{equation}\label{eq:oscillation-upper-gap}
 \sup_{D_s}q-\sup_{D_{s/4}}q
 \geq\delta\operatorname{osc}_{D_s}q.
\end{equation}
Since $D_{s/4}\subset D_s$, either inequality proves
\eqref{eq:potential-contraction}.

All constants above depend only on $n,\kappa$, so the same
contraction factor applies at every height. Iteration gives
\begin{equation}\label{eq:calc-s4-077}
 \operatorname{osc}_{D_{4^{-k}s_0}}q
 \leq(1-\delta)^k\operatorname{osc}_{D_{s_0}}q,
 \qquad k\geq0.
\end{equation}
Set $\alpha:=-\log(1-\delta)/\log4>0$.
For an arbitrary $0<s\leq s_0$, choose $k\geq0$ such that
$4^{-k-1}s_0<s\leq4^{-k}s_0$. By inclusion of the sublevel sets,
\begin{align}
\operatorname{osc}_{D_s}q
 &\leq\operatorname{osc}_{D_{4^{-k}s_0}}q\leq(1-\delta)^k\operatorname{osc}_{D_{s_0}}q\notag\\
 &= (4^{-k})^\alpha\operatorname{osc}_{D_{s_0}}q\leq4^\alpha(s/s_0)^\alpha\operatorname{osc}_{D_{s_0}}q.
 \label{eq:calc-s4-078}
\end{align}
This proves \eqref{eq:potential-holder}.
\end{proof}

For the toric domains $D_{p,s,t}$ introduced above, when
$p\in\Delta^\circ$, their logarithmic images are relatively
compact sections of $\Phi_t$. The interior oscillation estimates
of Caffarelli--Guti\'errez \cite{CaffarelliGutierrez1997} therefore
apply directly. When $p\in\partial\Delta$, the
logarithmic domains are unbounded. Our proof follows their
general strategy of deriving oscillation decay from measure
estimates and Harnack inequalities, adapted to the toric
boundary setting.

\subsubsection{Proof of Proposition~\ref{prop:common-modulus}}
We now return to the fixed character coordinates $z$ and the
associated variables $\xi,\theta$ of Section~\ref{sec:preliminaries},
and apply the oscillation estimate to the domains $D_{p,s,t}$
of Definition~\ref{def:moment-sublevel-sets}. To place nearby points
of $M$ in a common such domain of small height, we first establish
a uniform modulus of continuity for $u_t$ on the closed polytope.

\begin{lemma}
\label{lem:symplectic-common-modulus}
If $\sup_t\|\varphi_t\|_\infty<\infty$ and
$\omega_t^n/\omega_0^n\leq C$, then
\begin{equation}
 |u_t(x)-u_t(x')|\leq C|x-x'|^{1/2},
 \qquad x,x'\in\Delta.
 \label{eq:symplectic-half-holder}
\end{equation}
\end{lemma}

\begin{proof}
Equation~\eqref{eq:legendre-C0} bounds $u_t$ uniformly.
Using the volume-ratio bound and the background exponential decay
in Proposition~\ref{prop:logarithmic-tail}, we obtain on the open orbit
\begin{equation}\label{eq:calc-s4-079}
 \int_\Delta|Du_t|^{2n}\,dx
 =\int_{\mathbb R^n}|\xi|^{2n}\det D_\xi^2\Phi_t\,d\xi
 \leq C\int_{\mathbb R^n}|\xi|^{2n}
                       \det D_\xi^2\Phi_0\,d\xi<\infty .
\end{equation}
Together with the $L^\infty$ bound for $u_t$, this proves
$\|u_t\|_{W^{1,2n}(\Delta^\circ)}\leq C$.
Since $\Delta^\circ$ is a bounded Lipschitz domain, Morrey's
inequality gives the exponent $1-n/(2n)=1/2$.
The resulting continuous representative agrees with the
continuous extension of $u_t$ to $\Delta$, proving
\eqref{eq:symplectic-half-holder}.
\end{proof}

\begin{proof}[Proof of Proposition~\ref{prop:common-modulus}]
Uniqueness of the normalized elliptic solution implies that $P_t$
is $T^n$-invariant. 
If \eqref{eq:common-physical-modulus} were false, there would be
$t_j,y_j,y'_j$ with $d_0(y_j,y'_j)\to0$ but
$|P_{t_j}(y_j)-P_{t_j}(y'_j)|\geq\varepsilon_*>0$.
For each $j$, continuity of $P_{t_j}$ and density of the open
orbit allow us to perturb both points into that orbit, retaining
a fixed positive lower bound for the difference.
After passing to subsequences, let $y_j,y'_j\to y$ and
$x_j:=x_{t_j}(y_j)\to p\in\Delta$.

We first show that $y\in U_p$, that is, $y$ lies in none of the
divisors $\mathcal D_r$ with $\ell_r(p)>0$.
Suppose to the contrary that $y\notin U_p$. By
Definition~\ref{def:moment-sublevel-sets}, there is an index $r$
such that $\ell_r(p)>0$ and $\ell_r(x_0(y))=0$.
We will show that
$\left|\sum_k\ell_k(x_j)\log\ell_k(x_0(y_j))\right|$
is uniformly bounded, whereas this assumption forces it to tend
to $+\infty$.

On the one hand, subtracting the Legendre identities for
$\Phi_{t_j}$ and $\Phi_0$ at the same point $y_j$ gives
\begin{equation}\label{eq:calc-s4-080}
 \langle x_j-x_0(y_j),\xi(y_j)\rangle
 =\varphi_{t_j}(y_j)+u_{t_j}(x_j)-u_0(x_0(y_j)).
\end{equation}
The right-hand side is uniformly bounded by
\eqref{eq:basic-modulus-inputs} and \eqref{eq:legendre-C0}.
Since $\ell_k(x_j)-\ell_k(x_0(y_j))
=\langle x_j-x_0(y_j),\nu_k\rangle$, substituting
\eqref{eq:background-log-coordinates} gives
\begin{align}
 &\left|\sum_k\bigl(\ell_k(x_j)-\ell_k(x_0(y_j))\bigr)
                       \log\ell_k(x_0(y_j))\right|\notag\\
 &\quad=\left|\langle x_j-x_0(y_j),\xi(y_j)\rangle
       -\left\langle x_j-x_0(y_j),
         \sum_k\nu_k+Dh_0(x_0(y_j))\right\rangle\right|\notag\\
 &\quad\leq\|\varphi_{t_j}\|_\infty+\|u_{t_j}\|_\infty
                   +\|u_0\|_\infty+\operatorname{diam}(\Delta)
       \left(\left|\sum_k\nu_k\right|
                       +\|Dh_0\|_{L^\infty(\Delta)}\right)
 \leq C.
 \label{eq:calc-s4-081}
\end{align}
Here $C$ depends only on the fixed background and the bounds in
\eqref{eq:basic-modulus-inputs}, and is independent of $j$.
The terms $\ell_k(x_0(y_j))\log\ell_k(x_0(y_j))$ are also
uniformly bounded, since $s\log s$ extends continuously to
$s=0$ with value zero and $\ell_k$ is bounded on $\Delta$.
Adding these terms removes the background coefficient and gives
\begin{align}
 &\left|\sum_k\ell_k(x_j)\log\ell_k(x_0(y_j))\right|\leq C.
 \label{eq:calc-s4-082}
\end{align}

On the other hand, for the index $r$ chosen above, continuity of
$x_0$ and $x_j\to p$ give
\begin{equation}\label{eq:calc-s4-083}
 \ell_r(x_j)\longrightarrow\ell_r(p)>0,
 \qquad \log\ell_r(x_0(y_j))\longrightarrow-\infty.
\end{equation}
Thus the $r$-th product tends to $-\infty$.
No other term can cancel this divergence, because all the other
terms have a uniform upper bound. Indeed, we have, for $k\ne r$,
\begin{equation}\label{eq:boundary-face-other-terms}
 \ell_k(x_j)\log\ell_k(x_0(y_j))
 \leq\ell_k(x_j)\ell_k(x_0(y_j))
 \leq\left(\max_{x\in\Delta}\ell_k(x)\right)^2\leq C.
\end{equation}
For all sufficiently large $j$, since there are only finitely many
facets, it follows that
\begin{align}
 \sum_k\ell_k(x_j)\log\ell_k(x_0(y_j))
 &=\ell_r(x_j)\log\ell_r(x_0(y_j))
   +\sum_{k\ne r}\ell_k(x_j)\log\ell_k(x_0(y_j))\notag\\
 &\leq\frac{\ell_r(p)}2\log\ell_r(x_0(y_j))+C
 \longrightarrow-\infty.
 \label{eq:boundary-face-sum-divergence}
\end{align}
Consequently,
\begin{equation}\label{eq:boundary-face-absolute-divergence}
 \left|\sum_k\ell_k(x_j)\log\ell_k(x_0(y_j))\right|
 \longrightarrow+\infty,
\end{equation}
contradicting \eqref{eq:calc-s4-082}.
Hence $y$ lies in none of the divisors with $\ell_r(p)>0$, and
$y\in U_p$ by Definition~\ref{def:moment-sublevel-sets}. 

For this fixed limit $p$, we use the potentials $\psi_{p,t_j}$
from Definition~\ref{def:moment-sublevel-sets}, given by
\eqref{eq:fixed-moment-domains} on the open orbit and extended
smoothly to $U_p$ by Proposition~\ref{prop:sublevel-geometry}. 
We next show that $\psi_{p,t_j}(y_j)$ and
$\psi_{p,t_j}(y'_j)$ both tend to zero.
For all sufficiently large $j$, the reflected point
$2x_j-p$ lies in $\Delta$: if $\ell_r(p)=0$, then
$\ell_r(2x_j-p)=2\ell_r(x_j)\geq0$; if $\ell_r(p)>0$, then
$\ell_r(2x_j-p)=2\ell_r(x_j)-\ell_r(p)\to\ell_r(p)>0$.
There are only finitely many facet inequalities to check.
Since $x_j\in\Delta^\circ$ and $D u_{t_j}(x_j)=\xi(y_j)$,
convexity gives
\begin{equation}\label{eq:fixed-center-supporting-plane}
 u_{t_j}(x)\geq u_{t_j}(x_j)
       +\langle D u_{t_j}(x_j),x-x_j\rangle,
 \qquad x\in\Delta.
\end{equation}
The inequality extends to $x\in\partial\Delta$ by continuity;
the derivative is taken only at the interior point $x_j$.
Taking $x=2x_j-p$ gives
\begin{align}
 u_{t_j}(2x_j-p)-u_{t_j}(x_j)
 &\geq\langle D u_{t_j}(x_j),(2x_j-p)-x_j\rangle=\langle\xi(y_j),x_j-p\rangle.
 \label{eq:fixed-center-reflection-bound}
\end{align}
Substituting this bound and applying
Lemma~\ref{lem:symplectic-common-modulus}, we obtain
\begin{align}
\psi_{p,t_j}(y_j)
 &=\Phi_{t_j}(\xi(y_j))-\langle p,\xi(y_j)\rangle+u_{t_j}(p) \notag\\
 &=u_{t_j}(p)-u_{t_j}(x_j)
                +\langle x_j-p,\xi(y_j)\rangle \notag\\
 &\leq u_{t_j}(p)+u_{t_j}(2x_j-p)-2u_{t_j}(x_j) \notag\\
 &\leq |u_{t_j}(p)-u_{t_j}(x_j)|
       +|u_{t_j}(2x_j-p)-u_{t_j}(x_j)| \notag\\
 &\leq C|p-x_j|^{1/2}+C|(2x_j-p)-x_j|^{1/2} \notag\\
 &=2C|x_j-p|^{1/2}\longrightarrow0,
 \label{eq:fixed-center-small-height}
\end{align}
For large $j$, both points lie in a fixed neighborhood of $y$ in
$U_p$. By \eqref{eq:sublevel-background-comparison}, we have
\begin{align}
 \psi_{p,t_j}(y'_j)-\psi_{p,t_j}(y_j)
 &=\psi_p^0(y'_j)-\psi_p^0(y_j)
       +\varphi_{t_j}(y'_j)-\varphi_{t_j}(y_j),
 \label{eq:calc-s4-084}\\
 |\psi_{p,t_j}(y'_j)-\psi_{p,t_j}(y_j)|
 &\leq|\psi_p^0(y'_j)-\psi_p^0(y_j)|
       +C d_0(y'_j,y_j)^\beta\longrightarrow0.
 \label{eq:calc-s4-085}
\end{align}
Together with \eqref{eq:fixed-center-small-height}, this proves
$\psi_{p,t_j}(y'_j)\to0$.

On $U_p$ set
\begin{equation}
 q_t:=P_t-\frac{\underline R}{n}\psi_{p,t}
                 -\log\frac{\omega_0^n}{dV_{\mathrm{Eucl}}}.
 \label{eq:physical-harmonic-gauge}
\end{equation}
All three terms are smooth on $U_p$, including its toric divisors.
Since $dd^c\psi_{p,t}=\omega_t$ and
$\operatorname{Ric}(\omega_0)
 =-dd^c\log(\omega_0^n/dV_{\mathrm{Eucl}})$, the equation for $P_t$
gives $\Delta_{\varphi_t}q_t=0$.
The function $q_t$ is $T^n$-invariant.
Choose $s_j\downarrow0$ with both points in $D_{p,s_j,t_j}$.
Fix $s_0>0$. Proposition~\ref{prop:fixed-moment-domains}, applied
at height $s_0+1$, places all the closures
$\overline{D_{p,s_0,t_j}}$ in a fixed compact set $K\Subset U_p$.
On $K$, $|\omega_0^n/dV_{\mathrm{Eucl}}|$
is bounded. Therefore, on $D_{p,s_0,t_j}$,
\begin{equation}\label{eq:calc-s4-086}
 |q_{t_j}|
 \leq\|P_{t_j}\|_\infty+\frac{|\underline R|}{n}s_0
       +\max_K\left|\log\frac{\omega_0^n}{dV_{\mathrm{Eucl}}}\right|
 \leq C_p.
\end{equation}
The geometric properties in Propositions~\ref{prop:sublevel-geometry}
and~\ref{prop:sublevel-domains}
allow us to apply Proposition~\ref{prop:potential-oscillation}, giving,
for all sufficiently large $j$,
\begin{equation}\label{eq:calc-s4-087}
 |q_{t_j}(y_j)-q_{t_j}(y'_j)|
 \leq C_p(s_j/s_0)^{\alpha_p}\longrightarrow0.
\end{equation}
Here $p$ has already been fixed, so the dependence of the
constants on $p$ causes no loss of uniformity in $j$.
The other two terms in \eqref{eq:physical-harmonic-gauge} have a
uniform modulus of continuity near $y$. Hence
$|P_{t_j}(y_j)-P_{t_j}(y'_j)|\to0$, a contradiction. Hence  \eqref{eq:common-physical-modulus} is true.
\end{proof}

\subsection{A negative exponential estimate}
\label{sec:nonlocal-exponential}

We now estimate the normalized elliptic solution $P$ in
\eqref{eq:pcf} for $\varphi\in\mathcal H_0^T$, using
\eqref{eq:normalized-volume-measures}--\eqref{eq:entropy-dirichlet-definitions}.
The argument is static and uses the fixed logarithmic
coordinates $\xi$ on the open orbit.

\begin{proposition}\label{prop:negative-P-exponent}
There is $r>1$, depending only on the fixed toric background, with
the following property. For every $B>0$, there is $C_B$ such that
every smooth $T^n$-invariant potential satisfying
\begin{equation}\label{eq:calc-s4-088}
 E(\varphi)=0,\qquad J(\varphi)+\int_M F\,d\mu_\varphi\leq B
\end{equation}
has
\begin{equation}\label{eq:negative-P-exponent}
 \int_M e^{-rP}\,d\mu_0\leq C_B.
\end{equation}
\end{proposition}

We shall prove the proposition by obtaining a uniform linear lower
bound for $P-\log m_0$ in the fixed logarithmic coordinates. The
exponential bounds for $m_0$ in
\eqref{eq:background-exponential-tail} will then give
\eqref{eq:negative-P-exponent}. We first establish the required
$L^1$ bound for $P$.

\begin{lemma}\label{lem:toric-L1-dirichlet}
There is a background constant $C$ such that, for every
$\varphi\in\mathcal H_0^T$ and every smooth $T^n$-invariant
function $v$,
\begin{equation}\label{eq:toric-L1-dirichlet}
 \int_M|v-\overline v_\varphi|\,d\mu_\varphi
 \leq C\sqrt{1+J(\varphi)}\,\|\nabla_\varphi v\|_{L^2(d\mu_\varphi)},
 \qquad
 \overline v_\varphi:=\int_Mv\,d\mu_\varphi.
\end{equation}
Consequently, the normalized elliptic solution $P$ in
\eqref{eq:pcf} satisfies
\begin{equation}\label{eq:P-low-energy}
 \|\nabla_\varphi P\|_{L^2(d\mu_\varphi)}^2+\int_M|P|\,d\mu_\varphi
 \leq C(1+J(\varphi)).
\end{equation}
\end{lemma}

\begin{proof}
Let $u$ be the symplectic potential corresponding to $\varphi$,
with $E(\varphi)=0$. Since $v$ is $T^n$-invariant, we view it as
a function of $\xi$ on the open orbit and put
$a(x):=v(Du(x))$, $a_\Delta:=|\Delta|^{-1}\int_\Delta a\,dx$,
and $d(x):=\operatorname{dist}(x,\partial\Delta)$.
Write $U:=D_x^2u$ and $(u^{ij}):=U^{-1}$.
The bounded convex domain $\Delta^\circ$ is a John domain.
The constant weight $1$ is doubling and satisfies the local
Poincar\'e inequality on cubes, so the improved Poincar\'e
inequality of \cite[Theorem~2.2]{AcostaCejasDuran2017} applies
with $p=1$. Together with \eqref{eq:toric-integration} and
\eqref{eq:calc-s2-020}, it gives
\begin{align}
 \int_M|v-\overline v_\varphi|\,d\mu_\varphi
 &=\frac1{|\Delta|}\int_\Delta|a-a_\Delta|\,dx \leq\frac{C_\Delta}{|\Delta|}\int_\Delta d\,|Da|\,dx\notag\\
 &\leq\frac{C_\Delta}{|\Delta|}\int_\Delta
       d\,(\tr D^2u)^{1/2}(u^{ij}a_i a_j)^{1/2}\,dx\notag\\
 &\leq C_\Delta
       \left(\frac1{|\Delta|}\int_\Delta d^2\tr D^2u\,dx\right)^{1/2}
       \left(\frac1{|\Delta|}\int_\Delta u^{ij}a_i a_j\,dx\right)^{1/2}\notag\\
 &=C_\Delta
       \left(\frac1{|\Delta|}\int_\Delta d^2\tr D^2u\,dx\right)^{1/2}
       \|\nabla_\varphi v\|_{L^2(d\mu_\varphi)}.
 \label{eq:calc-s4-100}
\end{align}

Since $U$ is positive definite, $U\leq(\tr U)I$ and
$U^{-1}\geq(\tr U)^{-1}I$. Thus
\begin{equation}\label{eq:calc-s4-099}
 |Da|^2\leq(\tr U)u^{ij}a_i a_j.
\end{equation}

We now estimate the first integral factor in
\eqref{eq:calc-s4-100}. Fix $x_*\in\Delta^\circ$.
Subtract the supporting affine function at $x_*$ and define
\begin{equation}\label{eq:calc-s4-089}
 u^\sharp(x):=u(x)-u(x_*)-Du(x_*)\cdot(x-x_*)\geq0.
\end{equation}
The Hessian is unchanged: $D^2u^\sharp=D^2u$.
Since $\Delta$ is fixed and bounded, \eqref{eq:toric-J-L1} and
\eqref{eq:convex-interior-estimate} give
\begin{align}
 \int_\Delta u^\sharp\,dx
 &\leq\int_\Delta|u|\,dx+|\Delta|\,|u(x_*)|
       +|Du(x_*)|\int_\Delta|x-x_*|\,dx\notag\\
 &\leq C(1+J(\varphi)).
 \label{eq:calc-s4-090}
\end{align}

The facet description of $\Delta$ gives
\begin{equation}\label{eq:calc-s4-091}
 d(x)=\min_r\frac{\ell_r(x)}{|\nu_r|}.
\end{equation}
Thus $-d$ is convex, $|Dd|=1$ almost everywhere, and
$\Delta_{\R^n}d$ is a nonpositive measure on $\Delta^\circ$.
For $\eta>0$, define $w_\eta:=(d-\eta)_+^2$. This function has
compact support in $\Delta^\circ$, and the distributional chain
rule gives
\begin{equation}\label{eq:calc-s4-092}
 \Delta_{\R^n}w_\eta
 =2\mathbf1_{\{d>\eta\}}|Dd|^2\,dx
      +2(d-\eta)_+\Delta_{\R^n}d\leq2\,dx.
\end{equation}
The second term is nonpositive. At $d=\eta$, the first derivative
of the squared positive part vanishes, so no additional interface
measure occurs there. Since $w_\eta$ is compactly supported in
the interior, integration by parts gives
\begin{align}
 \int_\Delta w_\eta\tr D^2u\,dx
 &=\int_\Delta w_\eta\Delta_{\R^n}u^\sharp\,dx =\int_\Delta u^\sharp\,\Delta_{\R^n}w_\eta
 \leq2\int_\Delta u^\sharp\,dx.
 \label{eq:calc-s4-093}
\end{align}
The last integral against $\Delta_{\R^n}w_\eta$ is the pairing
with a measure, and the inequality uses $u^\sharp\geq0$.
Only the interior smoothness of $u$ is needed here: one may
multiply $u^\sharp$ by a smooth cutoff equal to one near the
support of $w_\eta$ before taking the distributional pairing.
As $\eta\downarrow0$, we have
$w_\eta\uparrow d^2$ and $\tr D^2u\geq0$.
Monotone convergence on the left therefore gives
\begin{equation}\label{eq:boundary-distance-hessian}
 \int_\Delta d^2\tr D^2u\,dx
 \leq2\int_\Delta u^\sharp\,dx
 \leq C(1+J(\varphi)).
\end{equation}

Substituting \eqref{eq:boundary-distance-hessian} into
\eqref{eq:calc-s4-100}, we obtain
\begin{equation}\label{eq:toric-L1-dirichlet-conclusion}
 \int_M|v-\overline v_\varphi|\,d\mu_\varphi
 \leq C\sqrt{1+J(\varphi)}\,
       \|\nabla_\varphi v\|_{L^2(d\mu_\varphi)}.
\end{equation}
This proves \eqref{eq:toric-L1-dirichlet}. Since $x_*$ and
$\Delta$ are fixed, $C$ depends only on the background.

Choose $C$ such that
$-C\omega_0\leq\Ric(\omega_0)\leq C\omega_0$.
Multiplying the elliptic equation in \eqref{eq:pcf} by $P$
and integrating gives
\begin{equation}\label{eq:calc-s4-101}
 \|\nabla_\varphi P\|_{L^2(d\mu_\varphi)}^2
 \leq C\int_M|P|
       \bigl(1+\tr_{\omega_\varphi}\omega_0\bigr)\,d\mu_\varphi.
\end{equation}
Since $\tr_{\omega_\varphi}\omega_0=n-\Delta_\varphi\varphi$,
integration by parts gives
\begin{equation}\label{eq:P-trace-energy}
 \int_M|P|\tr_{\omega_\varphi}\omega_0\,d\mu_\varphi
 \leq n\int_M|P|\,d\mu_\varphi+
       \|\nabla_\varphi P\|_{L^2(d\mu_\varphi)}
       \|\nabla_\varphi \varphi\|_{L^2(d\mu_\varphi)}.
\end{equation}
One may first use $(P^2+\eta^2)^{1/2}$ and let $\eta$ tend to zero.
The standard positive expansion
\begin{equation}\label{eq:calc-s4-102}
 J(\varphi)=\frac1{(n+1)V}
   \sum_{j=0}^{n-1}(n-j)
   \int_M i\partial\varphi\wedge\bar\partial\varphi
           \wedge\omega_\varphi^j\wedge\omega_0^{n-1-j}
\end{equation}
implies $\|\nabla_\varphi \varphi\|_{L^2(d\mu_\varphi)}^2\leq n(n+1)J(\varphi)$.
Using the mean zero normalization of $P$ and
\eqref{eq:toric-L1-dirichlet} in \eqref{eq:P-trace-energy},
we conclude that
\begin{align}
 \|\nabla_\varphi P\|_{L^2(d\mu_\varphi)}^2
 &\leq C\sqrt{1+J(\varphi)}\,
                 \|\nabla_\varphi P\|_{L^2(d\mu_\varphi)}\notag\\
 &\leq\frac12\|\nabla_\varphi P\|_{L^2(d\mu_\varphi)}^2
                 +C(1+J(\varphi)).
 \label{eq:calc-s4-103}
\end{align}
The last inequality is Young's inequality. Absorbing the first
term on the right gives
\begin{equation}\label{eq:calc-s4-104}
 \|\nabla_\varphi P\|_{L^2(d\mu_\varphi)}^2
 \leq C(1+J(\varphi)).
\end{equation}
Since $\int_M P\,d\mu_\varphi=0$, a second application of
\eqref{eq:toric-L1-dirichlet} yields
\begin{equation}\label{eq:calc-s4-105}
 \int_M|P|\,d\mu_\varphi
 \leq C\sqrt{1+J(\varphi)}\,
             \|\nabla_\varphi P\|_{L^2(d\mu_\varphi)}
 \leq C(1+J(\varphi)).
\end{equation}
Adding these two estimates and enlarging $C$ proves
\eqref{eq:P-low-energy}.
\end{proof}

We next convert integral control into a uniform lower bound by
selecting global supporting affine functions with controlled
intercepts. We use the following contact set.

\begin{definition}\label{def:lower-contact-set}
For $W\in C^1(\R^n)$, its global lower contact set is
\begin{equation}\label{eq:calc-s4-106}
 \mathcal C_W:=\{\xi\in\R^n:
 W(y)\geq W(\xi)+DW(\xi)\cdot(y-\xi)
                         \text{ for every }y\in\R^n\}.
\end{equation}
Thus the tangent affine function at a point of $\mathcal C_W$
supports $W$ from below on all of $\R^n$.
\end{definition}

The following lemma makes the additive constant in a linear lower
bound uniform. The function $W$ itself need not be convex.

\begin{lemma}\label{lem:global-contact}
Let $H$ be a smooth symmetric positive definite matrix field on $\R^n$, and
put $m:=\det H$. Suppose that $W\in C^2(\R^n)$ satisfies
\begin{equation}\label{eq:global-contact-assumptions}
 \tr(H^{-1}D^2W)=\Lambda>0,\qquad
 W(\xi)\geq\kappa|\xi|-C_W,\qquad
 \int_{\R^n}(|W|+|\xi|)m\,d\xi\leq A,
\end{equation}
where $\kappa>0$ is fixed, while $C_W$ may depend on $W$.
Then there is a constant $C_A$, depending only on
$n,\kappa,\Lambda,A$ and independent of $C_W$, such that
\begin{equation}\label{eq:global-contact-conclusion}
 W(\xi)\geq\frac\kappa4|\xi|-C_A\qquad\text{on }\R^n.
\end{equation}
\end{lemma}

\begin{proof}
Write $\mathcal C:=\mathcal C_W$ for the set in
Definition~\ref{def:lower-contact-set}; this set is closed.
We write $\boldsymbol{\gamma}=(\gamma_1,\ldots,\gamma_n)\in\R^n$
for the slope vector of an affine function, whose linear part is
$\boldsymbol{\gamma}\cdot\xi$ and whose gradient is $\boldsymbol{\gamma}$.
For every such vector with $|\boldsymbol{\gamma}|<\kappa$,
\begin{equation}\label{eq:calc-s4-107}
 W(\xi)-\boldsymbol{\gamma}\cdot\xi
 \geq(\kappa-|\boldsymbol{\gamma}|)|\xi|-C_W\longrightarrow+\infty
 \qquad\text{as }|\xi|\to\infty.
\end{equation}
Hence this function attains its minimum at some $\xi_{\boldsymbol{\gamma}}$, where
$DW(\xi_{\boldsymbol{\gamma}})=\boldsymbol{\gamma}$ and
\begin{equation}\label{eq:calc-s4-108}
 W(\xi)\geq W(\xi_{\boldsymbol{\gamma}})+\boldsymbol{\gamma}\cdot(\xi-\xi_{\boldsymbol{\gamma}})
                         \qquad\text{for all }\xi\in\R^n.
\end{equation}
Thus $\xi_{\boldsymbol{\gamma}}\in\mathcal C$: every slope in $B_\kappa(0)$ occurs
at a global lower contact point. Since $|\boldsymbol{\gamma}|<\kappa$, a bound for
$|W(\xi_{\boldsymbol{\gamma}})|+|\xi_{\boldsymbol{\gamma}}|$ gives a lower bound for the intercept
$W(\xi_{\boldsymbol{\gamma}})-\boldsymbol{\gamma}\cdot\xi_{\boldsymbol{\gamma}}$.

For $L>0$, set
\begin{equation}\label{eq:calc-s4-109}
 \mathcal C_L:=\{\xi\in\mathcal C:|W(\xi)|+|\xi|>L\}.
\end{equation}
The gradient image
\begin{equation}\label{eq:calc-s4-110}
 DW(\mathcal C_L):=\{DW(\xi):\xi\in\mathcal C_L\}
\end{equation}
is the set of slopes attained at these contact points.
Let $|DW(\mathcal C_L)|$ denote $n$-dimensional
Euclidean volume in the space of slopes. At a point of $\mathcal C$, the difference between $W$ and its
supporting affine function has a minimum, so $D^2W\geq0$ there.
The matrix $H^{-1/2}D^2W H^{-1/2}$ is therefore nonnegative
and its trace is
$\tr(H^{-1/2}D^2W H^{-1/2})=\tr(H^{-1}D^2W)=\Lambda$.
We have
\begin{equation}\label{eq:contact-Jacobian}
 0\leq\frac{\det D^2W}{m}
 =\det(H^{-1/2}D^2W H^{-1/2})
 \leq(\Lambda/n)^n\qquad\text{on }\mathcal C.
\end{equation}

The map $DW$ is $C^1$ and Lipschitz on every bounded ball.
We apply the area formula
\cite[Chapter~2, \S3, equation~(3.4)]{Simon2014}
on these balls and then use an increasing exhaustion of $\R^n$.
The formula counts a slope with the number of its preimages in
$\mathcal C_L$, which is at least one on $DW(\mathcal C_L)$.
Together with \eqref{eq:contact-Jacobian} and the integral
assumption, this gives
\begin{align}
|DW(\mathcal C_L)|
 &\leq\int_{\R^n}
       \#\{y\in\mathcal C_L:DW(y)=\boldsymbol{\gamma}\}\,d\boldsymbol{\gamma} =\int_{\mathcal C_L}|\det D^2W|\,d\xi \notag\\
 &\leq(\Lambda/n)^n\int_{\mathcal C_L}m\,d\xi \leq\frac{(\Lambda/n)^n}{L}
       \int_{\mathcal C_L}(|W|+|\xi|)m\,d\xi
 \leq\frac{(\Lambda/n)^n A}{L}.
 \label{eq:bad-contact-gradient-image}
\end{align}
Here $d\boldsymbol{\gamma}$ is Lebesgue measure on the space of
slope vectors, and $\#$ denotes the number of preimages, possibly infinite.
The left-hand side counts each slope only once, even when
several contact points have that slope; injectivity of $DW$
is therefore not required. In the last line we used
$1\leq(|W|+|\xi|)/L$ on $\mathcal C_L$.

Choose $L\geq1$ once and for all, depending only on
$n,\kappa,\Lambda,A$, such that
\begin{equation}\label{eq:calc-s4-111}
 \frac{(\Lambda/n)^n A}{L}<|B_{\kappa/4}(0)|.
\end{equation}
We now fix an arbitrary point $\xi\ne0$ and put
$e:=\xi/|\xi|$. The ball $B_{\kappa/4}(\kappa e/2)$ in the
space of slopes has the same volume as $B_{\kappa/4}(0)$,
independently of $e$. Each slope in this ball satisfies
$|\boldsymbol{\gamma}|<3\kappa/4<\kappa$. By
\eqref{eq:bad-contact-gradient-image}, the ball cannot be
contained in $DW(\mathcal C_L)$. For this point $\xi$, choose
\begin{equation}\label{eq:calc-s4-112}
 \boldsymbol{\gamma}\in B_{\kappa/4}(\kappa e/2)\setminus DW(\mathcal C_L).
\end{equation}
Since $|\boldsymbol{\gamma}|<\kappa$, the first part of the proof supplies a
global contact point $\xi_{\boldsymbol{\gamma}}$. If $\xi_{\boldsymbol{\gamma}}$ belonged to
$\mathcal C_L$, then $\boldsymbol{\gamma}=DW(\xi_{\boldsymbol{\gamma}})$ would belong to
$DW(\mathcal C_L)$, contrary to our choice. Hence
\begin{equation}\label{eq:calc-s4-113}
 |W(\xi_{\boldsymbol{\gamma}})|+|\xi_{\boldsymbol{\gamma}}|\leq L.
\end{equation}
The location of the chosen slope in the ball gives
\begin{align}
 \boldsymbol{\gamma}\cdot\xi
 &=\frac\kappa2|\xi|
       +\left(\boldsymbol{\gamma}-\frac\kappa2e\right)\cdot\xi \geq\frac\kappa2|\xi|
       -\left|\boldsymbol{\gamma}-\frac\kappa2e\right||\xi|
 \geq\frac\kappa4|\xi|.
 \label{eq:calc-s4-114}
\end{align}
Evaluating the supporting affine function at the prescribed
point $\xi$, we obtain
\begin{align}
 W(\xi)
 &\geq W(\xi_{\boldsymbol{\gamma}})+\boldsymbol{\gamma}\cdot(\xi-\xi_{\boldsymbol{\gamma}})\geq \boldsymbol{\gamma}\cdot\xi-|W(\xi_{\boldsymbol{\gamma}})|-|\boldsymbol{\gamma}|\,|\xi_{\boldsymbol{\gamma}}|\geq\frac\kappa4|\xi|-(1+\kappa)L.
 \label{eq:calc-s4-115}
\end{align}
\end{proof}

\begin{proof}[Proof of Proposition~\ref{prop:negative-P-exponent}]
We apply Lemma~\ref{lem:global-contact} to
\begin{equation}\label{eq:calc-s4-116}
 H:=D_\xi^2\Phi,\qquad \mathcal G:=P-\log m_0.
\end{equation}
The subtraction of $\log m_0$ absorbs the background Ricci term
in the equation for $P$.
On $T^n$-invariant functions, the Laplacian in these fixed coordinates is
$H^{ij}\partial_{\xi_i}\partial_{\xi_j}$, where $(H^{ij}):=H^{-1}$.
Since $\Ric(\omega_0)=-dd^c\log m_0$ on the open orbit,
the elliptic equation in \eqref{eq:pcf} becomes
\begin{equation}\label{eq:nonlocal-global-equation}
 \tr(H^{-1}D_\xi^2\mathcal G)=\underline R.
\end{equation}
Equation~\eqref{eq:toric-average-scalar-curvature} gives
$\underline R>0$.
For each individual smooth metric, $P$ is bounded on the compact
manifold. The upper bound for $m_0$ in
\eqref{eq:background-exponential-tail} therefore gives
\begin{equation}\label{eq:calc-s4-117}
 \mathcal G(\xi)\geq\kappa_0|\xi|-C_\varphi.
\end{equation}
This verifies the growth assumption of the lemma;
$C_\varphi$ need not be uniform.

The two-sided bounds in \eqref{eq:background-exponential-tail}
give $|\log m_0(\xi)|\leq C(1+|\xi|)$, and hence
$|\mathcal G(\xi)|\leq|P(\xi)|+C(1+|\xi|)$.
Using \eqref{eq:toric-volume}--\eqref{eq:toric-integration},
Lemma~\ref{lem:toric-L1-dirichlet} and
Proposition~\ref{prop:logarithmic-tail}, we obtain
\begin{equation}\label{eq:calc-s4-118}
 \int_{\R^n}(|\mathcal G|+|\xi|)m\,d\xi
 \leq C\left(1+\int_M|P|\,d\mu_\varphi+
                  \int_M|\xi|\,d\mu_\varphi\right)\leq C_B.
\end{equation}
All three hypotheses of Lemma~\ref{lem:global-contact} are now
verified, with $\kappa=\kappa_0$, $\Lambda=\underline R$, and
$A=C_B$. Its conclusion gives the uniform lower bound
\begin{equation}\label{eq:calc-s4-119}
 P(\xi)-\log m_0(\xi)=\mathcal G(\xi)
 \geq\frac{\kappa_0}{4}|\xi|-C_B.
\end{equation}
To return to $P$, use the lower bound for $m_0$ in
\eqref{eq:background-exponential-tail}. For $r>1$, its negative
power satisfies
\begin{equation}\label{eq:calc-s4-120}
 m_0^{\,1-r}(\xi)\leq C e^{(r-1)b_0|\xi|}.
\end{equation}
Choose $r>1$ sufficiently close to one that
$r\kappa_0/4-(r-1)b_0>0$. This choice depends only on the
background, independently of $B$. Using
\eqref{eq:toric-integration} for the background metric, we obtain
\begin{align}
\int_Me^{-rP}\,d\mu_0
 &=\frac1{|\Delta|}\int_{\R^n}
              e^{-rP}m_0\,d\xi =\frac1{|\Delta|}\int_{\R^n}
              e^{-r\mathcal G}m_0^{\,1-r}\,d\xi\notag\\
 &\leq C_B\int_{\R^n}
        \exp\bigl(-[r\kappa_0/4-(r-1)b_0]|\xi|\bigr)\,d\xi
 \leq C_B.
 \label{eq:calc-s4-121}
\end{align}
This proves the proposition.
\end{proof}

\section{Convergence under properness}\label{sec:proper-convergence}

In this section, we prove convergence of the flow under the
properness assumption. Corollary~\ref{cor:properness-reduction}
allows us to work with properness modulo $T_{\C}$ on the
space of $T^n$-invariant potentials. We first choose a continuous family
of automorphisms. The speed of this family is controlled by the
Dirichlet dissipation of the Mabuchi functional. We then obtain a
uniform $L^p$ estimate for the volume ratio, for some $p>1$.

Write $\omega(t):=\omega_{\varphi(t)}$ for the original flow,
which exists for all time by Theorem~\ref{thm:main}.
We use the normalization $E(\varphi(t))=0$ and the speed $f$ in
\eqref{eq:normalized-pcf-speed}. By Definition~\ref{def:pcf},
$f$ is determined intrinsically by the current metric and
commutes with pullback by a fixed holomorphic automorphism.

The dissipation identity \eqref{eq:Mabuchi-dissipation} and
the lower bound for the Mabuchi functional give
\begin{equation}\label{eq:mabuchi-dissipation-convergence}
 \int_0^\infty\|\nabla_\varphi f\|_{L^2(d\mu_\varphi)}^2\,dt<\infty.
\end{equation}
By Proposition~\ref{prop:toric-mabuchi-invariance},
$\mathcal M$ is invariant under $T_{\C}$, while
\eqref{eq:E-polytope-mean} fixes the mean of the symplectic potential $u(t)$.
Properness bounds the infimum of $J$ over each complex torus orbit,
but does not specify a family of automorphisms differentiable in time.
We choose such a family by fixing the weighted average of $D_xu$.
The interior support of the weight will allow us to control its
speed by the integral in \eqref{eq:mabuchi-dissipation-convergence}.

\begin{definition}
\label{def:continuous-torus-normalization}
Let $u(t)$ be a smooth family of symplectic potentials with fixed
mean, and put
\begin{equation}\label{eq:calc-s5-001}
 \bar x:=\frac1{|\Delta|}\int_\Delta x\,dx.
\end{equation}
Choose $\zeta\in C_c^\infty(\Delta^\circ)$, with
$\zeta\geq0$ and $\int_\Delta\zeta\,dx=1$, and define
\begin{equation}\label{eq:continuous-toric-gauge}
 b(t):=\int_\Delta\zeta D_xu(t)\,dx,\qquad
 \widetilde u(t,x):=u(t,x)-b(t)\cdot(x-\bar x).
\end{equation}
We call $\widetilde u(t)$ the continuous torus normalization
determined by $\zeta$. It satisfies
\begin{equation}\label{eq:calc-s5-002}
 \int_\Delta\zeta D_x\widetilde u(t)\,dx=0,\qquad
 \int_\Delta\widetilde u(t)\,dx=\int_\Delta u(t)\,dx.
\end{equation}
The first condition fixes the linear part, and the second follows
from $\int_\Delta(x-\bar x)\,dx=0$.
\end{definition}
For a smooth family $u(t)$, the parameter $b(t)$ is smooth in time,
since $\zeta$ has compact support in $\Delta^\circ$.
Replacing $u$ by $u-a\cdot(x-\bar x)$ replaces $b$ by $b-a$
and leaves $\widetilde u$ unchanged. Thus the normalization selects
a representative on each complex torus orbit.
Let $g_b$ denote the automorphism which is translation by $b$
in logarithmic coordinates, namely $z_i\mapsto e^{b_i/2}z_i$.
Legendre duality gives
\begin{align}
 \widetilde\Phi(t,\xi)
 &=\sup_{x\in\Delta}\{\xi\cdot x-u(t,x)+b(t)\cdot(x-\bar x)\}\notag\\
 &=\Phi(t,\xi+b(t))-b(t)\cdot\bar x.
 \label{eq:calc-s5-003}
\end{align}
Hence the metric corresponding to $\widetilde u$ is
$g_{b(t)}^*\omega_{\varphi(t)}$.
The unchanged mean and \eqref{eq:E-polytope-mean} imply
$E(\widetilde\varphi(t))=0$.

\subsection{Uniform \texorpdfstring{$L^p$}{Lp} bound for the volume form}
\begin{lemma}\label{lem:continuous-gauge-bounds}
There is a constant $C$, depending on the background, the initial
metric, the properness bound, and the fixed function $\zeta$, such that
\begin{equation}\label{eq:gauge-low-order-bounds}
 \begin{aligned}
 &\|\widetilde u(t)\|_{L^1(\Delta)}+J(\widetilde\varphi(t))
 +\int_M\log\frac{\omega_{\widetilde\varphi(t)}^n}{\omega_0^n}
                  \,d\mu_{\widetilde\varphi(t)}\leq C,
 \qquad t\geq0.
 \end{aligned}
\end{equation}
Moreover,
\begin{equation}\label{eq:gauge-speed}
 |b'(t)|^2\leq C\|\nabla_\varphi f\|_{L^2(d\mu_\varphi)}^2,\qquad
 \int_0^\infty|b'(t)|^2\,dt<\infty.
\end{equation}
\end{lemma}

\begin{proof}
By \eqref{eq:Mabuchi-dissipation},
$\mathcal M(\varphi(t))\leq\mathcal M(\varphi(0))$.
Corollary~\ref{cor:properness-reduction} therefore bounds
$J_{T_{\C}}(\varphi(t))$ uniformly. At each time choose
$a_t\in\mathbb R^n$ so that
$J(g_{a_t}^*\omega(t))\leq J_{T_{\C}}(\varphi(t))+1$.
The compact torus acts trivially on the metric, so
only its real translation parameter is needed. Put
\begin{equation}\label{eq:calc-s5-004}
 u^{a_t}:=u-a_t\cdot(x-\bar x).
\end{equation}
Equation \eqref{eq:toric-J-L1} gives
$\|u^{a_t}\|_{L^1(\Delta)}\leq C$.
Since $\int_\Delta\zeta\,dx=1$ and $\zeta$ has interior support,
integration by parts gives
\begin{align}
 b(t)-a_t
 &=\int_\Delta\zeta D_xu^{a_t}\,dx
 =-\int_\Delta D_x\zeta\,u^{a_t}\,dx,
 \label{eq:calc-s5-005}\\
 |b(t)-a_t|
 &\leq\|D_x\zeta\|_{L^\infty(\Delta)}
                       \|u^{a_t}\|_{L^1(\Delta)}\leq C.
 \label{eq:calc-s5-006}
\end{align}
Consequently,
\begin{align}
 \|\widetilde u(t)\|_{L^1(\Delta)}
 &\leq\|u^{a_t}\|_{L^1(\Delta)}
       +|b(t)-a_t|\int_\Delta|x-\bar x|\,dx
 \leq C,
 \label{eq:calc-s5-007}\\
 J(\widetilde\varphi(t))
 &\leq C\|\widetilde u(t)\|_{L^1(\Delta)}+C\leq C,
 \label{eq:calc-s5-008}
\end{align}
where the last line uses \eqref{eq:toric-J-L1} again.

The potential $\widetilde\varphi(t)$ represents
$g_{b(t)}^*\omega_{\varphi(t)}$ relative to the fixed form $\omega_0$.
We use this same reference form in all the functionals below.
Fix $t$. Integrating the first variation along
$\sigma\mapsto\sigma\widetilde\varphi(t)$, $0\leq\sigma\leq1$,
and integrating by parts gives the usual decomposition
\cite[(2.18)]{ChenCheng2021}:
\begin{align}
 \mathcal M(\widetilde\varphi(t))
 &=\frac1V\int_0^1\int_M\widetilde\varphi(t)
       \bigl(\underline R-R(\omega_{\sigma\widetilde\varphi(t)})\bigr)
       \,\omega_{\sigma\widetilde\varphi(t)}^n\,d\sigma\notag\\
 &=\int_M\log\frac{\omega_{\widetilde\varphi(t)}^n}{\omega_0^n}
       \,d\mu_{\widetilde\varphi(t)}
       +\underline R E(\widetilde\varphi(t))\notag\\
 &\quad-\frac1V\sum_{j=0}^{n-1}\int_M
       \widetilde\varphi(t)\,\Ric(\omega_0)
       \wedge\omega_{\widetilde\varphi(t)}^j
                                  \wedge\omega_0^{n-1-j}.
 \label{eq:calc-s5-009}
\end{align}
Since $E(\widetilde\varphi(t))=0$, we have
$\sup_M\widetilde\varphi(t)\geq0$.
Each mixed volume
$\int_M\omega_{\widetilde\varphi(t)}^j\wedge\omega_0^{n-j}$
equals $V$, so the definition of $E$ gives
\begin{align}
 &\sum_{j=0}^{n}\int_M
       (\sup_M\widetilde\varphi(t)-\widetilde\varphi(t))\,
       \omega_{\widetilde\varphi(t)}^j\wedge\omega_0^{n-j}\notag\\
 &\qquad=(n+1)V\sup_M\widetilde\varphi(t)
                    -(n+1)VE(\widetilde\varphi(t))\notag\\
 &\qquad=(n+1)V\sup_M\widetilde\varphi(t).
 \label{eq:calc-s5-010}
\end{align}
All summands are nonnegative. Using the fixed-background bound
$-C\omega_0\leq\Ric(\omega_0)\leq C\omega_0$ and
$|\widetilde\varphi(t)|\leq\sup_M\widetilde\varphi(t)
       +(\sup_M\widetilde\varphi(t)-\widetilde\varphi(t))$,
we obtain
\begin{align}
 &\left|\frac1V\sum_{j=0}^{n-1}\int_M
       \widetilde\varphi(t)\,\Ric(\omega_0)
       \wedge\omega_{\widetilde\varphi(t)}^j
                                  \wedge\omega_0^{n-1-j}\right|\notag\\
 &\qquad\leq\frac C V\sum_{j=0}^{n-1}\int_M
       |\widetilde\varphi(t)|\,
       \omega_{\widetilde\varphi(t)}^j\wedge\omega_0^{n-j}\notag\\
 &\qquad\leq Cn\sup_M\widetilde\varphi(t)
       +\frac C V\sum_{j=0}^{n-1}\int_M
       (\sup_M\widetilde\varphi(t)-\widetilde\varphi(t))
       \,\omega_{\widetilde\varphi(t)}^j\wedge\omega_0^{n-j}\notag\\
 &\qquad\leq C\sup_M\widetilde\varphi(t).
 \label{eq:calc-s5-011}
\end{align}
For the corresponding bound in terms of $d_1$, see
\cite[Lemma~4.4]{ChenCheng2021}.
Since $\widetilde\varphi(t)$ is $T^n$-invariant,
\eqref{eq:toric-sup-J} gives
$\sup_M\widetilde\varphi(t)\leq J(\widetilde\varphi(t))+C$. Hence
\begin{equation}\label{eq:calc-s5-012}
 \left|\mathcal M(\widetilde\varphi(t))
       -\int_M\log\frac{\omega_{\widetilde\varphi(t)}^n}{\omega_0^n}
          \,d\mu_{\widetilde\varphi(t)}\right|
 \leq C(1+J(\widetilde\varphi(t))).
\end{equation}
The first variation is unchanged by simultaneous pullback
of the metric and its variation. Integrating along a path from
$0$ to $\varphi(t)$ therefore gives
\begin{equation}\label{eq:gauge-mabuchi-cocycle}
 \mathcal M(\widetilde\varphi(t))
 =\mathcal M(\varphi(t))+\mathcal M(g_{b(t)}\cdot0)
 =\mathcal M(\varphi(t)).
\end{equation}
Here the additional term is zero by
Proposition~\ref{prop:toric-mabuchi-invariance} and $\mathcal M(0)=0$:
properness gives the required lower bound on $\mathcal H_0^T$.
Combining this invariance with the preceding estimate, as in
\cite[the proof of Lemma~4.18]{ChenCheng2021}, yields
\begin{align}
 &\int_M\log\frac{\omega_{\widetilde\varphi(t)}^n}{\omega_0^n}
                         \,d\mu_{\widetilde\varphi(t)}
 \leq\mathcal M(\widetilde\varphi(t))
                    +C(1+J(\widetilde\varphi(t)))\notag\\
 &\qquad =\mathcal M(\varphi(t))+C(1+J(\widetilde\varphi(t)))\leq\mathcal M(\varphi(0))+C\leq C.
 \label{eq:calc-s5-013}
\end{align}
This proves \eqref{eq:gauge-low-order-bounds}.

For the speed estimate, $\dot u$ denotes the derivative of the
original symplectic potential at fixed $x$.
By Legendre duality and \eqref{eq:polytope-dissipation},
\begin{equation}\label{eq:calc-s5-014}
 \dot u(t,x)=-f(t,D_xu(t,x)),\qquad
 \int_\Delta u^{jk}(\dot u)_j(\dot u)_k\,dx
 =|\Delta|\|\nabla_\varphi f\|_{L^2(d\mu_\varphi)}^2.
\end{equation}
Here $f$ on the right of the first identity is written in logarithmic
coordinates, and $(u^{jk})$ is the inverse of $D_x^2u$.
For each fixed $i$, the matrix Cauchy--Schwarz inequality gives
$| (\dot u)_i|^2\leq u_{ii}u^{jk}(\dot u)_j(\dot u)_k$,
with no summation over $i$. Thus
\begin{align}
 |b_i'(t)|
 &=\left|\int_\Delta\zeta(\dot u)_i\,dx\right|
 \leq\int_\Delta\zeta\sqrt{u_{ii}}\,
                  \sqrt{u^{jk}(\dot u)_j(\dot u)_k}\,dx\notag\\
 &\leq\left(\int_\Delta\zeta^2u_{ii}\,dx\right)^{1/2}
       \left(\int_\Delta u^{jk}(\dot u)_j(\dot u)_k\,dx\right)^{1/2}.
 \label{eq:calc-s5-015}
\end{align}
Squaring and summing over $i$ therefore gives
\begin{equation}\label{eq:calc-s5-016}
 |b'(t)|^2
 \leq |\Delta|\|\nabla_\varphi f\|_{L^2(d\mu_\varphi)}^2
               \int_\Delta\zeta^2\tr D_x^2u\,dx.
\end{equation}
Since $D_x^2u=D_x^2\widetilde u$ and $\zeta$ has interior support,
two integrations by parts give
\begin{align}
& \int_\Delta\zeta^2\tr D_x^2u\,dx
 =\int_\Delta\zeta^2\Delta_{\mathbb R^n}\widetilde u\,dx\notag\\
 &=-\int_\Delta D_x(\zeta^2)\cdot D_x\widetilde u\,dx
 =\int_\Delta\widetilde u\,
                    \Delta_{\mathbb R^n}(\zeta^2)\,dx\notag\\
 &\leq\|\Delta_{\mathbb R^n}(\zeta^2)\|_{L^\infty(\Delta)}
                         \|\widetilde u\|_{L^1(\Delta)}\leq C.
 \label{eq:calc-s5-017}
\end{align}
The resulting bound for $|b'(t)|^2$ and
\eqref{eq:Mabuchi-dissipation} imply
\begin{align}
 \int_0^T|b'(t)|^2\,dt
 &\leq C\int_0^T\|\nabla_\varphi f\|_{L^2(d\mu_\varphi)}^2\,dt\notag\\
 &=C\bigl(\mathcal M(\varphi(0))-\mathcal M(\varphi(T))\bigr)
 \leq C.
 \label{eq:calc-s5-018}
\end{align}
The last constant is independent of $T$, by the lower bound for
$\mathcal M$. Letting $T\to\infty$ proves \eqref{eq:gauge-speed}.
\end{proof}

In particular, for every fixed $L>0$,
\begin{equation}\label{eq:gauge-window-displacement}
 \sup_{s\leq t\leq s+L}|b(t)-b(s)|
 \leq C\left(L\int_s^{s+L}\|\nabla_\varphi f\|_{L^2(d\mu_\varphi)}^2\,dt\right)^{1/2}.
\end{equation}
Until the end of the proof of
Proposition~\ref{prop:proper-density-bound}, all quantities refer to
the pulled-back metric $g_{b(t)}^*\omega_{\varphi(t)}$, and we omit
the tildes.
Here $F$ is the logarithmic volume ratio of the pulled-back metric
relative to $\omega_0$, and $P$ is its normalized elliptic solution
in \eqref{eq:pcf}, also relative to this fixed background.
The intrinsic speed $f$ in \eqref{eq:normalized-pcf-speed} is
the pullback of the original normalized speed. The time derivative
of the potential in this moving gauge also contains a transport
term.

Hence the equations in the moving gauge are
\begin{equation}\label{eq:moving-gauge-equations}
 \partial_t\varphi=f+b'\cdot(D_\xi\Phi-\bar x),\qquad
 \partial_tF=\Delta_\varphi f+b'\cdot D_\xi\log m.
\end{equation}
The density $m$ is given by \eqref{eq:toric-volume}.
The added term in the first equation has zero mean with respect
to $d\mu_\varphi$, since the moment map pushes this measure
forward to $dx/|\Delta|$.
When $b'\equiv0$, in particular when $b\equiv0$, these equations
reduce to $\partial_t\varphi=f$ and $\partial_tF=\Delta_\varphi f$,
the normalized pseudo-Calabi flow identities valid on any compact
K\"ahler manifold.

\begin{proposition}\label{prop:proper-density-bound}
There exist $p>1$ and $C<\infty$ such that, in the gauge
\eqref{eq:continuous-toric-gauge},
\begin{equation}\label{eq:proper-density-bound}
 \sup_{t\geq0}\int_Me^{pF(t)}\,d\mu_0\leq C.
\end{equation}
\end{proposition}

\begin{proof}
By Proposition~\ref{prop:negative-P-exponent} and
Lemma~\ref{lem:continuous-gauge-bounds}, there are fixed
$r>1$ and $C$ such that
\begin{equation}\label{eq:P-negative-moment-along-flow}
 \int_Me^{-rP(t)}\,d\mu_0\leq C,\qquad t\geq0.
\end{equation}
Choose $\lambda>0$ so that
$\Ric(\omega_0)+\lambda\omega_0\geq0$.
By Tian's $\alpha$-invariant estimate
\cite[Proposition~2.1]{Tian1987}, choose a sufficiently small
$\alpha>0$, depending only on the fixed background, such that
\begin{equation}\label{eq:calc-s5-019}
 \int_Me^{-\alpha(\psi-\sup_M\psi)}\,d\mu_0\leq C_\alpha,
 \qquad \psi\in\mathcal H.
\end{equation}
Since $E(\varphi)=0$ implies $\sup_M\varphi\geq0$, this also
bounds $\int_Me^{-\alpha\varphi}\,d\mu_0$.
Fix $p>1$ sufficiently close to one that
\begin{equation}\label{eq:calc-s5-020}
 p<r,\qquad \frac{(p-1)\lambda r}{r-p}\leq\alpha.
\end{equation}
H\"older's inequality and
\eqref{eq:P-negative-moment-along-flow} then give
\begin{equation}\label{eq:weighted-negative-moment-convergence}
 \begin{aligned}
 \int_Me^{-pP-(p-1)\lambda\varphi}\,d\mu_0
 &\leq
 \left(\int_Me^{-rP}\,d\mu_0\right)^{p/r}\cdot
 \left(\int_Me^{-(p-1)\lambda r\varphi/(r-p)}\,d\mu_0
                                      \right)^{(r-p)/r}
 \leq C.
 \end{aligned}
\end{equation}
Consider the weighted integral
\begin{equation}\label{eq:weighted-moments-convergence}
 I(t):=\int_Me^{pF-(p-1)\lambda\varphi}\,d\mu_0
     =\int_Me^{(p-1)(F-\lambda\varphi)}\,d\mu_\varphi.
\end{equation}
Differentiating with respect to $t$ in the fixed measure $d\mu_0$
and using \eqref{eq:moving-gauge-equations}, we obtain the full expression
\begin{align} 
 I'(t)
 &=\int_Me^{pF-(p-1)\lambda\varphi}
       \bigl[p\partial_tF-(p-1)\lambda\partial_t\varphi\bigr]\,d\mu_0
 \notag\\
 &=p\int_Me^{(p-1)(F-\lambda\varphi)}
                         \Delta_\varphi f\,d\mu_\varphi
       -(p-1)\lambda\int_Me^{(p-1)(F-\lambda\varphi)}
                                      f\,d\mu_\varphi
 \notag\\
 &\quad+\frac1{|\Delta|}\int_{\mathbb R^n}
       e^{pF-(p-1)\lambda\varphi}m_0\,
       b'(t)\cdot\bigl[pD_\xi\log m
              -(p-1)\lambda(D_\xi\Phi-\bar x)\bigr]\,d\xi.
 \label{eq:weighted-density-evolution}
\end{align}
All spatial quantities in these identities are evaluated at time $t$.
We estimate the three terms in the last expression in order.

For the Laplacian term, use integration by parts on $M$ gives
\begin{align}
 &p\int_Me^{(p-1)(F-\lambda\varphi)}
                         \Delta_\varphi f\,d\mu_\varphi=p\int_Me^{(p-1)(F-\lambda\varphi)}
       \bigl[\Delta_\varphi F+\underline R
                  -\tr_{\omega_\varphi}\Ric(\omega_0)\bigr]\,d\mu_\varphi\notag\\
 &=p\int_Me^{(p-1)(F-\lambda\varphi)}
                         \Delta_\varphi(F-\lambda\varphi)\,d\mu_\varphi+p\lambda\int_Me^{(p-1)(F-\lambda\varphi)}
                   \bigl(n-\tr_{\omega_\varphi}\omega_0\bigr)\,d\mu_\varphi\notag\\
 &\qquad\qquad +p\underline R I(t)
       -p\int_Me^{(p-1)(F-\lambda\varphi)}
                  \tr_{\omega_\varphi}\Ric(\omega_0)\,d\mu_\varphi\notag\\
 &=-p(p-1)\int_Me^{(p-1)(F-\lambda\varphi)}
       |\nabla_\varphi(F-\lambda\varphi)|_\varphi^2\,d\mu_\varphi\notag\\
 &\qquad\qquad-p\int_Me^{(p-1)(F-\lambda\varphi)}
       \tr_{\omega_\varphi}\bigl(\Ric(\omega_0)+\lambda\omega_0\bigr)
                                             \,d\mu_\varphi\ +p(\underline R+n\lambda)I(t)\notag\\
 &\leq p(\underline R+n\lambda)I(t).
 \label{eq:calc-s5-021}
\end{align}

For the second term in \eqref{eq:weighted-density-evolution},
the coefficient is negative,
so we need a lower bound for the weighted integral of $f$.
First apply Jensen's inequality to the convex function
$s\mapsto e^{-ps}$, evaluated at $F+P$, with respect to the
probability measure
$I(t)^{-1}e^{pF-(p-1)\lambda\varphi}\,d\mu_0$. This gives
\begin{align}
 &\exp\!\left[-\frac p{I(t)}
      \int_Me^{pF-(p-1)\lambda\varphi}(F+P)\,d\mu_0\right]\notag\\
 &\quad\leq\frac1{I(t)}\int_M
       e^{pF-(p-1)\lambda\varphi}e^{-p(F+P)}\,d\mu_0\notag\\
 &\quad=\frac1{I(t)}\int_M
       e^{-pP-(p-1)\lambda\varphi}\,d\mu_0
 \leq\frac C{I(t)},
 \label{eq:calc-s5-023}
\end{align}
where the last inequality is
\eqref{eq:weighted-negative-moment-convergence}.
Hence we obtain
\begin{equation}\label{eq:relative-entropy-density}
 \begin{aligned}
 \int_Me^{(p-1)(F-\lambda\varphi)}(F+P)\,d\mu_\varphi
 &\geq\frac{I(t)}p
       \log\frac{I(t)}{\int_Me^{-pP-(p-1)\lambda\varphi}\,d\mu_0}
 &\geq\frac{I(t)}p\log\frac{I(t)}C.
 \end{aligned}
\end{equation}
Consequently, \eqref{eq:normalized-pcf-speed} yields
\begin{align}
 &-(p-1)\lambda\int_Me^{(p-1)(F-\lambda\varphi)}f\,d\mu_\varphi\notag\\
 &=-(p-1)\lambda\int_Me^{(p-1)(F-\lambda\varphi)}(F+P)\,d\mu_\varphi
       +(p-1)\lambda I(t)\int_M F\,d\mu_\varphi\notag\\
 &\leq-\frac{(p-1)\lambda}{p}I(t)\log\frac{I(t)}C
       +(p-1)\lambda I(t)\int_M F\,d\mu_\varphi.
 \label{eq:calc-s5-024}
\end{align}
It remains to estimate the last term in
\eqref{eq:weighted-density-evolution}, which comes from the moving
automorphism. Use
$F=\log(m/m_0)$ and $\varphi=\Phi-\Phi_0$. Writing its integrand
as a total derivative plus a fixed-background term gives 
\begin{align}
&\frac1{|\Delta|}\int_{\mathbb R^n}
 e^{pF-(p-1)\lambda\varphi}m_0\,
 b'(t)\cdot\bigl[pD_\xi\log m
              -(p-1)\lambda(D_\xi\Phi-\bar x)\bigr]\,d\xi \notag\\
 &=\frac1{|\Delta|}\int_{\mathbb R^n}
       b'(t)\cdot D_\xi\!\left(e^{pF-(p-1)\lambda\varphi}m_0\right)\,d\xi \notag\\
 &\quad+\frac{p-1}{|\Delta|}\int_{\mathbb R^n}
 e^{pF-(p-1)\lambda\varphi}m_0\,
 b'(t)\cdot\bigl[D_\xi\log m_0-\lambda(D_\xi\Phi_0-\bar x)\bigr]\,d\xi.
 \label{eq:weighted-density-drift-decomposition}
\end{align}

For the first term, fix $t$ and choose
$\chi_R\in C_c^\infty(\mathbb R^n)$ with $0\leq\chi_R\leq1$,
equal to one on $\{|\xi|\leq R\}$, supported in
$\{|\xi|<2R\}$, and satisfying $|D_\xi\chi_R|\leq C/R$.
Integration by parts gives
\begin{align}
 &\left|\int_{\mathbb R^n}\chi_R\,b'(t)\cdot D_\xi
       \!\left(e^{pF-(p-1)\lambda\varphi}m_0\right)\,d\xi\right|\notag\\
 &=\left|\int_{\mathbb R^n}e^{pF-(p-1)\lambda\varphi}m_0\,
       b'(t)\cdot D_\xi\chi_R\,d\xi\right|\notag\\
 &\leq\frac{C_t|b'(t)|}{R}\int_{|\xi|\geq R}m_0\,d\xi
       \longrightarrow0\qquad\text{as }R\to\infty.
 \label{eq:calc-s5-022}
\end{align}
The constant $C_t$ is used only to justify this identity at a
fixed time.

For the second term, $D_\xi\log m_0$ is bounded by
Proposition~\ref{prop:logarithmic-tail}, and $D_\xi\Phi_0$ takes
values in the fixed bounded polytope. Hence
\begin{align}
 &\left|\frac{p-1}{|\Delta|}\int_{\mathbb R^n}
 e^{pF-(p-1)\lambda\varphi}m_0\,
 b'(t)\cdot\bigl[D_\xi\log m_0-\lambda(D_\xi\Phi_0-\bar x)\bigr]
 \,d\xi\right|\notag\\
 &\leq\frac{C|b'(t)|}{|\Delta|}\int_{\mathbb R^n}
                    e^{pF-(p-1)\lambda\varphi}m_0\,d\xi
 =C|b'(t)|I(t),
 \label{eq:weighted-density-background-drift}
\end{align}
where $C$ is independent of $t$. Combining these two conclusions
with \eqref{eq:weighted-density-drift-decomposition}, we obtain
the estimate for the third term in
\eqref{eq:weighted-density-evolution}:
\begin{align}
 &\frac1{|\Delta|}\int_{\mathbb R^n}
 e^{pF-(p-1)\lambda\varphi}m_0\,
 b'(t)\cdot\bigl[pD_\xi\log m
              -(p-1)\lambda(D_\xi\Phi-\bar x)\bigr]\,d\xi\leq C|b'(t)|I(t).
 \label{eq:weighted-density-drift}
\end{align}

Substituting these three bounds into
\eqref{eq:weighted-density-evolution}, we obtain
\begin{equation}\label{eq:weighted-density-damping}
 \begin{aligned}
 I'(t)\leq{}&
 \bigl[p(\underline R+n\lambda)
        +(p-1)\lambda\int_M F\,d\mu_\varphi\bigr]I(t)\\
 &-\frac{(p-1)\lambda}{p}I(t)\log\frac{I(t)}C
       +C|b'(t)|I(t).
 \end{aligned}
\end{equation}
The entropy is uniformly bounded by
Lemma~\ref{lem:continuous-gauge-bounds}. Therefore we have
\begin{equation}\label{eq:log-density-damped-ode}
 \frac{I'(t)}{I(t)}+\frac{(p-1)\lambda}{p}\log I(t)
       \leq C+C|b'(t)|.
\end{equation} 
Then we have
\begin{align}
 \frac d{dt}\left(e^{(p-1)\lambda t/p}\log I(t)\right)
 &=e^{(p-1)\lambda t/p}
   \left(\frac{I'(t)}{I(t)}+\frac{(p-1)\lambda}{p}\log I(t)\right)\notag\\
 &\leq Ce^{(p-1)\lambda t/p}
       +Ce^{(p-1)\lambda t/p}|b'(t)|.
 \label{eq:log-density-integrating-factor}
\end{align}
Integrating from $0$ to $t$, we obtain
\begin{align}
 e^{(p-1)\lambda t/p}\log I(t)-\log I(0)
 &=\int_0^t\frac d{ds}
      \left(e^{(p-1)\lambda s/p}\log I(s)\right)\,ds\notag\\
 &\leq C\int_0^t e^{(p-1)\lambda s/p}\,ds
       +C\int_0^t e^{(p-1)\lambda s/p}|b'(s)|\,ds\notag\\
 &=\frac{Cp}{(p-1)\lambda}
       \left(e^{(p-1)\lambda t/p}-1\right)
       +C\int_0^t e^{(p-1)\lambda s/p}|b'(s)|\,ds.
 \label{eq:log-density-integrated-factor}
\end{align}
Then we get
\begin{align}
 \log I(t)
 &\leq e^{-(p-1)\lambda t/p}\log I(0)
       +\frac{Cp}{(p-1)\lambda}
          \left(1-e^{-(p-1)\lambda t/p}\right)\notag\\
 &\quad+Ce^{-(p-1)\lambda t/p}
          \int_0^t e^{(p-1)\lambda s/p}|b'(s)|\,ds\notag\\
 &=e^{-(p-1)\lambda t/p}\log I(0)
       +\frac{Cp}{(p-1)\lambda}
          \left(1-e^{-(p-1)\lambda t/p}\right)\notag\\
 &\quad+C\int_0^t e^{-(p-1)\lambda(t-s)/p}|b'(s)|\,ds\notag\\
 &\leq e^{-(p-1)\lambda t/p}\log I(0)
       +\frac{Cp}{(p-1)\lambda}
       +C\int_0^t e^{-(p-1)\lambda(t-s)/p}|b'(s)|\,ds.
 \label{eq:calc-s5-025}
\end{align}
By \eqref{eq:gauge-speed},
\begin{align}
 \int_0^t e^{-(p-1)\lambda(t-s)/p}|b'(s)|\,ds
 &\leq\left(\int_0^t e^{-2(p-1)\lambda(t-s)/p}\,ds\right)^{1/2}
          \|b'\|_{L^2(0,\infty)}\notag\\
 &\leq\sqrt{\frac{p}{2(p-1)\lambda}}\,
          \|b'\|_{L^2(0,\infty)}.
 \label{eq:calc-s5-026}
\end{align}
This bound is independent of $t$, and hence
\begin{equation}\label{eq:calc-s5-027}
 \log I(t)\leq e^{-(p-1)\lambda t/p}\log I(0)
       +C\bigl(1+\|b'\|_{L^2(0,\infty)}\bigr)\leq C.
\end{equation}
Finally, \eqref{eq:toric-sup-J}, with $E(\varphi)=0$, gives
$\sup_M\varphi\leq J(\varphi)+C\leq C$. Therefore
\begin{equation}\label{eq:calc-s5-028}
 \int_Me^{pF}\,d\mu_0
 \leq e^{(p-1)\lambda\sup_M\varphi}I(t)\leq C.
\end{equation}
This proves the proposition.
\end{proof}

\subsection{Proof of Theorem~\ref{thm:proper-convergence}}
We now return to the original flow $\omega(t)$.
To apply the parabolic estimates, we use a fixed automorphism
on each interval. For $s\geq0$ and $t\in[s,s+4]$, set
\begin{equation}\label{eq:calc-s5-029}
 \omega_s(t):=g_{b(s)}^*\omega(t),
 \qquad \omega_s(t)=\omega_0+dd^c\varphi_s(t),\qquad
 E(\varphi_s(t))=0.
\end{equation}
The subscript $s$ labels the fixed left endpoint of the interval;
time derivatives are taken with respect to $t$.
Since $g_{b(s)}$ is fixed in time,
$\partial_t\varphi_s(t)=g_{b(s)}^*f(t)$.
Thus $\varphi_s$ solves \eqref{eq:normalized-pcf-speed}, with $F_s,P_s$
defined relative to $\omega_0$ as in \eqref{eq:pcf}.
The automorphism is fixed on this interval, so the transport
terms in \eqref{eq:moving-gauge-equations} are absent.

\begin{proposition}\label{prop:uniform-frozen-window}
For every integer $k\geq0$, there is a constant $C_k$ independent
of $s\geq0$ such that
\begin{equation}\label{eq:frozen-window-smoothing}
 \sup_{s+2\leq t\leq s+3}
       \|\varphi_s(t)\|_{C^k(M,\omega_0)}\leq C_k,\qquad
 C^{-1}\omega_0\leq\omega_s(t)\leq C\omega_0,
 \quad s+2\leq t\leq s+3.
\end{equation}
\end{proposition}

\begin{proof}
The relative automorphism carrying the moving gauge to
$\omega_s(t)$ is $g_{b(s)-b(t)}$. By
\eqref{eq:gauge-window-displacement}, its parameter stays in a
fixed closed ball in $\mathbb R^n$ for $s\leq t\leq s+4$,
independently of $s$. The background volume Jacobian is a smooth
positive function of this parameter and the point of $M$.
Compactness of the parameter ball and $M$ therefore bounds this
Jacobian and its reciprocal uniformly. A change of variables
in \eqref{eq:proper-density-bound} gives
\begin{equation}\label{eq:frozen-density-lp}
 \int_Me^{pF_s(t)}\,d\mu_0\leq C,\qquad
 s\leq t\leq s+4.
\end{equation}

Proposition~\ref{prop:positive-time-volume}, applied on $[s,s+4]$ with
$L=4$ and $\delta=1$, gives
\begin{equation}\label{eq:frozen-low-order}
 \|\varphi_s(t)\|_{C^\beta(M,\omega_0)}
       +\|P_s(t)\|_\infty+\|F_s(t)\|_\infty
 \leq C,\qquad s+1\leq t\leq s+4.
\end{equation}
Its constants depend only on the fixed background and the uniform
density bound \eqref{eq:frozen-density-lp}, so they are independent
of $s$ and of the pointwise initial density on each interval.

Apply Proposition~\ref{prop:common-modulus} to the entire family
\begin{equation}\label{eq:calc-s5-030}
 \{\omega_s(t):s\geq0,\ s+1\leq t\leq s+4\}.
\end{equation}
Its hypotheses are uniform over this family, so it gives one
spatial modulus for all the functions $P_s(t)$, independent of
both $s$ and $t$. To use the normalization in
\cite{ChengTian2026}, replace
$\varphi_s(t)$ by
$\varphi_s(t)+\int_{s+1}^t\left(\int_M F_s(a)\,d\mu_{\varphi_s(a)}\right)\,da$.
This changes only a uniformly bounded time-dependent constant
on the interval.

After translating $s+1$ to time zero, we apply Cheng--Tian
\cite[Theorem~4.3 and Corollary~4.7]{ChengTian2026}, taking the
comparison function to be $P_s$ itself, and the higher order
regularity argument of Chen--Zheng
\cite{ChenZheng2013}.
This gives uniform $C^k$ bounds on $[s+2,s+3]$, with constants
independent of $s$.
Together with the determinant lower bound in
\eqref{eq:frozen-low-order} and the bounded change of normalization,
this proves \eqref{eq:frozen-window-smoothing}.
\end{proof} 

\begin{proof}[Proof of Theorem~\ref{thm:proper-convergence}]
Choose $t_j\geq3$ with $t_j\to\infty$ and
$\|\nabla_{\varphi(t_j)}f(t_j)\|_{L^2(d\mu_{\varphi(t_j)})}^2\to0$, using
\eqref{eq:mabuchi-dissipation-convergence}. Put
$s_j:=t_j-3$ and $g_j:=g_{b(s_j)}$.
Proposition~\ref{prop:uniform-frozen-window} bounds the normalized
potentials $\varphi_{s_j}(t_j)$ in $C^k(M,\omega_0)$ for every $k$.
A diagonal subsequence converges smoothly, and hence
\begin{equation}\label{eq:calc-s5-032}
 g_j^*\omega(t_j)\longrightarrow\omega_*
 \quad\hbox{smoothly on }M.
\end{equation}
The uniform lower bound $g_j^*\omega(t_j)\geq C^{-1}\omega_0$
passes to the limit, so $\omega_*$ is positive. The convergent
global potentials show that $\omega_*\in[\omega_0]$.

Let $f_j:=g_j^*f(t_j)$, with $f$ as in
\eqref{eq:normalized-pcf-speed}. By Definition~\ref{def:pcf},
uniform elliptic estimates give smooth convergence to the
corresponding mean-zero function $f_*$ for $\omega_*$. Since
the Dirichlet integral is invariant under pullback,
\begin{equation}\label{eq:calc-s5-033}
 \int_M|\nabla_{\omega_*}f_*|_{\omega_*}^2\,\frac{\omega_*^n}{V}
 =\lim_{j\to\infty}\|\nabla_{\varphi(t_j)}f(t_j)\|_{L^2(d\mu_{\varphi(t_j)})}^2=0.
\end{equation}
Connectedness and the mean-zero normalization imply $f_*=0$.
Its elliptic equation gives
$R(\omega_*)=\underline R$. Since $g_j\in T_{\C}$ commutes with the compact
torus, each $g_j^*\omega(t_j)$ is $T^n$-invariant, and so is
$\omega_*$.

By the smooth convergence and Chen--Zheng's stability theorem
\cite{ChenZheng2013}, for a fixed
sufficiently large $j$, the flow $g_j^*\omega(t_j+t)$ converges
smoothly and exponentially to a cscK metric.
Pulling back by this fixed $g_j^{-1}$ gives the asserted convergence
of the original flow, and elliptic estimates together with the
normalization $E=0$ give the same convergence for the potentials.
The limit is $T^n$-invariant since the original flow is
$T^n$-invariant. This completes the proof.
\end{proof}

\bigskip
\noindent
Kai Zheng, University of Chinese Academy of Sciences, Beijing, 100190, China.\\*
Email: \href{mailto:KaiZheng@amss.ac.cn}{\texttt{KaiZheng@amss.ac.cn}}

\medskip
\noindent
Junhao Tian, Stony Brook University, New York, 11794, USA.\\*
Email: \href{mailto:junhao.tian@stonybrook.edu}{\texttt{junhao.tian@stonybrook.edu}}
\end{document}